\documentclass[10pt]{article}
\usepackage{epsfig}
\usepackage{amssymb,amsmath,amsthm,url}
\usepackage{multirow}
\usepackage{booktabs}
\usepackage{natbib}
\usepackage{hyperref}
\usepackage{setspace}
\usepackage{dsfont}
\usepackage{marginnote}

\usepackage{amsfonts,times,eucal}
\usepackage{graphicx,ifthen}
\usepackage{wrapfig}
\usepackage{latexsym}
\usepackage{color}
\usepackage{mathrsfs}
\usepackage{bm}
\usepackage{bbm}
\usepackage{srctex}
\usepackage{enumitem}
\makeatletter
\newcommand{\mylabel}[2]{#2\def\@currentlabel{#2}\label{#1}}
\makeatother
\usepackage{mathtools}

\theoremstyle{plain}
\newtheorem{theorem}{Theorem}[section]
\newtheorem{corollary}[theorem]{Corollary}
\newtheorem{lemma}[theorem]{Lemma}
\newtheorem{proposition}[theorem]{Proposition}

\newtheorem{condition}[theorem]{Condition}

\theoremstyle{definition}

\newtheorem{algorithm}[theorem]{Algorithm}

\theoremstyle{remark}

\numberwithin{equation}{section}

\newcommand{\N}{\mathbb{N}}
\newcommand{\R}{\mathbb{R}}

\newcommand{\Z}{\mathbb{Z}}

\newcommand{\p}{\mathbb{P}}
\newcommand{\pc}{\mathbb{P}^*}

\newcommand{\cA}{\mathcal{A}}
\newcommand{\cB}{\mathcal{B}}
\newcommand{\cC}{\mathcal{C}}
\newcommand{\cD}{\mathcal{D}}

\newcommand{\cG}{\mathcal{G}}
\newcommand{\cF}{\mathcal{F}}
\newcommand{\cH}{\mathcal{H}}
\newcommand{\cI}{\mathcal{I}}

\newcommand{\cL}{\mathcal{L}}
\newcommand{\cS}{\mathcal{S}}
\newcommand{\cT}{\mathcal{T}}

\newcommand{\cO}{\mathcal{O}}
\newcommand{\cU}{\mathcal{U}}
\newcommand{\cV}{\mathcal{V}}
\newcommand{\cW}{\mathcal{W}}

\newcommand{\fG}{\mathfrak{G}}

\newcommand{\fS}{\mathfrak{S}}
\newcommand{\fT}{\mathfrak{T}}
\newcommand{\fU}{\mathfrak{U}}
\newcommand{\E}[1]{\mathbbm{E}\left [  #1 \right ]}
\newcommand{\Es}[1]{\mathbbm{E} [  #1 ]}
\newcommand{\Ec}[1]{\mathbbm{E}^*\left [  #1 \right ]}
\newcommand{\Et}[1]{\widetilde{\mathbbm{E} }\left [  #1 \right ]}

\newcommand{\veps}{\epsilon}
\let\veps\epsilon
\renewcommand{\epsilon}{\varepsilon}

\newcommand{\pspace}{(\Omega,\cA,\p)}

\newcommand{\intd}[1]{\mathrm{d}#1}

\newcommand{\norm}[1]{\left\| #1 \right\| }
\newcommand{\norms}[1]{ \| #1 \| }
\newcommand{\scalar}[2]{\left\langle #1,#2 \right\rangle}

\newcommand{\1}[1]{\,\mathds{1}\! \left\{ #1 \right\} }

\newcommand{\md}[2]{\operatorname{d}_{p} (#1,#2 )}
\newcommand{\tiom}{\widetilde{\omega}}
\newcommand{\feps}{f_{\epsilon}}
\newcommand{\geps}{g_{\epsilon}}
\newcommand{\fx}{f_{X}}
\newcommand{\fxt}{f_{X,t}}

\newcommand{\gxt}{g_{X,t}}
\newcommand{\gxtt}{g_{\tilde{X},t}}
\newcommand{\oas}{o_{a.s.}}
\newcommand{\Oas}{\cO_{a.s.}}
\newcommand{\op}{o_{p}}
\newcommand{\Op}{\cO_{p}}
\newcommand{\hatPsiB}{\hat\Psi_{n,b,h}^*}
\newcommand{\hatPsi}{\hat\Psi_{n,h}}
\newcommand{\hatPsib}{\hat\Psi_{n,b}}

\allowdisplaybreaks

\begin{document}

\title{
The autoregression bootstrap for kernel estimates \\ of smooth nonlinear functional time series\footnote{This research was supported by the German Research Foundation (DFG), Grant Number KR-4977/1.}
}

\author{Johannes T. N. Krebs\footnote{Department of Statistics, University of California, Davis, CA, 95616, USA, email: \tt{jtkrebs@ucdavis.edu} }\; \footnote{Corresponding author}
\and
J{\"u}rgen E. Franke\footnote{Department of Mathematics, University of Kaiserslautern, 67653 Kaiserslautern, Germany, email: \tt{franke@mathematik.uni-kl.de} }
}

\date{\today}
\maketitle

\begin{abstract}
\setlength{\baselineskip}{1.8em}
Functional times series have become an integral part of both functional data and time series analysis. This paper deals with the functional autoregressive model of order 1 and the autoregression bootstrap for smooth functions. The regression operator is estimated in the framework developed by \cite{ferraty2004nonparametric} and \cite{ferraty_nonparametric_2007} which is here extended to the double functional case under an assumption of stationary ergodic data which dates back to \cite{laib2010nonparametric}. The main result of this article is the characterization of the asymptotic consistency of the bootstrapped regression operator.\medskip\\
\noindent {\bf Keywords:}  Autoregression bootstrap; Functional data analysis; Functional time series; Functional kernel regression; Hilbert spaces; Nonparametric statistics; Resampling

\noindent {\bf MSC 2010:} Primary: 62M10, 62F40, 62M20; Secondary: 62G09, 60G25
\end{abstract}

\section{Introduction}
The pioneering work of \cite{bosq_linear_2000} and \cite{ramsay_applied_2002,ramsay_functional_2005} encouraged the breakthrough of functional data analysis (FDA). Nowadays FDA is an established field in both theoretical and practical statistical research. Functional data objects which are sequentially observed over time and which feature a dependence pattern are known as functional time series. Typical examples are high-frequent financial data and data concerning the electricity consumption. Here based on a long and almost continuous record, data are segmented into curves over consecutive time intervals.

In this paper we focus on the autoregression bootstrap in the first order functional autoregressive (FAR(1)) model
\begin{align}\label{Eq:FAR1Process}
	X_t = \Psi(X_{t-1}) + \epsilon_t, \quad t\in\Z,
	\end{align}
where both the observations $X_t$ and the innovations $\epsilon_t$ are Hilbert space valued elements (in particular functions). The regression operator $\Psi$ is not necessarily linear and is estimated with kernel methods which date back at least to \cite{ferraty_nonparametric_2007}. So for our autoregression bootstrap we merge the model of Ferraty et al.\ (see \cite{ferraty_nonparametric_2006}) with the ideas from \cite{franke1992bootstrap} and \cite{kreutzberger1993}, which have also been considered in \cite{franke_bootstrap_2002} for the autoregression bootstrap of one-dimensional AR(1) processes. The detailed assumptions of the model are stated in Sections~\ref{Section_NotationsDefinitions} and \ref{Section_AsymptoticNormality}. 

The research of functional autoregressive processes dates back at least to \cite{bosq_linear_2000} who studies the FAR-model for linear operators. Since then there has been extensive study of linear functional time series. We refer the readers to \cite{antoniadis2003wavelet}, \cite{antoniadis2006functional}, \cite{bosq2007general}, \cite{hormann2010weakly}, \cite{horvath2012inference}. Since the approach of \cite{bosq_linear_2000} has certain drawbacks when it comes to implementation, there have been proposed alternative methods of estimating the linear functional model.  \cite{aue_prediction_2015} present a method which is based on functional principal component analysis (FPCA). \cite{hormann2015dynamic} propose another way to address the problem of dimension reduction.

In this paper, we continue with the framework of Ferraty et al.\ which is based on kernel methods in the functional regression model. When dealing with nonparametric models kernel methods have proved to be a powerful tool. So far, the functional regression model $y = \psi(X) + \epsilon$ with a real-valued response variable $y$ has been extensively studied in the kernel regression setting. \cite{ferraty_nonparametric_2007} study the asymptotic properties of the kernel estimator. \cite{masry2005nonparametric} and \cite{delsol_advances_2009} extend this functional regression model with a real-valued response to time series data which is strongly mixing. \cite{laib2010nonparametric,laib2011rates} study the functional kernel regression model for stationary ergodic data, which allows in particular the applicability to non-mixing processes.

So far, the double functional setting has only been investigated in some pioneering work. \cite{ferraty_regression_2012} study the asymptotic properties of the kernel estimator in a double functional regression model for pairs of i.i.d.\ data. \cite{krebs2018doublefunctional} extends their results to pairs of dependent data in the framework of stationary ergodic data of Laib and Louani and also studies the approximating properties of the wild and naive bootstrap in the double functional setting.

The residual bootstrap has received much attention in a model with a real-valued response variable. \cite{gonzalez2011bootstrap} study it in the functional linear regression model. \cite{ferraty_validity_2010} consider both the naive and the wild bootstrap for pairs of i.i.d.\ data in the nonlinear kernel regression setting. \cite{rana_bootstrap_2016} study the same bootstrap variants in a model with an $\alpha$-mixing condition. \cite{politis_kernel_2016} extend these results to the regression bootstrap in the FAR(1)-model. \cite{franke_residual_2016} consider the residual bootstrap for the linear FAR(1)-process based on the framework of \cite{bosq_linear_2000}. \cite{paparoditis2016sieve} proposes a sieve bootstrap for functional time series.

Naturally the FAR(1)-process is a dependent time series, more precisely a Markov process with respect to its natural filtration $(\cF_t)$, where $\cF_t = \sigma(X_0,\ldots,X_t)$. As $\alpha$-mixing can fail even for Markov processes if certain smoothness conditions are not satisfied (see, e.g., \cite{andrews1984non}), we choose the concept of stationary ergodic data from \cite{laib2010nonparametric} to model the dependence within the time series. This concept is quite general. Denote the conditional small ball probabilities of the process by $F_x^{\cF_{t-1}}(h) = \p( \|X_t-x\|\le h | X_{t-1} )$ and $F_x(h) = \p(\|X_t-x\|\le h)$. Then our model mainly relies on the ergodicity of the empirical averages $n^{-1} \sum_{t=1}^n F_x^{\cF_{t-1}}(h) \approx F_x(h)$. Using this result and a multiplicative structure of the small ball probability, we obtain limiting laws for the bootstrapped operator.

The unique contribution of this paper is the study of the autoregression bootstrap in the functional kernel regression model \eqref{Eq:FAR1Process} and its asymptotic normality in the Hilbert space-sense. This means that we construct in the first step a new functional time series $X^*$ from the estimated residuals of the observed time series $X$ based on the kernel estimate of the regression operator $\hat{\Psi}_{n,b}$ (and an oversmoothing bandwidth $b$). In the second step, we consider the asymptotic distribution of $\hat{\Psi}_{n,h}$ at a point $x$ in the function space. For that reason, we use the bootstrapped time series $X^*$ to construct an approximation $\hatPsiB(x)$. We show that after rescaling $\sqrt{n F_x(h)} (\hatPsiB(x)-\hatPsib(x))$ and $\sqrt{n F_x(h)} (\hatPsi(x)-\Psi(x))$ tend to the same Gaussian distribution on the Hilbert space. When compared to the regression bootstrap, the situation here is much more complex. For the regression bootstrap $X^*_t$ coincides with $X_t$ by construction. So, the regressors $X^*_t$ in the bootstrap world trivially satisfy the same assumptions as the real regressors $X_t$ which is quite useful in showing that the same kind of asymptotic properties hold for the bootstrap, too. So in our case, we also have to verify that the constructed time series $X^*$ has the same properties concerning its distribution as the time series $X$.

The remainder of this manuscript is organized as follows. We introduce the model and notation in Section~\ref{Section_NotationsDefinitions}. The detailed assumptions and the main result are
contained in Section~\ref{Section_AsymptoticNormality}. The technical results are given in Section~\ref{Section_Proofs} and in the Appendix~\ref{AppendixA}.

\section{General notation and framework}\label{Section_NotationsDefinitions}
Let $\mathcal{D}$ be a convex and compact subset of $\R^d$ with nonempty interior and let $\nu$ be a finite measure on the Borel-$\sigma$-field $\cB(\cD)$ of $\cD$ which admits a strictly positive density w.r.t.\ the Lebesgue measure. Let $\cH$ be the separable Hilbert space $L^2(\mathcal{D},\cB(\mathcal{D}),\nu)$. $\cH$ is equipped with an orthonormal basis $\{e_k:k\in\N\}$, an inner product $\scalar{\cdot}{\cdot}$ and the corresponding norm $\norm{\,\cdot\,}$. The stationary FAR(1)-process from \eqref{Eq:FAR1Process} is defined on a probability space $\pspace$ and takes values in a subspace $\cS$ of $\cH$ of smooth, real-valued functions. The innovations $\epsilon_t$ are i.i.d, have mean 0 and are also $\cS$-valued.

The regression operator $\Psi\colon \cH\rightarrow\cH$ is not necessarily linear but Lipschitz continuous w.r.t.\ the norm on the Hilbert space, $\norm{\cdot}$. We estimate $\Psi$ with the methods given in the kernel regression framework of \cite{ferraty2004nonparametric} but in this manuscript, we work in a double functional setting similar as in \cite{ferraty_regression_2012}.

Define on the set of continuous functions on $\cD$, $C^0(\cD)$, the norm
\begin{align*}
		&\norm{x}_{1,C^0(\cD)} \coloneqq \norm{x}_{\infty} + \tiom_x, \\
		&\text{ where }  \norm{x}_{\infty} \coloneqq \sup_{u\in \mathcal{D} } |x(u)| \text{ and }   \tiom_x  \coloneqq \sup_{u,v\in \mathcal{D}, u\neq v } \frac{ \left| x(u) - x(v) \right|}{ \norm{u-v}_2 }.
\end{align*}
Moreover, we define sets of smooth functions by 
\[
	\cU(r) \coloneqq \big \{ x\in C^0(\mathcal{D}): \norm{x}_{1,C^0(\cD)} \le r \big\}, \quad r\ge 0.
\]
Then $\cup_{r\in\N} \cU(r)$ is an $\R$-vector space. Set $\fU = \{ \cU(r) : r\ge 0\}$. It is well-known (cf.\ \cite{van2013weak}) that each $\cU=\cU(r)\in\fU$ is totally bounded in $(\cH,\norm{\cdot})$ and there is a constant $C$ which only depends on the dimension $d$ such that the logarithm of the $\veps$-covering number w.r.t.\ $\norm{\cdot}$ on $\cH$ satisfies
\begin{align*}
	\log N(\veps,\cU,\norm{\cdot}) \le C \lambda( \mathcal{D}^{(1)} ) ( \sqrt{\nu(\mathcal{D})} r  \veps^{-1} )^{d},
	\end{align*}
where $\cD^{(1)} = \{u\in\R^d: \; \exists v\in\cD:\, \norm{u-v}_{\infty} < 1\}$ and $\lambda$ is the Lebesgue measure. Throughout this manuscript, we assume that $\cS \subseteq \bigcup_{r \in\N} \cU(r)$ is a linear subspace. In particular, each realization of $\epsilon_t$ and $X_t$ is Lipschitz continuous.

We come to the definition of the estimator of the regression operator. For that reason, let $(r_n:n\in\N)$ be a real-valued sequence which tends to infinity. We give a precise definition of this sequence below and in the appendix. Set $\Delta_{h,t}(x) = K(h^{-1}\norm{X_t-x} )$, where $K$ is a kernel function and $h>0$ is the bandwidth. The estimator $\hatPsi\colon \cS\to\cS$ is defined by
\begin{align*}
		\hat{\Psi}_{n,h}(x) = 
		\begin{cases}
			(\sum_{t=1}^n \Delta_{h,t}(x) )^{-1} \sum_{t=1}^n X_{t+1} \Delta_{h,t}(x) &\text{if } x \in \cS\cap \cU(r_n), \\
			 \overline{X}_n &\text{if }  x \in \cS\setminus \cU(r_n). \\
			\end{cases}
\end{align*}

The subindex $n$ in $\hat\Psi_{n,h}$ emphasizes the dependence of the estimator on the set $\cU(r_n)$. The subindex $h$ clarifies which bandwidth is used. In particular, if we use later an oversmoothing bandwidth $b >> h$, the estimator $\hat\Psi_{n,b}$ is computed on the same sets $\cU(r_n)$, only the bandwidth varies (and the latter can also depend on $n$).

The definition of $\hat\Psi_{n,h}$ on the complement of $\cU(r_n)$ in $\cS$ is technical because for the bootstrap procedure to work, we need the estimator to be well-behaved on regions of $\cS$ where the sample data is relatively sparse.

We describe our autoregressive bootstrap procedure. Let $X_0,\ldots,X_{n+1}$ be a given functional time series. Set $\cI_n = \{1\le j \le n: X_{j-1}\in \cU(r_n) \}$.

Then we write $\epsilon'_t$ for the centered residuals which are estimated with $\hat{\Psi}_{n,b}$ (for an oversmoothing bandwidth $b$), i.e.,
\begin{align}\label{DefCenteredResiduals}
	\epsilon'_t = \epsilon'_{n,t} = \hat\epsilon_t - \bar\epsilon, \text{ where } \hat\epsilon_t = \hat\epsilon_{n,t} = X_t - \hat{\Psi}_{n,b}(X_{t-1} ), \quad t=1,\ldots,n+1 \text{ and } \bar{\epsilon} = |\cI_n|^{-1} \sum_{t\in\cI_n} \hat\epsilon_{n,t}
	\end{align}
is the average of the innovations which are estimated with a relatively high precision. Then, $\sum_{t\in\cI_n} \epsilon'_t = 0$. Note that we actually, we have a triangular array of centered residuals $\epsilon'_{n,t}$ but we only write $\epsilon'_t$ for reasons of simplicity.

Denote the random variables which formalize the resampling scheme of the autoregression bootstrap by $\delta_{t,j} = \delta_{t,j}^{(n)}$ ($1\le t\le n$, $j\in\cI_n$). Let $\kappa^{(n)}_t$ be independent and uniformly distributed on the random set $\cI_n$ for $t=1,\ldots,n$. Set
\begin{align*}
		&\delta_{t,j} = \delta_{t,j}^{(n)} = \mathds{1}\{ \kappa^{(n)}_t = j \}, \quad 1\le t \le n, j\in\cI_n.
\end{align*}
Note that $\sum_{j\in\cI_n} \delta^{(n)}_{t,j} = 1$ for all $1\le t\le n$. The procedure for generating the pseudo-time series is the following.
\begin{algorithm}[Generation of the pseudo-data for the autoregression bootstrap]\
	\begin{itemize}\setlength\itemsep{0em}

	\item [(1)] Set  $\epsilon^*_{t} = \epsilon^*_{n,t} = \sum_{j\in\cI_n} \delta^{(n)}_{t,j} \epsilon'_{n,j}$ for $t=1,\ldots,n+1$.

\item [(2)] Generate the bootstrap data $X^*_t=X^*_{n,t}$, for $t=1,\ldots,n+1$ recursively by
$
	X^*_{n,t} = \hatPsib( X^*_{n,t-1} ) + \epsilon^*_{n,t}$, for $t=1,\ldots,n+1$ and for a suitable initial value $X^*_{n,0} = X^*_0$.
\end{itemize}
\end{algorithm}

As $\cS$ is a linear space, the elements of the bootstrapped process are also smooth functions in $\cS$.
We assume that the starting value $X_0^*$ of the bootstrapped process is bounded uniformly in the data, i.e., $\sup_{n\in\N} \E{ \|X^*_0\| \,|\, X_0,\ldots,X_n}$ is $a.s.$ finite.
For instance, a suitable choice is the initial value of the time series $X^*_0 = X_0$.

In practice and if $n>>1$ is large, the choice of $X^*_0$ is of minor importance because of the exponentially decreasing dependence in the stationary FAR(1)-process.

\vspace{1em}
We introduce some notation. We write $\norm{Z}_{p,\p'}$ for the $p$-norm of a real-valued random variable $Z$ defined on a probability space $(\Omega',\cA',\p')$. We denote 
the conditional expectation given the data $X_0,\ldots,X_n$ by $\Ec{\cdot} = \E{\cdot| X_0,\ldots,X_n}$. In the same way, we write $ \p^*_n( A ) \coloneqq \E{ \1{A} \,|\, X_0,\ldots,X_n }$ for the probability measure given the data and $\cL^*(U)$ for the distribution of a random variable $U$ conditional on the data $X_0,\ldots,X_n$.

Denote the projections $\scalar{\Psi(\cdot)}{e_k}$, $\scalar{X_t}{e_k}$ and $\scalar{\epsilon_t}{e_k}$ by $\psi_k(\cdot)$, $X_{k,t}$ and $\epsilon_{k,t}$ for each $k,t\in\N$. 

Let $V,W$ be two normed linear spaces. If $\Gamma\colon V\to W$ is linear, we write $\Gamma.x$ for the image of $x\in V$ under $\Gamma$. If $\Gamma$ is also continuous, its norm is $\norm{\Gamma}_{\cL(V,W)}  = \sup\{ \|\Gamma.x\|_W: x\in V, \|x\|_V \le 1 \} < \infty$ and we write $\Gamma \in \cL(V,W)$.

If $\zeta$ is a real-valued (resp.\ $\cH$-valued) random function which satisfies $\zeta(u)/u \rightarrow 0$ in probability (resp.\ $\norm{\zeta(u)}/u \rightarrow 0$ in probability) as $u$ converges to some limit in $[-\infty,\infty]$, we write $\zeta(u)= \op(u)$. In the same way, we say that $\zeta(u)$ is $\Op(u)$ if $\zeta(u)/u$ (resp.\ $\norm{\zeta(u)}/u$) is bounded in probability as $u$ converges to some limit in $[-\infty,\infty]$. Moreover, we agree to use the analogue notation for $a.s.$ and $a.c.$ convergence.
	
	A Borel probability measure $\mu$ on $\cH$ is a Gaussian measure if and only if its Fourier transform $\hat\mu$ is given by
	$
			\hat\mu(x) \equiv \exp(	i\scalar{m}{x} - \scalar{\cC x}{x}/2 ),
	$
	where $m\in\cH$ and $\cC$ is a positive symmetric trace class operator on $\cH$. $m$ is the mean vector and $\cC$ is the covariance operator of $\mu$. We also write $\fG(m,\cC)$ for this measure $\mu$.
	
We write $\Rightarrow$ for weak convergence on $\cH$, i.e., $\mu_n\Rightarrow\mu$ for probability measures $\mu$ and $(\mu_n:n\in\N)$ if and only if $\int_\cH F\intd{\mu_n}\rightarrow\int_\cH F \intd{\mu}$ as $n\rightarrow \infty$ for all $F\colon\cH\rightarrow\R$ which are bounded and Lipschitz continuous w.r.t.\ $\|\cdot\|$ on $\cH$ and the Euclidean norm on $\R$. Write $F_x(h)$ for the probability $\p(\norm{X_t-x}\le h)$ for $x\in\cH$ and $h>0$. Similarly, $F_x^{\cF_{t-1}}(x) = \p(\|X_t-x\|\le h \,|\, \cF_{t-1})$, where $\cF_t = \sigma(X_0,\ldots,X_t)$. Moreover, $U(x,h)$ is the $h$-neighborhood of $x$ w.r.t.\ $\norm{\cdot}$ in $\cH$.

\section{Asymptotic normality}\label{Section_AsymptoticNormality}
We give the detailed assumptions on the functional autoregression model; these concern the distributional properties of the process $X$, the characteristics of the regression operator and the kernel function as well as the properties of the small ball probability, the bandwidth choices and the covering assumptions of the function space.

We consider the bootstrap of the regression operator at a fixed point $x\in \cU(r_1)$ in the function space.

	\begin{itemize}
	\item [\mylabel{C:Kernel}{(A1)}]	The kernel $K$ has support $[0,1]$, has a continuous derivative $K'\le 0$ and $K(1)>0$.
	
	\item [\mylabel{C:Operator}{(A2)}] 
	\begin{itemize}
	\item  [\mylabel{C:Operator1}{(i)}] $\Psi\colon \cS \to\cH$ is Lipschitz w.r.t.\ $\norm{\cdot}$, i.e., $\norm{\Psi(x)-\Psi(y)}\le L_\Psi \norm{x-y}$, $L_{\Psi}<1$. $\sup_{x\in\cS}\tiom_{\Psi(x)} <\infty$.
		\item [\mylabel{C:Operator2}{(ii)}] 
For all $k\in\N$ and some $\delta>0$ the operator $\psi_k$ admits a linear approximation in a $\delta$-neighborhood of $x$ in the sense that for all $y,z\in U(x,\delta)$, $\psi_k(z)-\psi_k(y) = \intd{\psi_k}(y).(z-y) + R_k(z,y)$ for a bounded linear operator $\intd{\psi_k}(y)\in\cL(\cH,\R)$ and a remainder which satisfies $| R_k(z,y) | \le L_k \norm{z-y}^{1+\alpha}$ for an $\alpha>0$. 

$\sup_{y\in U(x,\delta)} \sum_{k\in\N} \norm{ \intd{\psi_k}(y)}_{\cL(\cH,\R)}^2 < \infty$, $\sum_{k\in\N} L_k^2< \infty$ and $\sum_{k\in\N} \|\intd{\psi_k}(y) - \intd{\psi_k}(x) \|^2_{\cL(\cH,\R)} = \cO( \norm{z-y}^{2(1+\alpha)} )$.
		\item [\mylabel{C:Operator3}{(iii)}] 
Define functions $\varphi_x(\| \epsilon_t-x \|) \coloneqq \E{ \epsilon_t - x |\| \epsilon_t-x \|}$ and $\tilde\varphi_x(\| X_t-x \|) \coloneqq \E{ X_t - x |\| X_t-x \|}$ for $x\in\cH$. Then $\varphi_x(u) = \intd{\varphi}_x(0).u + \bar{R}_x(u)$ and $\tilde\varphi_x(u)= \intd{\tilde\varphi}_x(0).u + \tilde{R}_x(u)$, where $\intd{\varphi}_x(0),\intd{\tilde\varphi}_x(0)\in\cH$ and $\bar{R}_x(u)=\cO(u^{1+\alpha})$ and $\tilde{R}_x(u)=\cO(u^{1+\alpha})$ uniformly in $x\in\cS$ as $u\rightarrow 0$. Moreover, the map $x\mapsto \intd{\varphi}_x(0)$ from $\cS$ to $\cH$ is H{\"o}lder continuous with exponent $\alpha$ and uniformly bounded.

\end{itemize}

\item [\mylabel{C:SmallBallProbability}{(A3)}]
\begin{itemize}
\item [\mylabel{C:SmallBallProbability1}{(i)}] $\p(\| \epsilon_t - y \| \le u) \le \p(\| \epsilon_t - z_0\| \le u)$ for all $y\in\cH$ and $u$ in a neighborhood of 0 for a certain $z_0\in\cH$. $\cS \ni y \mapsto \p( \|\epsilon_t-y\|\le u )/\p( \| \epsilon_t-x \|\le u )$ is H{\"o}lder continuous with H{\"o}lder exponent $\alpha$ (from \ref{C:Operator}~\ref{C:Operator2}) uniformly in $u$ in a neighborhood of 0. For all $y\in\cS$ and $u,\veps$ sufficiently small there is a $\tilde{L}\in\R_+$ such that
\[
		[ \p(\| \epsilon_t-y \| \le u+\veps ) - \p(\| \epsilon_t-y \| \le u-\veps ) ] \, \p(\| \epsilon_t-y \| \le u)^{-1} \le \tilde{L} \, \veps u^{-1}.
\]
\end{itemize}
$\p( \| \epsilon_t-x \|\le h ) = \feps(x)\phi(h)+\geps(h,x)$, where $\feps \ge 0$ and $\geps$ are random functionals with the property $\sup_{x\in\cU} |\geps(h,x)| = o(\phi(h))$ ($h\rightarrow 0$) for all $\cU\in\fU$ and for all $x\in \cS$.

 Set	$\fxt(x)\coloneqq \feps(x-\Psi(X_{t-1}))$, $\gxt(h,x)\coloneqq \geps(h,x-\Psi(X_{t-1}))$ and $\fx(x) \coloneqq \E{ \feps(x-\Psi(X_t)) }$. Assume that $\inf_{x\in \cU\cap \cS} \fx(x)> 0$ $a.s.$ for all $\cU\in\fU$. Then due to the Markov property of the FAR(1)-process	$F_x^{\cF_{t-1}}(h) = \feps(x-\Psi(X_{t-1})) \phi(h) +  \geps(h,x-\Psi(X_{t-1})) = \fxt(x) \phi(h) + \gxt(h,x)$.

\begin{itemize}\setlength\itemsep{0em}
\item [\mylabel{C:SmallBallProbability2}{(ii)}] $ \sup_{x\in\cU}| n^{-1} \sum_{t=1}^n \fxt(x) - \fx(x)| + n^{-1} \sum_{t=1}^n |\gxt(h,x)| \phi(h)^{-1} = \oas(1)$ for all $\cU\in\fU$. 

\item [\mylabel{C:SmallBallProbability3}{(iii)}] $\sup_{u\in [0,1] }	| \phi(hu)/\phi(h) - \tau_0(u) | = o(1)$ as $ h\rightarrow 0$ for some $\tau_0\colon [0,1]\to [0,1]$, where $\tau_0$ is continuous in a neighborhood of 0 and $\tau_0(0)=0$.
\end{itemize}

\item [\mylabel{C:Process}{(A4)}]
\begin{itemize}
\item [\mylabel{C:Process1}{(i)}]
$\sum_{k>m} \mathbb{E}[ |X_{k,t}|^2] \vee \sum_{k>m} \mathbb{E}[ |\epsilon_{k,t}|^2  ] \vee \sum_{k>m} \langle x,e_k \rangle^2 \le a_0\exp(-a_1 m)$, $\forall m\in\N$ for certain $a_0,a_1 \in\R_+$.
\item [\mylabel{C:Process2}{(ii)}]  $\Es{ \norm{\epsilon_t}^{m }} = \cO( m! \tilde{H}^{m-2})$ and $\tiom_{\epsilon_t} \le C < \infty$ $a.s.$ for some $\tilde{H}>1$ and some $C\in\R_+$.
\end{itemize}

\item [\mylabel{C:Bandwidth}{(A5)}]
$h,b\rightarrow 0$, $h/b\rightarrow 0$ as $n\rightarrow\infty$ such that $h(n \phi(h))^{1/2} \rightarrow L^*\in\R_+$, $(n\phi(h))^{1/2} (\log n)^{-(2+\alpha)} \rightarrow \infty$ (where $\alpha$ from \ref{C:Operator}~\ref{C:Operator2}), $[\phi(h)/\phi(b) ] (\log n)^2 = o(1)$, $b(\log n)^{1/2}=o(1)$, $b^{1+\alpha} (n\phi(h))^{1/2}=o(1)$.
\end{itemize}

The conditions on the kernel function and the small ball probability allow us to define the moments for $x\in\cS$
\begin{align*}
			M_{0} \coloneqq K(1) - \int_{0}^1 (s K(s))' \tau_0(s) \intd{s} \text{ and }		M_{j} \coloneqq K(1)^j  - \int_{0}^1 (K(s)^j)' \tau_0(s) \intd{s} > 0, \quad j \ge 1.
\end{align*}
It follows from the assumptions in Condition~\ref{C:SmallBallProbability} that there is a $D(z_0)\in\R_+$ such that $\feps(x) \le \feps(z_0) \le D(z_0)$ and $|\geps(h,x)|/\phi(h) \le \feps(z_0) + |\geps(h,z_0)|/\phi(h) \le D(z_0)$ $a.s.$ for all $h$ in a neighborhood of $0$ for $x\in\cS$. Moreover,
\begin{align}\label{SmallBallLipschitz}
				| \p(\| \epsilon_t-y_1 \| \le u ) - \p(\| \epsilon_t-y_2 \| \le u ) | \, \p(\| \epsilon_t-x \| \le u)^{-1} \le \tilde{L} \|y_1-y_2\|^\alpha\wedge \tilde{B},
\end{align}
for some $\tilde{B}\in\R_+$ and $\alpha$ from Condition~\ref{C:Operator} and for all $u$ in a neighborhood of 0.

Condition~\ref{C:Kernel} on the kernel function $K$ is standard in the functional kernel regression model of Ferraty et al. A Lipschitz constant which is smaller than 1 in Condition~\ref{C:Operator}~\ref{C:Operator1} is necessary to obtain an exponential decay of the influence of the past values of the time series. Condition~\ref{C:Operator}~\ref{C:Operator2} and \ref{C:Operator3} are important for the convergence results, in particular, for the limiting expression of the bias.

Condition~\ref{C:SmallBallProbability} on the small ball probability function appears in similar version also in \cite{laib2010nonparametric} and \cite{laib2011rates}, who also give examples of functional processes which satisfy these criteria.

Condition~\ref{C:Process} is a moment condition. Note that we require the error terms and the observations $X_t$ to have uniformly bounded Lipschitz constants. However, we do not require the functions $\epsilon_t$ and $X_t$ to be uniformly bounded.

Condition~\ref{C:Bandwidth} on the bandwidth is similar as in \cite{ferraty_validity_2010}.

\vspace{1em}
Before we proceed, we need two technical conditions on the estimator $\hatPsib(\cdot)$, the first is a covering condition, the second concerns dependence relations. 

Let $\overline{L}\in\R_+$, $\kappa_n\in\N_+$, $\ell_n > 0$ such that $\log \kappa_n \le \overline{L}\, b h^{-1} \log n$ and $\ell_n=o(b (n\phi(h))^{-1/2}  (\log n)^{-1} )$. We say a set $\cW \subseteq \cS$ satisfies ($\ast$) at $n$ if there are points $z_{n,1},\ldots,z_{n,\kappa_n}\in\cH$ such that $\cW \subseteq \cup_{i=1}^{\kappa_n} U(z_{n,i},\ell_n)$.

Set
$
		\cV_n = \{ \hatPsib(y_1)-\hatPsib(y_2)+X_t-\bar\epsilon \;|\; y_1,y_2\in \cS, 1\le t\le n+1 \} \cap U(x,h).
$
\begin{itemize}
\item [\mylabel{C:Covering}{(A6)}]
$\p( \cV_n \text{ satisfies ($\ast$) at $n$} ) \rightarrow 1$.
\end{itemize}
The intention of Condition~\ref{C:Covering} is that the complexity of the relevant set of functions in the $h$-neighborhood of $x$ w.r.t.\ $\|\cdot\|$ decreases sufficiently fast as $h$ tends to zero. Covering conditions as \ref{C:Covering} are frequently used in this context, see also \cite{ferraty_validity_2010} for a variant where the random set $\cV_n$ is replaced by the underlying space which in the present case is $\cS$.

Furthermore, define the random variable $A_{n,t} =  \Psi(X_{t-1})-\hat\Psi_{n,b}(X_{t-1})-\bar{\epsilon}\in\cS$ for $1\le t \le n$ and $n\in\N$. Let $\mathbb{Q}_{n,t}$ be the probability distribution of $(\epsilon_t,A_{n,t},\1{t\in\cI_n})$ on $\cS^2\times \{0,1\}$ and let $\tilde{\mathbb{Q}}_{n,t} = \p_{\epsilon_t}\otimes\p_{ (A_{n,t},\1{t\in\cI_n})}$ be the product measure of $\epsilon_t$ and $(A_{n,t},\1{t\in\cI_n})$ on $\cS^2\times \{0,1\}$. We need another technical condition regarding the estimates.
\begin{itemize}
\item [\mylabel{C:BetaMixing}{(A7)}] The $\beta$-mixing coefficient $\beta(\epsilon_t,(A_{n,t},\1{t\in \cI_n}))$ is of order $o(\phi(h) )$. $\mathbb{Q}_{n,t} \ll \tilde{\mathbb{Q}}_{n,t}$ such that the Radon-Nikod{\'y}m derivative is uniformly essentially bounded in that
\[
	g_{n,t} = \frac{\intd{\mathbb{Q}_{n,t}}}{\intd{\tilde{\mathbb{Q}}_{n,t}}} \text{ satisfies } \sup_{\substack{n\in\N\\ 1\le t\le n}} \|g_{n,t}\|_{\infty, \tilde{\mathbb{Q}}_{n,t} } < \infty.
	\]
	Moreover, there are $L_{n,t}\in\R_+$, $1\le t \le n$, $n\in\N$ such that $n^{-1} \sum_{t=1}^n L_{n,t}=o(1)$ and $\sup_{n\in\N, 1\le t\le n} L_{n,t} < \infty$ as well as for all $h$ in a neighborhood of 0
	\begin{align}
	\begin{split}\label{C:DiffConditionalDistribution1}
			&\sup_{z\in\cS} \: \phi(h)^{-1} \Big|\p(\|\epsilon_t-z\|\le h \,|\, A_{n,t}, X_1, \ldots, X_{t-1})- \p(\|\epsilon_t-z\|\le h )		\Big| \le L_{n,t} \; a.s.,  \end{split}\\
			\begin{split}\label{C:DiffConditionalDistribution2}
			&\sup_{k\in\N} \: \sup_{z\in\cS} \: \phi(h)^{-1} \Big| \E{K( h^{-1} \|\epsilon_t-z\|) (h^{-1}\langle \epsilon_t-z, e_k\rangle) \,|\, A_{n,t},X_1, \ldots, X_{t-1}} \\
			&\qquad\qquad\qquad\qquad -  \E{K( h^{-1} \|\epsilon_t-z\|) (h^{-1}\langle \epsilon_t-z, e_k\rangle)}		\Big| \le L_{n,t} \; a.s.\end{split}
	\end{align}
	\end{itemize}
The first requirement in Condition~\ref{C:BetaMixing} is satisfied if the $\beta$-mixing coefficient is $o(n^{-1})$. As by construction $\epsilon_t$ and $X_{t-1}$ are independent, this coefficient can be considered as the variation of $\epsilon_t$ with $\Psi(X_{t-1})-\hat\Psi_{n,b}(X_{t-1})$ and $\bar\epsilon$. The first part of Condition~\ref{C:BetaMixing} allows us to exchange the measure $\mathbb{Q}_{n,t}$ with $\tilde{\mathbb{Q}}_{n,t}$ at negligible costs, see also Lemma~\ref{L:BetaMixing}. \eqref{C:DiffConditionalDistribution1} roughly means the the averaged difference between the conditional and unconditional distribution of $\|\epsilon_t-x\|$ vanishes $a.s.$ \eqref{C:DiffConditionalDistribution2} has a similar implication.

We come to the construction of the sequence $(r_n: n\in\N)$ which determines which of estimated residuals are used in the resampling. $(r_n: n\in\N)$ satisfies the three technical conditions
\begin{align}
&\sup_{x\in \cU(r_n)} \| \hatPsib(x)-\Psi(x) \| = \op(1),\label{C:UniformConv1}\\
\begin{split}\label{C:UniformConv2}
&\sup_{z\in\cU(r_n)} \sup_{s\in [0,1]} (|\cI_n|\, \phi(h))^{-1} \Big| \sum_{t=1}^n \1{\|z+\epsilon'_t-x\|\le hs}\1{t\in\cI_n} \\
&\qquad\qquad\qquad\qquad -\p(\|z+\epsilon'_t-x\|\le hs, t\in\cI_n) \Big| = \op(1), \end{split}\\
\begin{split}\label{C:UniformConv3}
&\sup_{z\in\cU(r_n)} (|\cI_n|\,\phi(h))^{-1} \Big\| \sum_{t=1}^n K(h^{-1} \|z+\epsilon'_t-x\| ) (h^{-1} (z+\epsilon'_t-x) ) \1{t\in\cI_n} \\
&\qquad\qquad\qquad\qquad - \E{K(h^{-1} \|z+\epsilon'_t-x\| ) (h^{-1} (z+\epsilon'_t-x) ) \1{t\in\cI_n} }  \Big\| = \op(1).\end{split}
\end{align}
We show in Appendix~\ref{AppendixA} the existence of such a $(r_n: n\in\N)$ in the case where the dimension $d$ of $\cD$ is 1.

The main result of this manuscript concerns the asymptotic normality of the bootstrapped regression operator $\hatPsiB$ at the point $x$.
\begin{align*}
		\hatPsiB (x) \coloneqq 
		\begin{cases}
			\sum_{t=1}^n X^*_{t+1} \Delta^*_{h,t}(x) \big/ \sum_{t=1}^n \Delta^*_{h,t}(x)  &\text{if }  x\in \cS\cap\cU(r_n), \\
			\overline{X}_n &\text{if }  x \in \cS\setminus \cU(r_n),
			\end{cases}
			\end{align*}
where $\Delta^*_{h,t}(x) = K(h^{-1} \|X^*_{t} - x \| )$. We introduce the two families of probability distributions on $\cH$
\begin{align*}
		&\mu_{x,n} = \cL\big( \sqrt{n \phi(h) } (\hatPsi(x)- \Psi(x) ) \big),\\
		&\mu^*_{x,n} = \cL^*\big( \sqrt{n \phi(h)} (\hatPsiB(x)-\hatPsib(x) ) \big)
		\end{align*}
for $n\in\N_+$. Moreover, define $\bar{B}(x) = L^* \,M_{0}\fx(x)\, \{\sum_{k\in\N} \intd{\psi}_k(x) . \intd{\varphi_x(0)}  e_k \} \in\cH$, where $L^*$ is defined in \ref{C:Bandwidth}.

\begin{theorem}\label{BootstrapDistribution}\
Let $\cC_x$ be the unique covariance operator on $\cH$ which is characterized by the condition
$
		\scalar{\cC_x v}{v} = M_2 (M_1^2 f_X(x) )^{-1} \mathbb{E}[\scalar{\epsilon_t}{v}^2 ]
$
for all $v\in\cH$. Then
\begin{align*}
		&\cL(\sqrt{n \phi(h) } (\hatPsi(x)- \Psi(x) ) ) \Rightarrow \fG( \bar{B}(x), \cC_x),\\
		&\cL^*(\sqrt{n \phi(h)} (\hatPsiB(x)-\hatPsib(x) ) ) \Rightarrow \fG( \bar{B}(x), \cC_x) \text{ in probability.}
		\end{align*}
Furthermore, let $F\colon \cH\rightarrow\R$ be a bounded Lipschitz functional. Then $\big| \int_{\cH} F \intd{\mu^*_{x,n}} - \int_{\cH} F \intd{\mu_{x,n}} \big| \rightarrow 0$ in probability. In particular, for the projection in a specific direction $v \in\cH$
\begin{align*}
		\sup_{z\in\R} \left| \p^*\left( \sqrt{n \phi(h)} \scalar{\hatPsiB(x)-\hatPsib(x)}{v} \le z \right) - \p\left( \sqrt{n \phi(h)} \scalar{\hatPsi(x)- \Psi(x)}{v} \le z \right) \right| = o_p(1).
\end{align*}
\end{theorem}

This means that the bootstrapped regression operator has the same Gaussian limiting distribution in probability as the estimated regression operator. This convergence can also be characterized in terms of bounded Lipschitz continuous functionals $F$. Concerning the one-dimensional projections, we also obtain a uniform convergence result with P{\'o}lya's theorem. This is a generalization of the result given in \cite{franke_bootstrap_2002} for the autoregression bootstrap of a real-valued AR(1) process.

In contrast to the existing results for the wild and naive bootstrap in functional kernel regression (see \cite{ferraty_validity_2010} or in \cite{krebs2018doublefunctional}), the autoregression bootstrap only approximates the true distribution in probability and not $a.s.$ This is mainly because in our case the bootstrap regressors $X^*_t$ do not coincide with the original regressors $X_t$ and, thus, have different, more complex stochastic properties.

Finally, consider the scaling factor $(n\phi(h))^{1/2}$. Let $\hat{F}_{n,x}(h)=n^{-1} \sum_{i=1}^n \1{ d(x,X_{n,i})\le h}$ be the empirical version of $F_x(h)$. Then the above results remain valid if we replace $(n \phi(h))^{1/2}$ by $(n \hat{F}_{n,x}(h))^{1/2}$ and omit the factor $f_X(x)$ in the definition of the covariance operator $\cC_x$. Indeed, $F_x(h)/\phi(h) = f_X(x) + o(1)$ and if $\hat{F}_{n,x}(h)/F_x(h)\rightarrow 1$ as $n\rightarrow\infty$, the claim follows directly from Slutzky's theorem, see also \cite{ferraty_validity_2010} and \cite{laib2010nonparametric} for this extension.

\section{Technical results}\label{Section_Proofs}

In this section, we write $\cF^*_{t}$ for the $\sigma$-field generated by the random variables $X^*_0,\ldots,X^*_t$ for $t=0,\ldots,n$. So $\Ec{ Z | \cF^*_t} = \E{ Z | X_0,\ldots, X_n, X^*_0,\ldots,X^*_t }$ for a random variable $Z$.

Since the autoregression bootstrap is a residual-based bootstrapping procedure, we need to study the approximability of the innovations by the bootstrap innovations, where the latter are drawn from the centered sample residuals. This is done in the following in terms of the Mallows metric $\operatorname{d}_2$ which is discussed in detail in \cite{bickel1981}: let $F$ and $G$ be two distributions on $\cS$ and $p\ge 1$, then the Mallows distance between $F$ and $G$ is given by
\[
 \md{F}{G} \coloneqq \inf  \E{ \norm{U-V}^p }^{1/p},
\]
where the infimum is taken over all pairs of random variables $(U,V)$ with values in $\cS^2$ such that $\cL(U)=F$ and $\cL(V)=G$. \cite{bickel1981} prove that the infimum is attained. This fact ensures the existence of a coupled process $\{ \tilde{\epsilon}_t:t=1,\ldots,n+1 \}$ such that the residual $\epsilon_t$ and the bootstrap residual $\epsilon^*_t$ feature the relation
\begin{itemize}\setlength\itemsep{0em}
\item [(i)]	$\{ \tilde{\epsilon}_t:t=1,\ldots,n+1 \}$ are conditionally i.i.d.\ given the data $X_0,\ldots,X_{n+1}$; 
\item [(ii)] the components $\tilde{\epsilon}_t$ have a conditional distribution given the data $X_0,\ldots,X_{n+1}$ which coincides with the unconditional distribution of the $\epsilon_t$, i.e., $\cL^*(\tilde\epsilon_t) = \cL(\epsilon_t)$ for $t=1,\ldots,n+1$;
\item [(iii)] $\operatorname{d}_2(\epsilon_t,\epsilon^*_t) = \operatorname{d}_2(\tilde{\epsilon}_t,\epsilon^*_t) = \mathbb{E}^* [\| \tilde{\epsilon}_t - \epsilon^*_t \|^2 ]^{1/2} = o(1) $ for $t=1,\ldots,n+1$.
\end{itemize}
See also \cite{franke_bootstrap_2002} and the subsequent Theorem~\ref{ConsistencyFARBoot} for the existence of such a process $\tilde\epsilon$.

Define a process $\tilde{X}$ as follows: set $\tilde{X}_0$ such that $\cL^*(\tilde{X}_0) = \cL(X_0)$. Then use the residuals $\tilde{\epsilon}_t$ to make the recursive definition $\tilde{X}_{t} = \Psi(\tilde{X}_{t-1}) +  \tilde{\epsilon}_{t}$. Note that we are actually dealing with a triangular array of random variables $\{\tilde\epsilon_{n,1},\ldots,\tilde\epsilon_{n,n+1} \}$ and $\{ \tilde{X}_{n,1},\ldots,\tilde{X}_{n,n+1} \}$, $n\in\N$.

The proof of Theorem~\ref{BootstrapDistribution} decomposes in several lemmas, we consider the following quantities
\begin{align*}
	\hatPsi(x) - \Psi(x) &= \frac{ (n \phi(h) )^{-1} \sum_{t=1}^n \Delta_{h,t}(x) ( \Psi(X_t)-\Psi(x) ) }{ (n \phi(h) )^{-1}  \sum_{t=1}^n \Delta_{h,t}(x)  } +  \frac{ (n \phi(h) )^{-1} \sum_{t=1}^n \Delta_{h,t}(x) \epsilon_{t+1} }{ (n \phi(h) )^{-1}  \sum_{t=1}^n \Delta_{h,t}(x) } \nonumber \\
	&= \frac{ B_g(x) }{\hat{f}_h(x) } + \frac{ V_g(x) }{\hat{f}_h(x)}.\\
		&\nonumber \\
		\hatPsiB(x) - \hatPsib(x) &= \frac{ (n \phi(h) ) )^{-1} \sum_{t=1}^n \Delta^*_{h,t}(x) ( \hatPsib(X_t^*)-\hatPsib(x) ) }{ (n \phi(h) )^{-1}  \sum_{t=1}^n \Delta^*_{h,t}(x) }  +  \frac{ (n \phi(h) )^{-1} \sum_{t=1}^n \Delta^*_{h,t}(x) \epsilon_{t+1}^* }{ (n \phi(h) )^{-1}  \sum_{t=1}^n \Delta^*_{h,t}(x)  } \nonumber \\
	&= \frac{ B_g^*(x) }{\hat{f}_h^*(x) } + \frac{ V_g^*(x) }{\hat{f}_h^*(x)}	.	
\end{align*}
The approximating property of the $\epsilon^*_t$ is as follows.

\begin{theorem}[Approximation of the innovations in the FAR(1)-model]\label{ConsistencyFARBoot}
Assume that $\E{\norm{\epsilon_t}^p}<\infty$, for some $1\le p <\infty$. Then the bootstrapped innovations satisfy $
		\md{ \epsilon_t}{\epsilon^*_t} = \md{ \tilde{\epsilon}_t }{\epsilon^*_t} = \Ec{ \norm{\tilde{\epsilon}_t - \epsilon^*_t}^p }^{1/p} = \op(1)$ as $n\rightarrow \infty$. 
\end{theorem}
\begin{proof}
First note that by \cite{bickel1981} it is true that $\md{F_{\epsilon} }{F_n } \rightarrow 0$ $a.s.$, where $F_n$ is the empirical distribution of the error terms $\{\epsilon_1,\ldots,\epsilon_t\}$. Thus, it suffices to consider $\md{\hat{F}_n} {F_n}$, where $\hat{F}_n$ is the distribution of the $\epsilon'_t$, $t\in\cI_n$. Therefore, consider the independent random variables
\[
			J\sim \operatorname{unif}(\cI_n),\quad J' \sim \operatorname{unif}(\{1,\ldots,n\}\setminus \cI_n ), \quad W \sim \operatorname{Ber}( (n-|\cI_n|) n^{-1} ).
\]
Set
\[
 U \coloneqq \epsilon'_J, \quad V \coloneqq \epsilon_J \delta(W=0) + \epsilon_{J'} \delta(W=1).
\]
Clearly, $\p^*( U = \epsilon'_t ) = |\cI_n|^{-1}$ for $t\in\cI_n$. Moreover,
\[
			\p^*( V = \epsilon_t ) = \p^*(W=0) \p^*( V = \epsilon_t | W=0 ) + \p^*(W=1) \p^*( V = \epsilon_t | W=1 ) = n^{-1}
\]
if $t\in \cI_n$ and if $t\notin\cI_n$. So $U$ (resp. $V$) has distribution $\hat{F}_n$ (resp. $F_n$). It remains to compute their Mallows distance.
\begin{align}
		\Ec{ \norm{ U - V }^p  } &\le \Ec{ \norm{ U - V }^p  \1{W=0} } + 2^p \Ec{ \|U\|^p + \|V\|^p  \1{W=1} } \nonumber\\
		&= \frac{|\cI_n|}{n} \Ec{ \| \epsilon'_J-\epsilon_J \|^p  } + 2^p \frac{n-|\cI_n|}{n}  \Ec{ \|\epsilon'_J \|^p + \| \epsilon_{J'} \|^p   } \nonumber\\
		\begin{split}\label{EqConsistencyFARBootEq1}
		&\le n^{-1} \sum_{t\in\cI_n} \| \epsilon'_t - \epsilon_t \|^p + 4^p \frac{n-|\cI_n|}{n} \Big\{ |\cI_n|^{-1} \sum_{t\in\cI_n} \|\epsilon'_t-\epsilon_t\|^p \\
		&\qquad\qquad\qquad\qquad + |\cI_n|^{-1} \sum_{t\in\cI_n} \|\epsilon_t\|^p + |n-\cI_n|^{-1} \sum_{t\notin\cI_n} \|\epsilon_t\|^p  \Big\} .
		\end{split}
\end{align}
Consider the first and the second sum in \eqref{EqConsistencyFARBootEq1}
\begin{align}
		 |\cI_n|^{-1}  \sum_{t\in\cI_n} \norm{\epsilon'_t - \epsilon_t }^p &\le 2^p\Big( |\cI_n|^{-1}  \sum_{t\in\cI_n} \norm{\hat\epsilon_t - \epsilon_t }^p + \| \bar\epsilon \|^p	\Big). \label{EqConsistencyFARBootEq2}
\end{align}
It follows by the strong law of large numbers for Hilbert space valued sequences (cf.\ \cite{bosq_linear_2000}) that the last term in \eqref{EqConsistencyFARBootEq2} converges to zero $a.s.$ Moreover,
\begin{align*}
		  |\cI_n|^{-1}  \sum_{t\in\cI_n} \norm{\hat\epsilon_t - \epsilon_t }^p &=   |\cI_n|^{-1} \sum_{t\in\cI_n} \| \hatPsib(X_{t-1})-\Psi(X_{t-1}) \|^p  \le \sup_{x\in \cU(r_n) } \|\hatPsib(x) - \Psi(x) \|^p.
		\end{align*}
This term converges to zero in probability by Corollary~\ref{C:UnifConsistencyHatPsi}. It remains to consider the second and the third sum in the curly parentheses in \eqref{EqConsistencyFARBootEq1}. Clearly, $(n-|\cI_n|) n^{-1} = \op(1)$. We have for the second sum
\begin{align*}
			|\cI_n|^{-1} \sum_{t\in\cI_n} \|\epsilon_t\|^p  &\rightarrow \E{\|\epsilon_t\|^p} \; \text{ in probability}.
\end{align*}
So, the second sum is $\op(1)$. And, if we multiply the third sum in the curly bracket by the factor $(n-|\cI_n|) n^{-1}$, we find
\begin{align*}
\mathbb{E}\Big[ n^{-1} \sum_{t\notin\cI_n} \|\epsilon_t\|^p \Big] &=  \E{ \|\epsilon_t \|^p } \p(t\notin \cI_n) \rightarrow 0.
\end{align*}
 Hence, $\Ec{ \norm{ U - V }^p  } = \op(1)$ and $\operatorname{d}_p(\epsilon_t,\epsilon^*_t) = \operatorname{d}_p(\tilde{\epsilon}_t,\epsilon^*_t) = \op(1)$.
\end{proof}

It follow several lemmas which generalize the results of \cite{franke_bootstrap_2002} to the functional case. 

\begin{lemma}\label{BootstrapCondProb}\
\begin{itemize}\setlength\itemsep{0em}
\item [(i)]	$\max_{1\le t \le n} \p^*( X^*_t \notin  \cU(r_n)  ) = \Op( r_n^{-1}) $ and $\max_{1\le t \le n} \p^*( \hatPsib(X^*_{t-1}) \notin  \cU(r_n)  )= \Op( r_n^{-1}) $.
\item [(ii)] $n^{-1} \sum_{t=1}^n \1{ X^*_t \notin  \cU(r_n)  } = \Op( r_n^{-1})$, $n^{-1} \sum_{t=1}^n \mathds{1}\{ \hat\Psi_{n,b}(X^*_{t-1}) \notin  \cU(r_n)  \} = \Op( r_n^{-1})$  \\
and $n^{-1} \sum_{t=1}^n \1{ \Psi(X^*_{t-1}) \notin  \cU(r_n)  } = \Op( r_n^{-1})$.
\end{itemize}
Moreover, $\1{  X^*_t \notin  \cU(r_n)  , \norm{X^*_t - x}\le h}	\le  \1{r_n < c }$ for some nonrandom $c\in\R_+$ uniformly in $t$ and $n$.
\end{lemma}
\begin{proof}
In the first part, we prove that both $\max_{1\le t \le n} \Ec{ \norm{X^*_t} } = \Op(1)$ and $n^{-1} \sum_{t=1}^n \| X^*_t\| = \Op(1)$. Consider
\begin{align*}
		\norm{ X^*_{t} } &\le \sup_{x\notin  \cU(r_n) } \norm{\hatPsib(x)} \1{X^*_{t-1} \notin  \cU(r_n) } + \sup_{x \in  \cU(r_n) } \norm{ \hatPsib(x)-\Psi(x) } \1{X^*_{t-1} \in  \cU(r_n) } \\
		&\quad + \norm{ \Psi(X^*_{t-1})} \1{X^*_{t-1} \in  \cU(r_n) } + \norm{\epsilon^*_{t}}  .
\end{align*}
Hence, the conditional expectation of $\norm{ X^*_{t} }$ is at most
\begin{align*}
		\Ec{\norm{ X^*_{t} } } &\le \sup_{x\notin \cU(r_n) }\norm{\hatPsib(x)} + \sup_{x \in  \cU(r_n) } \norm{ \hatPsib(x)-\Psi(x) } + L_{\Psi} \Ec{ \norm{X^*_{t-1}} } + \norm{\Psi(0)} + \Ec{\norm{\epsilon^*_{t}} },
		\end{align*}
where by assumption the number $L_{\Psi}<1$. Set 
\[
D\coloneqq \sup_{x\notin \cU(r_n) }\norm{\hatPsib(x)}+ \sup_{x \in  \cU(r_n) } \norm{ \hatPsib(x)-\Psi(x) }+\norm{\Psi(0)} + \Ec{\norm{\epsilon^*_{t}} }.\]
Then $D$ is independent of $t$ and bounded in probability by Theorem~\ref{ConsistencyFARBoot} and Corollary~\ref{C:UnifConsistencyHatPsi}. If we iterate this inequality another $(t-1)$ times, we obtain
	\[ 
	\Ec{\norm{ X^*_{t} } } \le L_{\Psi}^{t}\, \Ec{ \norm{X^*_0} } + \sum_{k=0}^{t-1} L_{\Psi}^k D \le  \Ec{ \norm{X^*_0}} + D/(1-L_{\Psi}) . \]
In particular, $\max_{0\le t \le n} \Ec{\norm{ X^*_{t} } } \le \Ec{ \norm{X^*_0}} + D/(1-L_{\Psi})$ in probability and uniformly in $n\in\N$ because $\sup_{n\in\N} \Ec{ \|X_0^*\| }<\infty$. Moreover, we see that
\begin{align*}
			\| X^*_t\| &\le \|\epsilon^*_t\| + L_\Psi  \| X^*_{t-1}\| + D',
\end{align*}
where $D' = \sup_{x\notin \cU(r_n) } \| \hatPsib(x) \| +  \sup_{x \in  \cU(r_n) } \| \hatPsib(x)-\Psi(x) \| + \| \Psi(0) \|$.
Iterating this inequality and summing up, we obtain
\begin{align}\label{BootstrapCondProbE1}
		n^{-1} \sum_{t=1}^n \| X^*_t\| \le (1-L_\Psi)^{-1} \Big(n^{-1} \sum_{t=1}^n \| \epsilon^*_t\| + D' + \| X^*_0\| \Big).
\end{align}
Now, we use
\begin{align}\label{BootstrapCondProbE2}
		n^{-1} \sum_{t=1}^n \|\epsilon^*_t\| \le \sum_{j\in\cI_n} (n^{-1}\sum_{t=1}^n \delta_{t,j}) (\|\epsilon_j\|+\|\hat\Psi_{n,b}(X_{j-1})-\Psi(X_{j-1})\| ) + \| \bar{\hat{\epsilon}}\|,
\end{align}
where
\[
			 \sum_{j\in\cI_n} (n^{-1}\sum_{t=1}^n \delta_{t,j}) \|\hat\Psi_{n,b}(X_{j-1})-\Psi(X_{j-1})\| \le L_\Psi \sum_{j\in\cI_n} (n^{-1}\sum_{t=1}^n \delta_{t,j}) \|X_{j-1}\| + \Op(1).
\]
Clearly, $n^{-1} \sum_{j\in\cI_n} (\sum_{t=1}^n \delta_{t,j} +1 ) ( \norm{X_j} + \norm{\epsilon_j} )$ is bounded in probability. Thus, \eqref{BootstrapCondProbE2} is $\Op(1)$ and so is \eqref{BootstrapCondProbE1}. This finishes the first part.

We come to the second part of the proof. Consider $\p^*( X^*_t\notin  \cU(r_n)  ) \le  r_n^{-1} \Ec{ \norm{X^*_t}_\infty + \tiom_{X^*_t}  }$. We show that this last conditional expectation is bounded in probability. Note that $\norm{X^*_t}_\infty  \le C \tiom_{X^*_t}^{d/(2+d)} \norm{X^*_t}^{2/(2+d)} $ by Lemma~\ref{BoundInftyNorm} for a constant $C$ which does not depend on the functions $X^*_t$. Consequently,
\begin{align*}
			\Ec{\norm{X^*_t}_\infty } \le C \Ec{ \tiom_{X^*_t} }^{d/(2+d)} \Ec{ \norm{X^*_t} }^{2/(2+d)} .
\end{align*}
So, it remains to consider $ \tiom_{X^*_t}$. For that reason, we use the sublinearity of the map $x\mapsto \tiom_x$ and the definition of the estimator $\hatPsib$. As $\tiom_{\epsilon_t}$ and $\tiom_{\Psi(X_t)}$ are essentially bounded, there is a constant $c^*\in\R_+$ such that for all $t$ and $n$
\begin{align}
			\tiom_{X^*_t} &\le  \tiom_{\epsilon^*_t} + \tiom_{\hat\Psi_{n,b}(X^*_{t-1}) } \le  \tiom_{\epsilon^*_t} + \sum_{j=1}^n \frac{\Delta_{b,j}( X^*_{t-1} ) }{ \sum_{j=1}^n \Delta_{b,j}( X^*_{t-1} )  } \tiom_{X_{j+1}} \nonumber \\
			&\le \sum_{j=1}^n  \delta_{t,j} \left( \tiom_{X_j } + \tiom_{\epsilon_j} \right) + n^{-1}\sum_{j=1}^n \tiom_{X_j} + 2 \sup_{x\in\cS} \sum_{j=1}^n \frac{\Delta_{b,j}( x) }{ \sum_{j=1}^n \Delta_{b,j}( x )  } \tiom_{X_{j+1}} \le c^*. \label{Eq:BootstrapCondProb2}
\end{align}
In particular,	$\Ec{ \tiom_{X^*_t} } \le c^*$ (uniformly in $t$ and $n$). This proves the first statement in (i) which is $\max_{1\le t\le n} \p^*( X^*_t\notin  \cU(r_n)  ) = \Op( r_n^{-1} )$. The second statement in (i) follows similarly. This completes (i). 

We come to (ii). Using Lemma~\ref{BoundInftyNorm}, $\1{a>b}\le b^{-1} a$ if $a,b>0$ and H{\"o}lder's inequality, we obtain
\begin{align*}
			n^{-1} \sum_{t=1}^n \1{ X^*_t \notin  \cU(r_n)  } = \cO\bigg\{ r_n^{-1} \Big( \Big(n^{-1} \sum_{t=1}^n \tiom_{X^*_t} \Big)^{d/(d+2)} (n^{-1}  \sum_{t=1}^n \norm{X^*_t} )^{2/(d+2)} + n^{-1} \sum_{t=1}^n \tiom_{X^*_t} \Big) \bigg\}.
\end{align*}
Note that $n^{-1} \sum_{t=1}^n \norm{X^*_t} $ is $\Op(1)$ by the above \eqref{BootstrapCondProbE1} et seq. Moreover, from \eqref{Eq:BootstrapCondProb2}, $\tiom_{X^*_t} \le c^*$. This shows that $n^{-1} \sum_{t=1}^n \1{ X^*_t \notin  \cU(r_n)  } = \Op( r_n^{-1})$. The second statement in (ii) follows similarly as the first one. For the last statement in (ii), we use additionally the smoothness assumption on the regression operator, i.e., $\tiom_{\Psi(x)} = \cO( 1)$.

We need again the inclusion 
$
					\{ X^*_t \notin  \cU(r_n)   \} \subseteq \{ \norm{X^*_t}_\infty > r_n/2 \quad\text{or} \quad \tiom_{X^*_t} > r_n /2\}
$ for the amendment of the lemma. If we use additionally the requirement that $\norm{X^*_t - x}\le h$, we obtain from Lemma~\ref{BoundInftyNorm}
\[
		\norm{ X^*_t}_\infty \le C \norm{ X^*_t }^{2/(2+d)} \tiom_{X^*_t}^{d/(2+d)} \le C (\norm{x}+h)^{d/(2+d)} \tiom_{X^*_t}^{d/(2+d)} = \cO(1).
\]
Hence, since $r_n \rightarrow\infty$, there is a constant $c\in\R_+$ such that $\{ X^*_t \notin  \cU(r_n) , \norm{X^*_t - x}\le h  \} \subseteq \{ r_n < c \}$.
\end{proof}

\begin{lemma}\label{BootstrapMean}
There are random variables $S_{n,1} = \op(1)$ and $S_{n,2} =\Op(1)$ ($n\rightarrow\infty$) such that
\begin{align*}
		\mathbb{E}^*[ \| X^*_t - \tilde{X}_t \| ] = S_{n,1} + L_{\Psi}^{t-1 } S_{n,2}. \quad \text{In particular, } \mathbb{E}^*\Big[ n^{-1} \sum_{t=1}^n \| X^*_t - \tilde{X}_t \| \Big] = \op(1).
		\end{align*}
\end{lemma}
\begin{proof}
We begin with the decomposition
\begin{align*}
		\|  X^*_t - \tilde{X}_t \| &\le \| \hatPsib(X^*_{t-1}) - \Psi(\tilde{X}_{t-1} ) \| + \| \epsilon^*_t - \tilde{\epsilon}_t \|
		\end{align*}
The conditional expectation of the first difference is bounded above by
\begin{align}
		&\Ec{ \| \hatPsib( X^*_{t-1} ) - \Psi( \tilde{X}_{t-1} )  \| } \nonumber \\
		\begin{split}\label{EqBootstrapMean1}
		&\le \Ec{ \| \hatPsib( X^*_{t-1} ) \| \1{ X^*_{t-1} \notin  \cU(r_n)   } }
				+ \Ec{\| \hatPsib( X^*_{t-1} ) - \Psi( X^* _{t-1} )  \| 	\1{ X^*_{t-1} \in  \cU(r_n)  }	}  \\
		&\quad + \Ec{\| \Psi( X^*_{t-1} ) - \Psi( \tilde{X}_{t-1} )  \| 	\1{ X^*_{t-1} \in  \cU(r_n)  }	}		+ \Ec{\|\Psi( \tilde{X}_{t-1} )\| \1{ X^*_{t-1} \notin  \cU(r_n)  } } 
		\end{split}
	\end{align}
The first term in \eqref{EqBootstrapMean1} is at most $\|\bar{X}_n\| \max_{1\le t\le n} \pc\left(	X^*_{t-1} \notin  \cU(r_n)  \right) $ and vanishes in probability by Lemma~\ref{BootstrapCondProb}. The second term is $\op(1)$ by Corollary~\ref{C:UnifConsistencyHatPsi}. We apply the Lipschitz continuity of the regression operator to the third term and obtain the bound $L_{\Psi} \mathbb{E}^*[ \| X^*_{t-1} - \tilde{X}_{t-1} \|]$. 

The fourth term vanishes by the following argument: let $\tilde{\cU} = \cU(\tilde{r})$ for some $\tilde{r}>0$. Then
\begin{align}
		\Ec{\| \Psi( \tilde{X}_{t-1} ) \| \mathds{1}\{ X^*_{t-1} \notin  \cU(r_n)  \} }  &= \Ec{\| \Psi( \tilde{X}_{t-1} ) \| \mathds{1}\{ \tilde{X}_{t-1} \notin \tilde{\cU} \} \1{ X^*_{t-1} \notin  \cU(r_n)  } } \nonumber \\
		&\quad + \Ec{\| \Psi( \tilde{X}_{t-1} ) \| \mathds{1}\{ \tilde{X}_{t-1} \in \tilde{\cU} \} \mathds{1}\{ X^*_{t-1} \notin  \cU(r_n)  \}}  \nonumber \\
		\begin{split}\label{EqBootstrapMean2}
		&\le \Ec{\| \Psi( \tilde{X}_{t-1} ) \| \mathds{1}\{ \tilde{X}_{t-1} \notin \tilde{\cU}  \}  }\\
		&\quad + \Ec{ \left(\norm{\Psi(0)} + L_{\Psi} \nu(\mathcal{D})^{1/2} \tilde{r} \right)  \1{ X^*_{t-1} \notin  \cU(r_n)  } },
		\end{split}
\end{align}
where we use in the last inequality for the second summand that $\|\tilde{X}_{t-1} \| \le \nu(\mathcal{D})^{1/2} \| \tilde{X}_{t-1} \|_{\infty}$.

The first term in \eqref{EqBootstrapMean2} is equal to $\Es{\norm{ \Psi( X_{t-1} ) } \mathds{1} \{X_{t-1} \notin \tilde{\cU}\} } $ and can be made arbitrarily small with an appropriate choice of $\tilde{r}$ and with an application of Lebesgue's dominated convergence theorem. The second term in \eqref{EqBootstrapMean2} vanishes in probability for a fixed $\tilde{r}$ with the help of Lemma~\ref{BootstrapCondProb}. All in all, for each $t$
\[
 \Ec{ \| \hatPsib(X^*_{t-1}) - \Psi( \tilde{X}_{t-1} ) \| } \le L_{\Psi} \Ec{ \| X^*_{t-1} - \tilde{X}_{t-1} \| } + U_n,
\]
where the random variable $U_n= \op(1)$ ($n\rightarrow\infty$) and does not depend on the index $t$ but only on the sample size $n$. The convergence of $\norm{ \epsilon^*_t - \tilde{\epsilon}_t }$ is considered in Theorem~\ref{ConsistencyFARBoot}. We have that $\Ec{ \norm{ \epsilon^*_t - \tilde{\epsilon}_t } }  = \op(1)$ and uniformly in $t$ because the innovations are i.i.d.\

Thus, there is a random variable $V_n$ which is in $\op(1)$ and  which is independent of the index $t$ such that
\[
\Ec{ \| X^*_t - \tilde{X}_t \|} \le L_{\Psi} \Ec{ \| X^*_{t-1} - \tilde{X}_{t-1} \| } + V_n.\]
Consequently, $\mathbb{E}^*[ \| X^*_t - \tilde{X}_t \| ] \le \sum_{k=0}^{t-1} L_{\Psi}^{k} V_n + L_{\Psi}^{t-1}(\mathbb{E}^*[\| \tilde{X}_0 \|] + \Ec{\| X^*_0 \|} )$ and the claim follows because by construction $\mathbb{E}^*[\| \tilde{X}_0 \|] = \E{\|X_0\|}$. Moreover, $\sup_{n\in\N}\mathbb{E}^*[\| X^*_0 \|]<\infty $ $a.s.$
\end{proof}

\begin{lemma}\label{ErgodicityBootstrap}
$n^{-1} \sum_{t=1}^n \|X^*_t - \tilde{X}_t \| = \op(1)$.
\end{lemma}
\begin{proof}
We begin with the same decomposition as in the proof of Lemma~\ref{BootstrapMean}:
\begin{align}
\begin{split}\label{EqErgodicityBootstrap1}
		\| X^*_t-\tilde{X}_t \| &\le \Big\{\norm{\epsilon^*_t-\tilde{\epsilon}_t} + \sup_{x\in \cU(r_n) } \| \hatPsib(x)-\Psi(x) \| \Big\} + L_{\Psi} \| X^*_{t-1}-\tilde{X}_{t-1} \| \\
		&\quad  + \Big\{ \sup_{x\notin \cU(r_n) } \| \hatPsib(x) \| + \| \Psi(\tilde{X}_{t-1}) \| \Big\} \1{X^*_{t-1}\notin \cU(r_n) }.
		\end{split}
\end{align}
Define the quantities 
\begin{align*}
		&A_n \coloneqq\sup_{x\in \cU(r_n) } \| \hatPsib(x)-\Psi(x) \| \text{ and } 	\\
		&B_{n,t} \coloneqq  \| \epsilon^*_{t}-\tilde{\epsilon}_{t} \| + \Big\{ \sup_{x\notin \cU(r_n) } \norms{\hatPsib(x)} + \norms{\Psi(\tilde{X}_{t-1})} \Big\} \1{X^*_{t-1 } \notin \cU(r_n) }.
		\end{align*}
Then the empirical mean of \eqref{EqErgodicityBootstrap1} is bounded above by
\begin{align}\label{EqErgodicityBootstrap2}
		n^{-1} \sum_{t=1}^n \| X^*_t - \tilde{X}_t \| &\le n^{-1} \sum_{t=1}^n \Big\{ \sum_{k=0}^{t-1} L_{\Psi}^k (A_n + B_{n,t-k})  + L_{\Psi}^t \left(\norm{X^*_0} + \| \tilde{X}_0 \| \right) \Big\}.
\end{align}
Clearly, the sum involving the term $A_n$ and the norm of the initial values $\tilde{X}_0$ and $X^*_0$ vanishes in probability. Hence, it remains to consider the sum involving the $B_{n,t-k}$. We obtain after changing the order of summation
\begin{align}
\begin{split}\label{EqErgodicityBootstrap3}
			n^{-1} \sum_{t=1}^n\sum_{k=0}^{t-1} L_{\Psi}^k  B_{n,t-k} &= L_{\Psi}^{-1} n^{-1} \sum_{k=1}^{n} L_{\Psi}^k \sum_{t=k}^n B_{n,t-k+1}.
			\end{split}
\end{align}	
Consider the right-hand side of \eqref{EqErgodicityBootstrap3} for the component $\|\Psi(\tilde{X}_{t-k})\| \1{ X^*_{t-k} \notin  \cU(r_n) }$ of $B_{n,t-k+1}$
\begin{align}
			&\sum_{k=1}^n L_{\Psi}^k \Big( n^{-1} \sum_{t=k}^n	\|\Psi(\tilde{X}_{t-k})\| \1{ X^*_{t-k} \notin  \cU(r_n) }	\Big) \nonumber \\
			&\le \sum_{k=1}^n L_{\Psi}^k \Big( n^{-1} \sum_{t=k}^n	\|\Psi(\tilde{X}_{t-k})\|^2 	\Big)^{1/2}  \Big( n^{-1} \sum_{t=k}^n	\1{ X^*_{t-k} \notin  \cU(r_n) }	\Big)^{1/2} \nonumber\\
			&\le \left( \sum_{k=1}^n L_{\Psi}^k \right) \Big( n^{-1} \sum_{t=0}^n	\|\Psi(\tilde{X}_t)\|^2 	\Big)^{1/2}  \Big( n^{-1} \sum_{t=0}^n	\1{ X^*_t \notin  \cU(r_n) }	\Big)^{1/2}  \label{EqErgodicityBootstrap4}
\end{align}
The first and second factor in \eqref{EqErgodicityBootstrap4} are bounded in probability. The third factor is $\op(1)$ by Lemma~\ref{BootstrapCondProb}. Hence, \eqref{EqErgodicityBootstrap4} is $\op(1)$.

Next, we treat the term which involves the innovations: by construction $n^{-1}\sum_{t=1}^n \norm{\epsilon^*_t-\tilde{\epsilon}_t}$ is a sum of i.i.d. random variables (conditionally on the data $X_0,\ldots,X_n$). Hence, it converges in probability to its (conditional) mean which converges to 0 by Theorem~\ref{ConsistencyFARBoot}. In particular,
\[
\sum_{k=1}^n L_{\Psi}^k \; n^{-1}\sum_{t=k}^n \norm{\epsilon^*_{t-k}-\tilde{\epsilon}_{t-k} } \le \Big(\sum_{k=1}^n L_{\Psi}^k \Big) \Big(n^{-1}\sum_{t=1}^n \norms{\epsilon^*_{t}-\tilde{\epsilon}_{t} } \Big)  = \op(1).
\]
All in all, \eqref{EqErgodicityBootstrap3} is $\op(1)$. This completes the proof of the assertion.
\end{proof}

The following result is crucial for the uniform approximation $(\tilde{X}_t)$ by the bootstrap process $(X^*_t)$.
\begin{proposition}\label{BootstrappedInnovations}
\begin{align}
	\sup_{s\in [0,1]} \phi(h)^{-1} \left| \p^*(\|\tilde{X}_t-x \| \le hs ) -  n^{-1}\sum_{t=1}^n \p^*\left( \norm{X^*_t-x} \le hs	\big| \cF^*_{t-1} \right) \right| = \op(1). \label{Eq:BootstrappedInnovations0} \\
	\Ec{ \sup_{s\in [0,1]} \phi(h)^{-1} \left| \p^*(\|\tilde{X}_t-x \| \le hs ) - n^{-1} \sum_{t=1}^n \p^*\left( \norm{X^*_t-x} \le hs	\big| \cF^*_{t-1} \right) \right| } = \op(1). \label{Eq:BootstrappedInnovations0b}
	\end{align}
\end{proposition}
\begin{proof}
First we use that \eqref{Eq:BootstrappedInnovations0} is at most
\begin{align}
		&\sup_{s \in [0,1]} (n\phi(h) )^{-1} \left| \sum_{t=1}^n \p^*\left( \norm{ z +\epsilon^*_0 -x}  \le h s \right)  -\p^*\left( \norm{z+  \tilde\epsilon_{0} -x}  \le hs  \right) \Big|_{z=\hatPsib(X^*_{t-1})} \right| \label{Eq:BootstrappedInnovations1} \\
		&\quad + \sup_{s\in [0,1]}  (n\phi(h))^{-1} \Bigg| \sum_{t=1}^n  \p^*( \norm{z+\tilde\epsilon_0-x} \le hs  )\Big|_{z = \hatPsib(X^*_{t-1}) } - \p^*( \norm{z+\tilde\epsilon_0-x} \le hs  )\Big|_{z = \Psi(X^*_{t-1}) }  \Bigg| \label{Eq:BootstrappedInnovations2} \\
		&\quad + \sup_{s\in [0,1]}  (n\phi(h))^{-1} \left| \sum_{t=1}^n \p^*( \norm{z+\tilde\epsilon_0-x} \le hs  ) \Big|_{z = \Psi(X^*_{t-1}) } - \p^*( \norm{z+\tilde\epsilon_0-x} \le hs  ) \Big|_{z = \Psi(\tilde{X}_{t-1}) } \right|  \label{Eq:BootstrappedInnovations3}  \\
		&\quad + \sup_{s\in [0,1]} \phi(h)^{-1} \left| n^{-1} \sum_{t=1}^n  \p^*( \norm{z+\tilde\epsilon_0-x} \le hs  ) \Big|_{z = \Psi(\tilde{X}_{t-1}) } - \Ec{ \p^*( \norm{z+\tilde\epsilon_0-x}\le hs) \Big|_{z = \Psi(\tilde{X}_{t-1}) } } \right| \label{Eq:BootstrappedInnovations4}
\end{align}
In the rest of the proof, we focus on the first statement and bound above the quantities from \eqref{Eq:BootstrappedInnovations1} to \eqref{Eq:BootstrappedInnovations4}. It is straightforward to use these upper bounds in combination with the previous lemmata to prove the second statement, so we only give details for the second statement at points where it is not immediately obvious.

We start with \eqref{Eq:BootstrappedInnovations2}. Using the H{\"o}lder continuity of the small ball probability from \eqref{SmallBallLipschitz}, \eqref{Eq:BootstrappedInnovations2} is of order
\begin{align*}
			n^{-1} \sum_{t=1}^n (\| \hat\Psi_{n,b}(X^*_{t-1}) - \Psi(X^*_{t-1}) \|^\alpha \wedge \tilde{B} ) &= \cO\Big( \sup_{x\in \cU(r_n)  } \|\hat\Psi_{n,b}(x)-\Psi(x)\|^\alpha \Big)  \\
			&\quad + \Op\Big( n^{-1} \sum_{t=1} ^n  \1{X^*_{t-1}\notin \cU(r_n) }   \Big).
\end{align*}
Using Corollary~\ref{C:UnifConsistencyHatPsi}, the first term vanishes in probability. Moreover, using Lemma~\ref{BootstrapCondProb}, we see that also the second term is $\op(1)$.

Similarly, we find that \eqref{Eq:BootstrappedInnovations3} is $\Op( n^{-1}\sum_{t=1}^n (L_\Psi^\alpha \| X^*_{t-1}-\tilde{X}_{t-1} \|^\alpha \wedge \tilde{B} ) )$ which is $\op(1)$, using Lemma~\ref{ErgodicityBootstrap}

Next, we use that the r.v.'s $\tilde\epsilon_t$ and $\tilde{X}_t$ have the same law conditional on the data as the r.v.'s $\epsilon_t$ and $X_t$. Then \eqref{Eq:BootstrappedInnovations4} can be expressed in terms of the small ball probability functions $\fxt$ and $\gxt$ as
\begin{align}		\begin{split}\label{Eq:BootstrappedInnovations5}
			&\left|n^{-1} \sum_{t=1}^n \feps(x-\Psi(\tilde{X}_{t-1}) ) -	\Ec{ \feps(x-\Psi(\tilde{X}_{t-1}))}	\right|  \\
			&\quad +  (n\phi(h))^{-1} \sum_{t=1}^n |\geps(hs,x-\Psi(\tilde{X}_{t-1}))|  +  \Ec{  (n\phi(h))^{-1} \sum_{t=1}^n |\geps(hs,x-\Psi(\tilde{X}_{t-1}))| }.
			\end{split}
\end{align}
The first term and the second term in \eqref{Eq:BootstrappedInnovations5} vanish $a.s.$ by assumption.

Moreover, $|\gxtt(h,x)| \phi(h)^{-1} \le  D(z_0) < \infty$ $a.s.$ So, an application of Lebesgue's dominated convergence theorem yields that also the third term is $\oas(1)$. (For completeness, we remark that		$\mathbb{E}^*[|n^{-1} \sum_{t=1}^n \feps(x-\Psi(\tilde{X}_{t-1}) ) -	\mathbb{E}^*[ \feps(x-\Psi(\tilde{X}_{t-1})) ]	| ] = \oas(1)$ because of the properties of the coupled process $\tilde{X}_t$ and because $f_{X,t}(x) \le D(z_0)$. This proves also that the conditional expectation of \eqref{Eq:BootstrappedInnovations5} and \eqref{Eq:BootstrappedInnovations4} vanishes $a.s.$)

We conclude with \eqref{Eq:BootstrappedInnovations1} which is bounded above by
\begin{align}
		&\sup_{s \in [0,1]} \sup_{z\in \cU(r_n) } \phi(h)^{-1} \Big| \p^*\left( \norm{ z +\epsilon^*_0 -x}  \le h s \right)  -\p^*\left( \norm{z+  \tilde\epsilon_{0} -x}  \le hs  \right) \Big| \label{Eq:BootstrappedInnovations6} \\
		&\quad+  (n\phi(h))^{-1} \sum_{t=1}^n \p^*\left( \norm{ z +\epsilon^*_0 -x}  \le h\right) \Big|_{z=\hatPsib(X^*_{t-1})} \1{\hatPsib(X^*_{t-1})\notin \cU(r_n) } \label{Eq:BootstrappedInnovations7} \\
		&\quad+ (n\phi(h))^{-1} \sum_{t=1}^n  \p^*\left( \norm{z+  \tilde\epsilon_{0} -{x} }  \le h  \right) \Big|_{z=\hatPsib(X^*_{t-1})} \1{\hatPsib(X^*_{t-1})\notin \cU(r_n) } . \label{Eq:BootstrappedInnovations8}
\end{align}
The first summand is bounded above by
\begin{align}
		&\sup_{s \in [0,1]} \sup_{z\in \cU(r_n) } (|\cI_n| \phi(h))^{-1} \Big| \sum_{t=1}^n \1{\|z+\epsilon'_t-x\|\le hs} \1{t\in\cI_n}  -	\p(\|z+\epsilon'_t-x\|\le hs, t\in \cI_n) \Big| \label{Eq:BootstrappedInnovations6a} \\
		&+\sup_{s \in [0,1]} \sup_{z\in \cU(r_n) } \phi(h)^{-1} \Big| \p(\|z+\epsilon'_t-x\|\le hs, t\in \cI_n)  - \p(\|z+\epsilon_t-x\|\le hs, t\in \cI_n) \Big| \label{Eq:BootstrappedInnovations6b} \\ 
				&+\sup_{s \in [0,1]} \sup_{z\in \cU(r_n) } \phi(h)^{-1} \Big| \p(\|z+\epsilon_t-x\|\le hs, t\in \cI_n) - \p(\|z+\epsilon_t-x\|\le hs)\Big| \label{Eq:BootstrappedInnovations6c} 
\end{align}
\eqref{Eq:BootstrappedInnovations6a} is $\op(1)$ by \eqref{C:UniformConv2} resp.\ Proposition~\ref{P:AsConvergenceResiduals}. For \eqref{Eq:BootstrappedInnovations6b}, use that $A_{n,t}=\Psi(X_{t-1})-\hat\Psi(X_{t-1})-\bar\epsilon$. We apply Proposition 2 in \cite{krebs2018large} which allows us to exchange the first probability in \eqref{Eq:BootstrappedInnovations6b} by the expectation w.r.t.\ the product measure, i.e.,
\begin{align*}
			&\phi(h)^{-1} |\E{ \p( \|z+\epsilon_t+a-x\|\le hs)|_{a=A_{n,t}} \1{t\in\cI_n} } -  \p( \|z+\epsilon_t+A_{n,t}-x\|\le hs,t\in\cI_n) | \\
			&\le 4 \phi(h)^{-1} \beta(\epsilon_t,(A_{n,t},\1{t\in \cI_n})).
\end{align*}
Hence, using $\p( \|z+\epsilon_t-x\|\le hs,t\in\cI_n)= \p( \|z+\epsilon_t-x\|\le hs)\p(t\in\cI_n)$, \eqref{Eq:BootstrappedInnovations6b} is at most
\begin{align*}
			&\E{ \phi(h)^{-1} |\p( \|z+\epsilon_t+a-x\|\le hs)|_{a=A_{n,t}} - \p( \|z+\epsilon_t-x\| \le hs) | \1{t\in\cI_n} } \\
			&\quad+ 4 \phi(h)^{-1} \beta(\epsilon_t,(A_{n,t},\1{t\in \cI_n}))\\
 &\le \E{ [ (\tilde{L} \|A_{n,t}\|^\alpha ) \wedge \tilde{B} ] \1{t\in\cI_n} }  + 4 \phi(h)^{-1} \beta(\epsilon_t,(A_{n,t},\1{t\in \cI_n})).
\end{align*}
The factor inside the expectation is $\op(1)$ and is uniformly bounded. Hence its expectation is $o(1)$. 

Moreover, $\phi(h)^{-1} \beta(\epsilon_t,(A_{n,t},\1{t\in \cI_n}))$ is $o(1)$ by assumption

Finally, consider \eqref{Eq:BootstrappedInnovations6c}. We use the independence between $X_{t-1}$ and $\epsilon_t$ and the fact that $\sup_{z\in\cH} \p( \|\epsilon_t-z\|\le hs ) \le \p( \|\epsilon_t-z_0\|\le hs)$ for a certain $z_0\in\cH$ and $h$ sufficiently small. Then, \eqref{Eq:BootstrappedInnovations6c} is $o(1)$. So, \eqref{Eq:BootstrappedInnovations6} is $\op(1)$.

The last two summands \eqref{Eq:BootstrappedInnovations7} and \eqref{Eq:BootstrappedInnovations8} are also vanishing. Consider \eqref{Eq:BootstrappedInnovations7} which is at most
\begin{align*}
			 &(n\phi(h))^{-1} n^{-1} \sum_{t,\ell=1}^n \1{ \| \hat\Psi_{n,b}(X^*_{t-1}) + \epsilon'_\ell -x \|\le h, \hat\Psi_{n,b} (X^*_{t-1}) \notin  \cU(r_n)  }
\end{align*}
Consider the indicator which is at most
\begin{align}
			\begin{split}
			&\1{ \hat\Psi_{n,b} (X^*_{t-1}) \notin  \cU(r_n)  } \1{X^*_{t-1}\notin \cU(r_n) } \\
			&\qquad\qquad \quad + \1{ \hat\Psi_{n,b} (X^*_{t-1}) \notin  \cU(r_n)  } \1{X^*_{t-1}\in \cU(r_n) }. \label{Eq:BootstrappedInnovations8b}
\end{split}
\end{align}
The first term in \eqref{Eq:BootstrappedInnovations8b} is ultimately 0 $a.s.$: if $X^*_{t-1}\notin \cU(r_n)$, then $\hat\Psi_{n,b} (X^*_{t-1}) = \bar{X}_n$. Firstly, $\tiom_{\bar{X}_n} \le c^*$ $a.s.$ for some $c^*\in\R_+$. Secondly, $\bar{X}_n$ is $a.s.$ bounded w.r.t.\ $\|\cdot\|$, consequently, also w.r.t.\ $\|\cdot\|_{\infty}$. 

Moreover, the second term in \eqref{Eq:BootstrappedInnovations8b} is also ultimately 0 $a.s.$ Indeed, on the one hand, as $X^*_{t-1}\in\cU(r_n)$
\begin{align} \label{Eq:BootstrappedInnovations8c1}
			r_n \ge \| X^*_{t-1}\|_{\infty} \ge \nu(\cD)^{-1/2} \| X^*_{t-1}\|.
\end{align}
On the other hand, and as $\hat{\Psi}_{n,b}(X^*_{t-1})\notin\cU(r_n)$ and as $X^*_{t-1}\in\cU(r_n)$
\begin{align} 
			r_n &< \| \hat\Psi_{n,b}(X^*_{t-1}) \|_{\infty} + \tiom_{ \hat\Psi_{n,b}(X^*_{t-1})} \le C \|  \hat\Psi_{n,b}(X^*_{t-1}) \|^{2/(2+d)} \; \tiom_{ \hat\Psi_{n,b}(X^*_{t-1})}^{d/(2+d)} + \tiom_{ \hat\Psi_{n,b}(X^*_{t-1})} \nonumber\\
			&= \cO( \| X^*_{t-1} \|^{2/(2+d)} ) + \cO(1) + \op(1),\label{Eq:BootstrappedInnovations8c2}
\end{align}
where we use for the derivation of the last equality that
\[
		\|  \hat\Psi_{n,b}(X^*_{t-1}) \| \le \sup_{x\in\cU(r_n)} \| \hat\Psi_{n,b}(x)-\Psi(x)\| + \|\Psi(	X^*_{t-1}) - \Psi(0)\| + \|\Psi(0)\| = \cO(\|	X^*_{t-1}\| ) + \cO(1) + \op(1).
\]
Now, as $2/(2+d)<1$, we can combine \eqref{Eq:BootstrappedInnovations8c1} with \eqref{Eq:BootstrappedInnovations8c2} to obtain the claim.

The summand in \eqref{Eq:BootstrappedInnovations8} is also $\op(1)$. Indeed, by Lemma~\ref{BootstrapCondProb}
\begin{align*}
			&(n\phi(h))^{-1} \sum_{t=1}^n  \p^*\left( \norm{z+  \tilde\epsilon_{0} -{x} }  \le h  \right) \Big|_{z=\hatPsib(X^*_{t-1})} \1{\hatPsib(X^*_{t-1})\notin \cU(r_n) } \\
			&\le \phi(h)^{-1} \p^*\left( \norm{  \tilde\epsilon_{0} - z_0 }  \le h  \right) \, n^{-1}\sum_{t=1}^n  \1{\hatPsib(X^*_{t-1})\notin \cU(r_n) } = \Op\Big(  n^{-1}\sum_{t=1}^n  \mathds{1}\{\hatPsib(X^*_{t-1})\notin \cU(r_n) \} \Big ) = \op(1).
\end{align*}
This completes the proof.
\end{proof}

\begin{lemma}\label{BootstrapHatF} Let $j\ge 1$. Then
\begin{itemize}\setlength\itemsep{0em}
\item [(i)] $(n\phi(h))^{-1} \sum_{t=1}^n \Ec{ (\Delta^*_{h,t}(x))^j } = M_j f_X(x) + \op(1) $.
\item [(ii)] $ (n\phi(h))^{-1} \sum_{t=1}^n \Ec{ h^{-1} \| X^*_t - x \| \Delta^*_{h,t}(x)  } = M_0 f_X(x) + \op(1)$.
\end{itemize}
Moreover,
\begin{itemize}\setlength\itemsep{0em}
\item [(iii)] $(n\phi(h))^{-1} \sum_{t=1}^n (\Delta^*_{h,t}(x))^j = (n\phi(h))^{-1} \sum_{t=1}^n \Ec{ (\Delta^*_{h,t}(x))^j  \,\big|\, \cF^*_{t-1} } + \op(1)$\\
and $(n\phi(h))^{-1} \sum_{t=1}^n \Ec{ (\Delta^*_{h,t}(x))^j  \,\big|\, \cF^*_{t-1} }= M_j f_X(x) + \op(1) $.
\item [(iv)] $(n\phi(h))^{-1} \sum_{t=1}^n \| X^*_t - x \|\, h^{-1} \Delta^*_{h,t}(x) = (n\phi(h))^{-1} \sum_{t=1}^n \Ec{  \| X^*_t - x \|\, h^{-1} \Delta^*_{h,t}(x)  \,\big|\, \cF^*_{t-1} } + \op(1)$ \\
and $(n\phi(h))^{-1} \sum_{t=1}^n \Ec{  \| X^*_t - x \|\, h^{-1}\Delta^*_{h,t}(x)  \,\big|\, \cF^*_{t-1} } = M_0 f_X(x) +\op(1)$.
\end{itemize}
In particular, $\hat{f}_h^*(x) \rightarrow M_1 f_X(x)$ in probability.
\end{lemma}
\begin{proof}
Let $j\ge 1$ and $\ell \in \{0,1\}$. Set $M^*=K(1)^j-\int_0^1 (K(s)^j s^\ell)' \tau_0(s) \intd{s}$. Throughout the proof we will use the following fundamental decomposition
\begin{align}
			&(n\phi(h))^{-1} \sum_{t=1}^n \Ec{ (\Delta^*_{h,t}(x))^j (\| X^*_t - x \|h^{-1})^\ell | \cF^*_{t-1} } \nonumber\\
			\begin{split}\label{Eq:BootstrapHatF1}
			&=  K(1)^j \Big\{ (n\phi(h))^{-1} \sum_{t=1}^n \p^*(\norm{X^*_t-x}\le h|\cF^*_{t-1}) \Big\} \\
			&\quad - \int_0^1 (K(s)^j s^\ell)' \Big\{ (n\phi(h))^{-1} \sum_{t=1}^n \p^*(\norm{X^*_t-x}\le hs |\cF^*_{t-1} ) \Big\} \, \intd{s}. 
\end{split}\end{align}
In the first step, we prove (i) and (ii). Therefore it suffices to apply the result from \eqref{Eq:BootstrappedInnovations0b} in Proposition~\ref{BootstrappedInnovations} to \eqref{Eq:BootstrapHatF1} when additionally conditioned on the sample. We obtain that
\begin{align}
			&(n\phi(h))^{-1} \sum_{t=1}^n \Ec{ \Ec{ (\Delta^*_{h,t}(x))^j (\| X^*_t - x \|h^{-1})^\ell | \cF^*_{t-1} } } \nonumber \\
			\begin{split}\label{Eq:BootstrapHatF2}
			&= K(1)^j\, \p^*(\| \tilde{X}_t -x \| \le h) \phi(h)^{-1} - \int_0^1 [(K(s)^j s^\ell)'  ] [\p^*(\| \tilde{X}_t -x \| \le hs) \phi(h)^{-1} ] \, \intd{s} + \op(1) \\
			&= \fx(x) M^* + \op(1).
\end{split}\end{align}
The first statements in (iii) and (iv) follow from the second statements in (iii) and (iv) because
\[
			(n\phi(h))^{-1} \sum_{t=1}^n (\Delta^*_{h,t}(x))^j (\| X^*_t - x \|h^{-1})^\ell = (n\phi(h))^{-1} \sum_{t=1}^n \Ec{ (\Delta^*_{h,t}(x))^j (\| X^*_t - x \|h^{-1})^\ell | \cF^*_{t-1} } + o_p(1),
\]
which follows from (i) resp.\ (ii) when considering the conditional variances. Indeed, define
\[
		Z_n \coloneqq (n\phi(h))^{-1} \sum_{t=1}^n (\| X^*_t - x \|\, h^{-1})^\ell (\Delta^*_{h,t}(x))^j - \Ec{ (\| X^*_t - x \|\, h^{-1})^\ell (\Delta^*_{h,t}(x))^j \, \big|\, \cF^*_{t-1}  }.	
\]
Then, $\Ec{ Z_n^2 } = \Op( (n\phi(h))^{-1} ) = \op(1)$. Apply Markov's inequality, the (conditional) Jensen inequality and Lebesgue's dominated convergence theorem to obtain for $\veps>0$
\begin{align}\label{Eq:BootstrapHatF3}
		\p( |Z_n|^2 > \veps ) = \E{ \Ec{ \1{ |Z_n|^2 > \veps} \wedge 1 }} \le \E{ (\veps^{-1}\, \Ec{ |Z_n|^2 }) \wedge 1 } \rightarrow 0.
\end{align}
This shows the first statements in (iii) and (iv). Finally, the second statements in (iii) and (iv) are then obtained from \eqref{Eq:BootstrapHatF1} in the same way as \eqref{Eq:BootstrapHatF2} but this time using \eqref{Eq:BootstrappedInnovations0} instead of \eqref{Eq:BootstrappedInnovations0b} from Proposition~\ref{BootstrappedInnovations}.
\end{proof}

\begin{lemma}\label{ConvergenceVariance}
Let $\fG(0,\tilde\cC_x )$ be a Gaussian distribution on $\cH$, where the covariance operator is characterized by the condition $\langle\tilde\cC_x v ,v\rangle = M_2 f_X(x) \E{ \langle\epsilon_0,v\rangle^2 } $ for all $v\in\cH$. Then 
\begin{align}\begin{split}\label{ConvergenceNormalEq0a}
		&\cL( \sqrt{n \phi(h)} \,  V_g(x) ) \rightarrow  \fG(0,\tilde\cC_x),\\
		&\cL^*( \sqrt{ n \phi(h)} \, V^*_g(x) ) \rightarrow \fG(0,\tilde\cC_x) \text{ in probability.}
\end{split}\end{align}
\end{lemma}
\begin{proof}
We prove the claim for the bootstrapped process, the proof of the part which concerns the original process follows similarly. Set $\xi_{n,t} = (n\phi(h))^{-1/2} \Delta^*_{h,t}(x) \epsilon^*_{t+1}$. It suffices to verify in the bootstrap world three conditions from a version of the central limit theorem for $\cH$-valued martingale difference arrays given in \cite{kundu2000central}. These are
	\begin{condition}[Lindeberg condition]\label{C:Lindeberg}\
	\begin{itemize}\setlength\itemsep{0em}

	\item [\mylabel{C:Lindeberg1}{(i)}] $\lim_{n\rightarrow \infty} \sum_{t=1}^n \Ec{ \scalar{\xi_{n,t}}{v}^2 \,\big|\, \cF^*_{t} } = M_2 f_X(x) \E{ \scalar{\epsilon_0}{v}^2} $, for every $v\in\cH$ in probability. 
		
	\item [\mylabel{C:Lindeberg2}{(ii)}] $\lim_{n\rightarrow \infty} \sum_{k\in\N } \sum_{t=1}^n \Ec{ \scalar{\xi_{n,t}}{e_k}^2 } = \sum_{k\in\N} M_2 f_X(x) \E{ \scalar{\epsilon_0}{e_k}^2} = M_2 f_X(x) \E{ \norm{\epsilon_0}^2} < \infty$. 
		
	\item [\mylabel{C:Lindeberg3}{(iii)}]
		$
				\sum_{t=1}^n \Ec{ \scalar{\xi_{n,t}}{e_k}^2 \1{|\! \scalar{\xi_{n,t}}{e_k}| > \rho } \,\big|\, \cF^*_t } = o_p(1),
				$
				for every $\rho>0$ and every $k\ge 1$.
	\end{itemize}
	\end{condition}
We begin with Condition~\ref{C:Lindeberg}~\ref{C:Lindeberg1}.
\begin{align*}
		\sum_{i=1}^n \Ec{ \scalar{\xi_{n,t}}{v}^2 \,\big|\, \cF^*_{t} } &=  \Ec{ \scalar{ \epsilon^*_{0}}{v}^2 } (n \phi(h))^{-1}\sum_{t=1}^n (\Delta^*_{h,t}(x) )^2.
\end{align*}	
This expression converges to $ \mathbb{E}[\scalar{\epsilon_0}{v}^2] M_2 f_X(x)$ in probability by Lemma~\ref{BootstrapHatF}. and by Theorem~\ref{ConsistencyFARBoot}. Similarly, Condition~\ref{C:Lindeberg}~\ref{C:Lindeberg2} is satisfied because
\begin{align*}
		\sum_{k=1}^n \sum_{t=1}^n \Ec{\scalar{\xi_{n,t}}{e_k}^2 } &= \Ec{ \norm{ \epsilon^*_0 }^2 } (n\phi(h))^{-1} \sum_{t=1}^n \Ec{ (\Delta^*_{h,t}(x))^2 } = \Ec{ \norm{ \tilde\epsilon_0 }^2 } M_2 f_X(x) + \op(1),
\end{align*}
where we make again use of Theorem~\ref{ConsistencyFARBoot} and Lemma~\ref{BootstrapHatF}.

We end with Condition~\ref{C:Lindeberg}~\ref{C:Lindeberg3}. Observe that for $a= (2+\delta')/2>1$, $\delta'>0$ and $b$ such that $a^{-1} + b^{-1}=1$
\begin{align*}
			&\sum_{t=1}^n \Ec{ \scalar{\xi_{n,t}}{e_k}^2 \1{|\! \scalar{\xi_{n,t}}{e_k}| > \rho } \!\,\big|\, \cF^*_t }\\
			&\le \rho^{-2a/b} \sum_{t=1}^n \Ec{ \scalar{\xi_{n,t}}{e_k}^{2a}  \,\big|\, \cF^*_t } \\
			&\le \rho^{-2a/b} (n\phi(h))^{-(a-1)} \Ec{\scalar{\epsilon^*_{1}}{e_k}^{2a} } \Big\{ (n\phi(h))^{-1} \sum_{t=1}^n (\Delta^*_{h,t}(x))^{2a} \Big\}.
\end{align*}
Now, $(n\phi(h))^{-1} \sum_{t=1}^n (\Delta^*_{h,t}(x))^{2a} = \cO_p(1)$ by Lemma~\ref{BootstrapHatF} and $ \mathbb{E}^*[\scalar{\epsilon^*_{1}}{e_k}^{2a} ] = \Op(1)$ by Theorem~\ref{ConsistencyFARBoot}. Consequently, Condition~\ref{C:Lindeberg}~\ref{C:Lindeberg3} is also satisfied.
\end{proof}

\begin{lemma}\label{ConvergenceBias}
Let $\bar{B}(x) = L^* \,M_{0} \fx(x)\, \{ \sum_{k\in\N} \intd{\psi}_k(x).\intd{\tilde\varphi_x(0)} e_k \} \in\cH$, where $L^*$ is defined in \ref{C:Bandwidth}. Then
\begin{align*}
		&\lim_{n\rightarrow\infty} \norm{  \sqrt{n \phi(h)}  B_g(x)  - \bar{B}(x) } = 0,\\
		&\lim_{n\rightarrow\infty} \norm{  \sqrt{n \phi(h)}  B^*_g(x)  - \bar{B}(x) } = 0 \text{ in probability.}
\end{align*}
\end{lemma}
\begin{proof}
First consider the original process $X$. Using the conditions on the differentiability of the coordinate functions of $\Psi$ at $x$ and the properties of the function $\tilde\varphi_x$ near 0, we obtain
\begin{align}
		&(n\phi(h))^{-1/2} \sum_{t=1}^n \E{ \Delta_{h,t}(x) (\Psi(X_t)-\Psi(x)) } \nonumber \\
		&= (n\phi(h))^{1/2} h\, \E{\frac{\Delta_{h,t}(x)}{\phi(h)} \frac{ \| X_t-x \|}{h}} \left( \sum_{k\in\N} ( \intd{\psi_k}(x). \intd{\tilde\varphi_x}(0) ) . e_k \right) + o(1) \nonumber\\
			&= L^*\, M_0 \fx(x) \left( \sum_{k\in\N} ( \intd{\psi_k}(x). \intd{\tilde\varphi_x}(0) ) . e_k \right) + o(1).\label{BiasConvEq1a} 
\end{align}

Next, we consider the bootstrapped time series $X^*$ and perform the following split
\begin{align}
		&(n \phi(h))^{-1/2} \sum_{t=1}^n \Delta^*_{h,t}(x) \left( \hat\Psi_{n,b}(X^*_t)-\hat\Psi_{n,b}(x)\right)  \nonumber \\
		&=	(n \phi(h))^{-1/2}  \sum_{t=1}^n \Delta^*_{h,t}(x)  \left\{ \hat\Psi_{n,b}(X^*_t)-\hat\Psi_{n,b}(x)  -\Psi(X^*_t) + \Psi(x)\right\} 	 \label{BiasConvEq1} \\ 
			&\quad + (n \phi(h))^{-1/2}  \sum_{t=1}^n \Delta^*_{h,t}(x) \left( \Psi(X^*_t)-\Psi(x)\right)  -  \Ec{ \Delta^*_{h,t}(x) \left(\Psi(X^*_t)-\Psi(x)\right) \big| \cF^*_{t-1} }  \label{BiasConvEq2} \\		
		& \quad + (n \phi(h))^{-1/2}  \sum_{t=1}^n \Ec{ \Delta^*_{h,t}(x) \left(\Psi(X^*_t)-\Psi(x)\right) \big| \cF^*_{t-1} }.  \label{BiasConvEq3}
\end{align}
Then \eqref{BiasConvEq1} is in the $\|\cdot\|$-norm at most
\begin{align}
\begin{split} \label{BiasConvEq4} 
			&(n \phi(h))^{-1/2}  \sum_{t=1}^n \Delta^*_{h,t}(x) \1{ X^*_t\notin \cU(r_n) } \left\{	\norm{ \bar{X}_n } + \|\hat\Psi_{n,b}(x) - \Psi(x) \| + \norm{\Psi(x)} + L_\Psi h  \right\} \end{split} \\
			\begin{split}	\label{BiasConvEq5}
			&\quad+ (n \phi(h))^{-1/2}  \sum_{t=1}^n \Delta^*_{h,t}(x) \1{ X^*_t\in \cU(r_n) } \|\hatPsib(X^*_t)-\hatPsib(x)  -\Psi(X^*_t) + \Psi(x) \|. \end{split}
\end{align}
We study \eqref{BiasConvEq4}, the last factor is $\Op(1)$. Note that $\1{ X^*_t\notin \cU(r_n) , \norm{X^*_t-x}\le h } \le \1{ r_n \le c }$, for some $c\in\R_+$ which is non-random and does neither depend on $n$ nor on $t$ by Lemma~\ref{BootstrapCondProb}. In particular, \eqref{BiasConvEq4} is $o_p(1)$.

It remains \eqref{BiasConvEq5}; using the definition of the random sets $\cV_n$ and Condition~\ref{C:Covering}, it is enough to show that
\begin{align*}
		&(n\phi(h))^{-1/2} \sum_{t=1}^n \Delta^*_{h,t}(x) \times \Big( \sup_{y\in \cV_n\cap \cU(r_n)} \| \hatPsib(y)-\Psi(y) - (\hatPsib(x)-\Psi(x)) \| \1{ \cV_n \text{ satisfies ($\ast$) at $n$} } \Big) = \op(1).
\end{align*}
We infer from Lemma~\ref{L:UniformConvergenceInNeighborhood} that the last factor is $\oas( (n\phi(h))^{-1/2} )$. Moreover, as $(n\phi(h))^{-1} \sum_{t=1}^n \Delta^*_{h,t}(x) = \Op(1)$, consequently, \eqref{BiasConvEq5} is $o_p( (n\phi(h))^{-1/2} )$. 

The conditional expectation of \eqref{BiasConvEq2} is zero. Moreover, 
\begin{align*}
		\Ec{ \Big \| (n \phi(h))^{-1/2}  \sum_{t=1}^n \Delta^*_{h,t}(x) \left(\Psi(X^*_t)-\Psi(x)\right)  -  \Ec{ \Delta^*_{h,t}(x) \left(\Psi(X^*_t)-\Psi(x)\right) \big| \cF^*_{t-1} } \Big\|^2	} = \Op(h^2),
\end{align*}
using the Lipschitz continuity of $\Psi$ and Lemma~\ref{BootstrapHatF}. Arguing as in \eqref{Eq:BootstrapHatF3}, this shows that also \eqref{BiasConvEq2} is $o_p(1)$.

Finally, consider the summand in \eqref{BiasConvEq3} which decomposes as
\begin{align}\begin{split}\label{BiasConvEq6}
			&(n\phi(h))^{-1/2} \sum_{t=1}^n \Ec{ K(h^{-1} \|z+\epsilon^*_t-x\| ) \bigg\{ \sum_{k\in\N} \intd{\psi}_k(x).(z+\epsilon^*_t-x) e_k \bigg\}}\Bigg|_{z=\hatPsib(X^*_{t-1}) }+ \Op( h^{\alpha} ).
\end{split}\end{align}
Next, we show that we can replace the residuals $\epsilon^*_t$ in \eqref{BiasConvEq6} by the $\tilde{\epsilon}_t$. From the computations, which follow \eqref{Eq:BootstrappedInnovations8b}, we have $\{ \hatPsib(X^*_{t-1})\notin\cU(r_n) \} = \emptyset$ if $n$ is larger than a deterministic $n_0\in\N$. Consequently, it is enough to study
\begin{align}
		&\sup_{z\in \cU(r_n) }	\1{\exists y\in\cS: \hat{\Psi}_{n,b}(y) = z}\phi(h)^{-1} \Big\| \mathbb{E}^*\Big[  K(h^{-1} \|z+\epsilon^*_t-x\| )(z+\epsilon^*_t-x) \\
		&\qquad\qquad\qquad-  K(h^{-1} \|z+\tilde\epsilon_t-x\| )(z+\tilde\epsilon_t-x)  \Big] \Big\| \nonumber \\
		\begin{split}\label{BiasConvEq6a}
		&\le\sup_{z\in \cU(r_n) }	\1{\exists y\in\cS: \hat{\Psi}_{n,b}(y) = z} (n\phi(h))^{-1}\Big\| \sum_{t=1}^n  K(h^{-1} \|z+\epsilon'_t-x\| )(z+\epsilon'_t-x) \1{t\in\cI_n}  \\
		&\qquad\qquad\qquad -  \E{ K(h^{-1} \|z+\epsilon'_t-x\| )(z+\epsilon'_t-x) \1{t\in\cI_n} } \Big\| 
		\end{split} \\
		\begin{split}
		&\quad+\sup_{z\in \cU(r_n) }	(n\phi(h))^{-1} \sum_{t=1}^n \Big\| \E{ K(h^{-1} \|z+\epsilon'_t-x\| )(z+\epsilon'_t-x) \1{t\in\cI_n} }  \\
		&\qquad\qquad\qquad - \E{ K(h^{-1} \|z+\epsilon_t-x\| )(z+\epsilon_t-x) \1{t\in\cI_n} }  \Big\|  \label{BiasConvEq6b} \end{split}\\
		\begin{split}
				&\quad+ \sup_{z\in \cU(r_n) }	\phi(h)^{-1} \Big\| \E{ K(h^{-1} \|z+\epsilon_t-x\| )(z+\epsilon_t-x) \1{t\in\cI_n} } \\
		&\qquad\qquad\qquad - \E{ K(h^{-1} \|z+\epsilon_t-x\| )(z+\epsilon_t-x) }  \Big\|.  \label{BiasConvEq6c} \end{split}
\end{align}
\eqref{BiasConvEq6a} is $\op(h)$ by \eqref{C:UniformConv3} resp.\ Proposition~\ref{P:AsConvergenceResiduals}. \eqref{BiasConvEq6c} is $o(h)$. Indeed, it is at most
\begin{align*}
		&h\, \sup_{z\in \cU(r_n) } \phi(h)^{-1} \E{K(h^{-1} \|z+\epsilon_t-x\| ) \, |1-\1{t\in\cI_n}|} \\
		&= h\, \sup_{z\in \cU(r_n) } \phi(h)^{-1} \E{K(h^{-1} \|z+\epsilon_t-x\| )} \p( X_{t-1}\notin \cU(r_n) ) = o(h),
\end{align*}
where the first equality follows from the independence of $\epsilon_t$ and $X_{t-1}$. So \eqref{BiasConvEq6b} remains. Again as $\epsilon'_t=\epsilon_t + A_{n,t}$, each summand in \eqref{BiasConvEq6b} is at most
\begin{align}
\begin{split}\label{BiasConvEq6d}
		& \sup_{z\in \cU(r_n) }	\phi(h)^{-1} \bigg\| \E{ K(h^{-1} \|z+\epsilon'_t-x\| )(z+\epsilon'_t-x) \1{t\in\cI_n} } \\
		&\qquad\qquad\qquad - \E{ \E{  K(h^{-1} \|z+\epsilon_t+a-x\| )(z+\epsilon_t+a-x) }\Big|_{a=A_{n,t}} \1{t\in\cI_n} } \bigg\| \end{split}\\
		\begin{split}\label{BiasConvEq6e}
		&\quad + \sup_{z\in \cU(r_n) }	\phi(h)^{-1} \bigg\| \E{ \E{  K(h^{-1} \|z+\epsilon_t+a-x\| )(z+\epsilon_t+a-x) }\Big|_{a=A_{n,t}} \1{t\in\cI_n} } \\
		&\qquad\qquad\qquad\qquad - \E{ K(h^{-1} \|z+\epsilon_t-x\| )(z+\epsilon_t-x) \1{t\in\cI_n} }  \bigg\|. \end{split}
\end{align}
Using Lemma~\ref{L:BetaMixing}, \eqref{BiasConvEq6d} is $\cO(\phi(h)^{-1/2} \beta(\epsilon_t,(A_{n,t},\1{t\in \cI_n}))^{1/2}\, h )$ which is $o(h)$. Furthermore, using the functions $\varphi_y$ which model the conditional expectation $\E{ \epsilon_t - y | \|\epsilon_t-y\| }$, we can split \eqref{BiasConvEq6e} as
\begin{align}
\begin{split}\label{BiasConvEq6f}
			&\sup_{z\in \cU(r_n) }	\phi(h)^{-1} \Big\| \mathbb{E}\Big[  \mathbb{E}\Big[  K(h^{-1} \|z+\epsilon_t+a-x\| ) \\
			&\qquad\qquad \times (\varphi_{x-z-a}(\|z+\epsilon_t+a-x\|)-\varphi_{x-z}(\|z+\epsilon_t+a-x\|) ) \Big] \Big|_{a=A_{n,t}} \1{t\in\cI_n} \Big] \Big\| \end{split} \\
			\begin{split}\label{BiasConvEq6g}
				&\quad + \sup_{z\in \cU(r_n) }	\phi(h)^{-1} \Big\| \mathbb{E}\Big[  \mathbb{E}\Big[  K(h^{-1} \|z+\epsilon_t+a-x\| ) \varphi_{x-z}(\|z+\epsilon_t+a-x\|)   \\
				&\qquad\qquad - K(h^{-1} \|z+\epsilon_t-x\| )\varphi_{x-z}(\|z+\epsilon_t-x\|) \Big] \Big|_{a=A_{n,t}} \1{t\in\cI_n} \Big] \Big\|. \end{split}
\end{align}
Consider \eqref{BiasConvEq6f}. We have
\begin{align*}
		&\| \varphi_{x-z-a}(\|z+\epsilon_t+a-x\|)-\varphi_{x-z}(\|z+\epsilon_t+a-x\|) \| \\
		&\le \| \intd{\varphi_{x-z-a}(0)} - \intd{\varphi_{x-z}(0)} \|
		\, \|z+\epsilon_t+a-x\| + R_{x-z-a}( \|z+\epsilon_t+a-x\| ) \\
		&\quad + R_{x-z}( \|z+\epsilon_t+a-x\| )\\
		&= \cO( \|a\|^\alpha \wedge C ) + o( \|z+\epsilon_t+a-x\| ),
\end{align*}
for some constant $C\in\R_+$ and where the $o$-expression is uniform in the arguments $x-z-a$ and $x-z$. Consequently, \eqref{BiasConvEq6f} is $\cO( \E{ (\|A_{n,t}\|^\alpha \wedge C ) \1{t\in\cI_n} } h ) + o(h)$ which is $o(h)$.

Next, we come to \eqref{BiasConvEq6g} and consider the inner expectation. Again, we make use of the expansion of $\varphi_x$ near 0. Then the inner expectation is at most
\begin{align}
			&\phi(h)^{-1} \Big\|	\mathbb{E}\Big[  K(h^{-1} \|z+\epsilon_t+a-x\| ) \Big(\intd{\varphi_{x-z}(0)} \|z+\epsilon_t+a-x\| + R_{x-z}(\|z+\epsilon_t+a-x\|) \Big)  \nonumber \\
				&\qquad\qquad - K(h^{-1} \|z+\epsilon_t-x\| ) \Big(\intd{\varphi_{x-z}(0)} \|z+\epsilon_t-x\| + R_{x-z}(\|z+\epsilon_t-x\|) \Big)  \Big] 		\Big\| \nonumber \\
				\begin{split}\label{BiasConvEq6h}
				&\le \| \intd{\varphi_{x-z}(0)}  \| \, \phi(h)^{-1}  \Big|\mathbb{E}\Big[  K(h^{-1} \|z+\epsilon_t+a-x\| ) (h^{-1}\|z+\epsilon_t+a-x\|) \\
					&\qquad - K(h^{-1} \|z+\epsilon_t-x\| ) (h^{-1}\|z+\epsilon_t-x\|) \Big]			\Big| \,h + o(h). \end{split}
\end{align}
By \eqref{SmallBallLipschitz}, $\phi(h)^{-1} |\p(\|z+\epsilon_t+a-x\|\le hs) - \p(\|z+\epsilon_t-x\|\le hs) | \le \tilde{L} \|a\|^\alpha \wedge \tilde{B}$ uniformly if $h$ is sufficiently small. Using this upper bound in \eqref{BiasConvEq6h} and applying this result afterwards in \eqref{BiasConvEq6g}, we obtain that \eqref{BiasConvEq6g} is $\cO( \mathbb{E}[(\|A_{n,t}\|^\alpha \wedge \tilde{B} ) \1{t\in\cI_n} ] ) + o(h) = o(h)$.

Consequently, we can replace the $\epsilon^*_t$ by $\tilde\epsilon_t$ in \eqref{BiasConvEq6} and obtain
\begin{align}
				&(n\phi(h))^{-1/2} \sum_{t=1}^n \Ec{ K(h^{-1} \|z+\tilde\epsilon_t-x\| ) \bigg\{ \sum_{k\in\N} \intd{\psi}_k(x).(z+\tilde\epsilon_t-x) e_k \bigg\}}\Bigg|_{z=\hatPsib(X^*_{t-1}) } + \op(1) \nonumber\\
				&=(n\phi(h))^{-1/2} \sum_{t=1}^n \Ec{ K(h^{-1} \|z+\tilde\epsilon_t-x\| ) \bigg\{ \sum_{k\in\N} \intd{\psi}_k(x).\varphi_{x-z}(\|z+\tilde\epsilon_t-x\|) e_k \bigg\}}\Bigg|_{z=\hatPsib(X^*_{t-1}) } + \op(1) \nonumber \\
				\begin{split}\label{BiasConvEq7}
				&=h (n\phi(h))^{1/2} \; (n\phi(h))^{-1} \sum_{t=1}^n \Ec{ K(h^{-1} \|z+\tilde\epsilon_t-x\| ) (h^{-1}\|z+\tilde\epsilon_t-x\|) } \\
				&\qquad\qquad\qquad \times \bigg\{ \sum_{k\in\N} \intd{\psi}_k(x).\intd{\varphi_{x-z}(0)} e_k \bigg\}\Bigg|_{z=\hatPsib(X^*_{t-1}) }  + \op(1). 
				\end{split}
\end{align}
We can use the continuity properties of $z\mapsto\intd{\varphi_{x-z}}(0)$ and similar arguments as those used to show that \eqref{Eq:BootstrappedInnovations2} to \eqref{Eq:BootstrappedInnovations4} are $\op(1)$ to find that \eqref{BiasConvEq7} equals
\begin{align}\begin{split}\label{BiasConvEq8}
		&h (n\phi(h))^{1/2} \; (n\phi(h))^{-1} \sum_{t=1}^n \Ec{ K(h^{-1} \|z+\tilde\epsilon_t-x\| ) (h^{-1}\|z+\tilde\epsilon_t-x\|) } \\
				&\qquad\qquad\qquad \times \bigg\{ \sum_{k\in\N} \intd{\psi}_k(x).\intd{\varphi_{x-z}(0)} e_k \bigg\}\Bigg|_{z=\Psi(\tilde{X}_{t-1}) }  + \op(1).
				\end{split}\end{align}
Indeed, one can show that the difference between \eqref{BiasConvEq7} and \eqref{BiasConvEq8} is of order
\[
		n^{-1} \sum_{t=1}^n \|\hat\Psi_{n,b}(X^*_{t-1} - \Psi(\tilde{X}_{t-1}) \|^\alpha = \cO\Big( n^{-1}\sum_{t=1}^n \|X^*_{t-1}-\tilde{X}_{t-1}\|^\alpha \Big) + \cO\Big( \sup_{x\in\cU(r_n)} \|\hat\Psi_{n,b}(x)-\Psi(x)\|^\alpha \Big) + \op(1),
		\]
which is $\op(1)$, we skip the details here. However, \eqref{BiasConvEq8} equals
\begin{align*}
			&(n\phi(h))^{-1/2} \sum_{t=1}^n \Ec{ K(h^{-1} \|z+\tilde\epsilon_t-x\| ) \bigg\{ \sum_{k\in\N} \intd{\psi}_k(x).\varphi_{x-z}(\|z+\tilde\epsilon_t-x\|) e_k \bigg\}}\Bigg|_{z=\Psi(\tilde{X}_{t-1}) } + \op(1) \\
			&= (n\phi(h))^{-1/2} \sum_{t=1}^n \Ec{ K(h^{-1} \|z+\tilde\epsilon_t-x\| ) ( \Psi(z+\tilde\epsilon_t)-\Psi(x)) }\Big|_{z=\Psi(\tilde{X}_{t-1}) } + \op(1) \\
			&= (n\phi(h))^{-1/2} \sum_{t=1}^n \Ec{ K(h^{-1} \|\Psi(\tilde{X}_{t-1})+\tilde\epsilon_t-x\| ) ( \Psi(\Psi(\tilde{X}_{t-1})+\tilde\epsilon_t)-\Psi(x)) } + \op(1) \\
			&= (n\phi(h))^{-1/2} \sum_{t=1}^n \Ec{ K(h^{-1} \|\tilde{X}_t-x\| ) ( \Psi(\tilde{X}_t)-\Psi(x)) } + \op(1).
\end{align*}
And this last line is  $L^*\, M_0 \fx(x) ( \sum_{k\in\N} ( \intd{\psi_k}(x). \intd{\tilde\varphi_x}(0) ) . e_k ) + \op(1)$ as in \eqref{BiasConvEq1a}. 
\end{proof}

We can now give a proof of the main theorem
\begin{proof}[Proof of Theorem~\ref{BootstrapDistribution}]
The claim follows when combining Lemma~\ref{BootstrapHatF}, Lemma~\ref{ConvergenceVariance} and Lemma~\ref{ConvergenceBias}. We have both $\cL(\sqrt{n \phi(h) } (\hatPsi(x)- \Psi(x) )) \rightarrow \fG( \bar{B}(x), \cC_x)$ and $\cL^*(\sqrt{n \phi(h)} (\hatPsiB(x)-\hatPsib(x) )) \rightarrow \fG( \bar{B}(x), \cC_x))$ in probability.
As in the proof of Theorem 2.2 in \cite{krebs2018doublefunctional} we conclude with Theorem 3.1 from \cite{rao1962relations} that 
\[
		\lim_{n\rightarrow \infty}  \left | \int_{\cH} F \intd{\mu^*_{x,n}} - \int_{\cH} F \intd{\mu_{x,n}} \right| =0,
\]
where $F\colon\cH\rightarrow\R$ is Lipschitz continuous and bounded. The amendment concerning the one-dimensional projections follows with P{\'o}lya's theorem (see e.g., \cite{serfling2009approximation}) as in the proof of Corollary 2.3 in \cite{krebs2018doublefunctional}.
\end{proof}

\appendix
\section{Appendix}\label{AppendixA}
\begin{lemma}[Uniform convergence properties of $\hat{f}_h$]\label{UnifConvHatF} Let $\cU\in\fU$. Then for all $j \ge 1$
\begin{itemize}\setlength\itemsep{0em}
\item [(i)] $\sup_{x\in\cU} \left| \left\{(n\phi(h))^{-1} \sum_{t=1}^n \E{ \Delta^j_{h,t}(x) \big| \cF_{t-1} }	\right\}	- M_j \fx(x) \right| = \oas(1)$.
\item [(ii)] $\sup_{x\in\cU} \left| \left\{(n\phi(h))^{-1} \sum_{t=1}^n  \E{ h^{-1} \norm{X_t-x} \Delta_{h,t}(x) \big| \cF_{t-1} }\right\} - M_0 \fx(x)		\right| = \oas(1)$.
\item [(iii)] $\sup_{x\in\cU} \left| \phi(h)^{-1} \E{ \Delta^j_{h,t}(x) }		- M_j \fx(x) \right| = o(1)$.
\item [(iv)] $\sup_{x\in\cU} \left| \phi(h)^{-1} \E{ h^{-1} \norm{X_t-x} \Delta_{h,t}(x)  } - M_0 \fx(x)		\right| = o(1)$.
\end{itemize}
\end{lemma}
\begin{proof}
We fix $j\ge 1$ and an $\ell\in\{0,1\}$. Set $M^* = K(1)^j -\int_0^1 (K(s)^j s^\ell )' \tau_0(s) \intd{s}$. We prove (i) and (ii). (iii) and (iv) follow then from (i) and (ii) by Lebesgue's dominated convergence theorem. First note that
\[
			\phi(h)^{-1} \E{ (h^{-1} \| X_t-x\|)^\ell \Delta^j_{h,t}(x)	| \cF_{t-1} } = K(1)^j \phi(h)^{-1} F^{\cF_{t-1}}_x(h) - \int_0^1 (K^j(s) s^\ell)' \phi(h)^{-1} F^{\cF_{t-1}}_x(hs) \intd{s}.
\]
This implies that
\begin{align*}
			&\sup_{x\in\cU} \left| \left\{(n\phi(h))^{-1} \sum_{t=1}^n  \E{ (h^{-1} \norm{X_t-x})^\ell \Delta^j_{h,t}(x) \big| \cF_{t-1} }\right\} - M^* \fx(x)		\right| \\
			&\le K(1)^j \sup_{x\in\cU} \left\{ \left|n^{-1} \sum_{t=1}^n \fxt(x) - \fx(x) \right| + \left|n^{-1}\sum_{t=1}^n \gxt(h,x) \phi(h)^{-1}  \right|		\right\} \\
			&\quad+ \fx(x)	\left|	\int_0^1 (K^j(s) s^\ell)' \left(\tau_0(s) - \frac{\phi(hs)}{\phi(h)} \right) \intd{s}	\right|	\\
			&\quad+ 	\sup_{x\in\cU} \left|	\int_0^1 (K^j(s) s^\ell)' \left( n^{-1} \sum_{t=1}^n \fxt(x) - \fx(x) \right) \frac{\phi(hs)}{\phi(h)} \intd{s}	\right|	\\
			&\quad+ 		\sup_{x\in\cU} \left|	\int_0^1 (K^j(s) s^\ell)' \left( n^{-1} \sum_{t=1}^n \gxt(hs,x) \phi(hs)^{-1} \right) \frac{\phi(hs)}{\phi(h)} \intd{s}	\right|.
\end{align*}
The first line and the third line on the right-hand side are $\oas(1)$ because both 
\[
	\sup_{x\in\cU} \left|n^{-1} \sum_{t=1}^n \fxt(x) - \fx(x) \right| = \oas(1)\text{ and } \sup_{x\in\cU} \left|n^{-1} \sum_{t=1}^n \gxt(h,x) \phi(h)^{-1}\right| = \oas(1).
	\]
Moreover,	$\sup_{x\in \cU} n^{-1}\sum_{t=1}^n |\gxt(hs,x)| \phi(hs)^{-1} \le D(z_0) <\infty$ for all $s\ge 0$. Thus, the fourth line is $\oas(1)$, with an application of Lebesgue's dominated convergence theorem to the measure $\intd{s}$. Moreover, the second term is $o(1)$ because $|\phi(hs)/\phi(h)-\tau_0(s)| \rightarrow 0$ as $h\downarrow 0$ uniformly in $s\in [0,1]$.
\end{proof}

The first result generalizes the theorem of \cite{laib2011rates} to an $\cH$-valued regression operator.
\begin{theorem}[Uniform consistency of $\hat{\Psi}_{n,h}$]\label{UnifConsistencyHatPsi}
Let $h>0$ be sufficiently small such that Condition~\ref{C:SmallBallProbability} is satisfied. Let $\veps_n = h \lambda_n$ for a positive sequence $(\lambda_n:n\in\N)$ with limit zero. Consider a totally bounded set $\cU\subseteq\cH$ such that $\inf_{x\in\cU} \fx(x) >0$ and assume that
\[
			\sum_{n\in\N} \exp[ -\lambda_n^2 \cO(n \phi(h) (\log n)^{-1} ) ] \log n < \infty \text{ and } \frac{\log N(\cU,\veps_n, \norm{\cdot} ) \log n }{n \phi(h) \lambda_n^2 } \rightarrow 0.
\]
Then $\sup_{x\in\cU} \| \hat{\Psi}_{n,h} (x) - \Psi(x) \| \le a_n + B_n$, where $a_n$ is real-valued, and $B_n$ is a random variable such that $a_n = \cO(h)$ and $B_n = \cO_{a.c.}(\lambda_n)$. In particular, $\sup_{x\in\cU} \| \hat{\Psi}_{n,h} (x) - \Psi(x) \| = \cO_{a.c.}(\lambda_n)$ if $\limsup_{n\rightarrow\infty} h/\lambda_n < \infty$.
\end{theorem}
\begin{proof}[Proof of Theorem~\ref{UnifConsistencyHatPsi}]
Write $\cF_t = \sigma( X_s, 0\le s\le t)$. We begin with a useful lemma
\begin{lemma}\label{L:AppliedMDS}
Let $k\in\N$ and $z\in\R_+$. Then for all $x\in\cU$, $c \in\R_+$, $h>0$ sufficiently small and $n$ there are $A,B\in\R_+$ which do not depend on $n,h$ and $c$ such that
\[
			\p\left( (n\phi(h))^{-1} \left|\sum_{t=1}^n X_{k,t+1} \Delta_{h,t}(x) - \E{X_{k,t+1} \Delta_{h,t}(x)|\cF_{t-1} } \right| \ge c \right) \le 2 \exp\left( -\frac{1}{2}\frac{n\phi(h) c^2}{A+B c} \right).
\]
\end{lemma}
\begin{proof}
The proof relies on an application of Lemma~\ref{ExpIneqMDS}. We compute for $m\in\N$ the moments
\begin{align*}
		\phi(h)^{-m} \E{ |X_{k,t+1}|^m |\Delta_{h,t}(x)|^m |\cF_{t-1} } &\le C \tilde{H}^{m-2} m! \phi(h)^{-(m-1)},
\end{align*}
for some $C,\tilde{H}\in\R_+$. Note that we use $\E{\norm{\epsilon_t}^m}+(L_\Psi h + \norm{\Psi(x)})^m = \cO( \tilde{H}^{m-2} m!)$ and $\phi(h)^{-1} \p( \|X_t-x\|\le h|\cF_{t-1}) \le 2 D(z_0)$ by assumption.
\end{proof}

Next, we introduce the following informal notation, write
\[
			\Et{ \hat{g}_h(x) } = (n\phi(h))^{-1} \sum_{t=1}^n \E{ X_{t+1} \Delta_{h,t}(x) | \cF_{t-1}} \text{ and } \Et{ \hat{f}_h(x) } = (n\phi(h))^{-1} \sum_{t=1}^n \E{ \Delta_{h,t}(x) | \cF_{t-1}}. 
\]
We follow \cite{collomb1977estimation} and consider the difference $\hat{\Psi}_{n,h}(x) - \Psi(x)$ on $\cU$:
\begin{align*}
\hat{\Psi}_{n,h}(x) - \Psi(x)  &= (\hat{f}_h(x))^{-1} \Big\{	\left( \hat{g}_h(x) - \Et{\hat{g}_h(x) } \right) - \Psi(x)\left( \hat{f}_h(x) - \Et{\hat{f}_h(x) } \right) 	\\
&\quad +\left( \Et{ \hat{g}_h(x) } - \Psi(x)\Et{\hat{f}_h(x)} \right) \Big\} .
	\end{align*}
	Thus,
	\begin{align}\begin{split}\label{UnifConsistencyHatPsiEq1}
	\sup_{x\in\cU} \norm{ \hat{\Psi}_{n,h}(x) - \Psi(x) }
&\le \Bigg\{	\sup_{x\in\cU} \norm{ \hat{g}_h(x) - \Et{\hat{g}_h(x) } }  +  \sup_{x\in\cU} \norm{\Psi(x)} \cdot \sup_{x\in\cG}  \left| \hat{f}_h(x) - \Et{\hat{f}_h(x) }  \right|  \\
	&\qquad\qquad + \sup_{x\in\cU}  \norm{ \Et{ \hat{g}_h(x) } - \Psi(x)\Et{\hat{f}_h(x)} } \Bigg\}  \Biggl/ \inf_{x\in\cU} \hat{f}_h(x).
\end{split}\end{align} 
The third term in the numerator is $\Oas(h)$, this follows from the Lipschitz continuity of the regression operator $\Psi$ and the dependence structure of the FAR(1)-process and Lemma~\ref{UnifConvHatF}.

The denominator in \eqref{UnifConsistencyHatPsiEq1} can be bounded as
\begin{align}
		\inf_{x\in\cU} \hat{f}_h(x) &\ge \inf_{x\in\cU} \Et{ \hat{f}_h(x) } - \sup_{x\in \cU} \left| \hat{f}_h(x) - \Et{\hat{f}_h(x)  } \right| \nonumber \\
		&\ge M_1 \inf_{x\in\cU} f_X(x) - \sup_{x\in \cU} \left| \Et{\hat{f}_h(x)}-M_1 f_X(x) \right| - \sup_{x\in\cU} \left| \hat{f}_h(x) - \Et{ \hat{f}_h(x)} \right|.\label{UnifConsistencyHatPsiEq2}
\end{align}
By assumption, the infimum on the right-hand side of \eqref{UnifConsistencyHatPsiEq2} is positive and the first supremum converges to zero $a.s.$ by Lemma~\ref{UnifConvHatF}. Thus, it is sufficient to consider in the rest of the proof the first two terms in the numerator of \eqref{UnifConsistencyHatPsiEq1}
\[
	\sup_{x\in\cU} \norm{ \hat{g}_h(x) - \Et{\hat{g}_h(x) } } 
	\text{ and }
	 \sup_{x\in\cU} \norm{\Psi(x)}  \cdot \sup_{x\in\cU}  \left| \hat{f}_h(x) - \Et{\hat{f}_h(x) }  \right|
\]
and show that these are $\cO_{a.c.}(\lambda_n)$. We only focus on the first supremum, the second is easier to handle and follows similarly. Consider a covering of $\cU$ with $k_n = N(\cU,\veps_n,\| \cdot \|)$ balls of diameter $\veps_n$ centered at points $z_1,\ldots,z_{k_n}$. Then the first supremum is bounded above by
\begin{align}
			&\sup_{x\in\cU}\, (n\phi(h))^{-1} \norm{ \sum_{t=1}^n X_{t+1} \Delta_{h,t}(x) - \E{  X_{t+1} \Delta_{h,t}(x) | \cF_{t-1} } } \nonumber \\
			&\le \max_{1\le u\le k_n } (n\phi(h))^{-1} \norm{ \sum_{t=1}^n X_{t+1} \Delta_{h,t}( z_u ) - \E{  X_{t+1} \Delta_{h,t}(z_u ) | \cF_{t-1} } } \label{UnifConsistencyHatPsiEq3} \\
			\begin{split}\label{UnifConsistencyHatPsiEq4}
			&\quad+ \max_{1\le u\le k_n } \sup_{x\in U(z_u,\veps_n) \cap \cU }  (n\phi(h))^{-1} \Bigg\| \sum_{t=1}^n X_{t+1} (\Delta_{h,t}(x)- \Delta_{h,t}(z_u) ) \\
			&\qquad\qquad\qquad- \E{  X_{t+1} (\Delta_{h,t}(x) -  \Delta_{h,t}(z_u )) | \cF_{t-1} }   \Bigg\|,
			\end{split}
\end{align}
where $U(z_u,\veps_n)$ is the closed $\veps_n$-neighborhood of $z_u$ in $\cH$.
We choose $m$ proportional to $\log n$ (an exact constant can be derived in the subsequent lines). First consider the tail probability of \eqref{UnifConsistencyHatPsiEq3}
\begin{align}
		&\p\left( \max_{1\le u\le k_n } (n\phi(h))^{-1} \norm{ \sum_{t=1}^n X_{t+1} \Delta_{h,t}( z_u ) - \E{  X_{t+1} \Delta_{h,t}(z_u ) | \cF_{t-1} } } \ge c \right) \nonumber\\
		&\le \p\left( \max_{1\le u\le k_n } (n\phi(h))^{-2} \sum_{k\le m} \Big(\sum_{t=1}^n X_{k,t+1} \Delta_{h,t}( z_u ) - \E{  X_{k,t+1} \Delta_{h,t}(z_u ) | \cF_{t-1} } \Big)^2 \ge c^2 \right) \label{UnifConsistencyHatPsiEq5} \\
		&\quad+  (n\phi(h) c)^{-2}\, \E{ \max_{1\le u\le k_n } \sum_{k > m}  \Big(\sum_{t=1}^n X_{k,t+1} \Delta_{h,t}( z_u ) - \E{  X_{k,t+1} \Delta_{h,t}(z_u ) | \cF_{t-1} } \Big)^2 }. \label{UnifConsistencyHatPsiEq6}
\end{align}
Set $c$ proportional to $\lambda_n$, the exact constant can be derived in the following lines. Then by Lemma~\ref{L:AppliedMDS} there are $A,B\in\R_+$ (which also depend on the constant that governs the growth of $m$ but not on $n$) such that the probability in \eqref{UnifConsistencyHatPsiEq5} is
\begin{align}\begin{split}\label{UnifConsistencyHatPsiEq7}
			&\cO\left\{ k_n  \exp\left( - \frac{ n \phi(h) \lambda_n^2 (\log n)^{-1} }{A + B \lambda_n (\log n)^{-1/2}} \right)  \log n \right\} \\
			&= \cO\left\{ \exp\left[ -\frac{n \phi(h) \lambda_n^2}{ \log n } \left(\frac{1}{A + B \lambda_n (\log n)^{-1/2} } - \frac{\log N(\cU,\veps_n,\| \cdot \|) \log n }{n \phi(h) \lambda_n^2 } \right) \right] \log n\right\}. 
\end{split}\end{align}
By assumption, the last upper bound in \eqref{UnifConsistencyHatPsiEq7} is summable over $n\in \N$ which implies
\[
			\max_{1\le u\le k_n } (n\phi(h))^{-2} \sum_{k\le m} \Big(\sum_{t=1}^n X_{k,t+1} \Delta_{h,t}( z_u ) - \E{  X_{k,t+1} \Delta_{h,t}(z_u ) | \cF_{t-1} } \Big)^2 = \cO_{a.c.}\left(\lambda_n^2 \right).
\]

Moreover, one finds that \eqref{UnifConsistencyHatPsiEq6} is $\cO( (\phi(h)c)^{-2}  \sum_{k>m}\E{ |\epsilon_{k,t}|^2  + |\Psi(X_{k,t-1})|^2 } )$, where we bound $\Delta_{h,t}(z_u) \le K(0)$ uniformly in $u$. Next, use that $\sum_{k>m} \E{ |\epsilon_{k,t}|^2  + |\Psi(X_{k,t-1})|^2 } = \sum_{k>m} \E{ | X_{k,t}|^2 } \le a_0\exp(-a_1 m)$. Consequently, as $c$ is proportional to $\lambda_n$, \eqref{UnifConsistencyHatPsiEq6} is $\cO( (\phi(h) c)^{-2} \exp( - a_1 m) ) = \cO( (\phi(h) \lambda_n)^{-2} n^{-a_3} )$, where $a_3$ can be arbitrarily large. Then obviously,
\[
		\max_{1\le u\le k_n } (n\phi(h))^{-2} \sum_{k > m} \Big(\sum_{t=1}^n X_{k,t+1} \Delta_{h,t}( z_u ) - \E{  X_{k,t+1} \Delta_{h,t}(z_u ) | \cF_{t-1} } \Big)^2 = \cO_{a.c.}(\lambda_n^2).
\]
Consequently, \eqref{UnifConsistencyHatPsiEq3} is $\cO_{a.c.}(\lambda_n)$.

Second we demonstrate that also \eqref{UnifConsistencyHatPsiEq4} attains the rate $\cO(\lambda_n)$. It is at most
\begin{align}\begin{split}\label{UnifConsistencyHatPsiEq8}
		  &\max_{1\le u\le k_n } \sup_{x\in U(z_u,\veps_n) }  (n\phi(h))^{-1} \sum_{t=1}^n \norm{X_{t+1}} |\Delta_{h,t}(x)- \Delta_{h,t}(z_u) | \\
			&\quad + \max_{1\le u\le k_n } \sup_{x\in U(z_u,\veps_n) }  (n\phi(h))^{-1} \sum_{t=1}^n  \E{ \norm{ X_{t+1}} |\Delta_{h,t}(x) -  \Delta_{h,t}(z_u )| | \cF_{t-1} }. 
			\end{split}
\end{align}
We only show that the first summand in \eqref{UnifConsistencyHatPsiEq8} attains the desired rate, the second works similarly. Set
\[
			Z_t(y) =	(n\phi(h))^{-1} \norm{X_{t+1}} \left[ \frac{\veps_n}{h} \norm{K'}_\infty \1{\| X_t-y\|\le h} + K(0) \1{ h-\veps_n < \|X_t-y\|\le h+\veps_n } \right].
\]
Here the first term in the square bracket is relevant if both $\|X_t-x\|\le h$ and $\|X_t-z_u\|\le h$. The second term is relevant if either $\|X_t-x\| \le h$ and $\|X_t-z_u\|>h$ or $\|X_t-x\| > h$ and $\|X_t-z_u\| \le h$. 

Then the first summand in \eqref{UnifConsistencyHatPsiEq8} is bounded above by $\max_{1\le u\le k_n } \sum_{t=1}^n  Z_t(z_u)$ which is $\cO_{a.c.}( \veps_n h^{-1} ) = \cO_{a.c.}(\lambda_n)$. Indeed, consider in the first place the sum of conditional expectations
\begin{align}\begin{split}\label{UnifConsistencyHatPsiEq9}
			\max_{1\le u\le k_n} \sum_{t=1}^n  \E{ Z_t(z_u) | \cF_{t-1} } &= (n\phi(h))^{-1} \max_{1\le u\le k_n} \sum_{t=1}^n \mathbb{E}\Big[ \E{ \norm{X_{t+1}} | X_t } \Big\{ \frac{\veps_n}{h} \norm{K'}_\infty \1{ \|X_t-z_u\|\le h} \\
			&\quad + K(0) \1{ h-\veps_n < \|X_t-z_u\|\le h+\veps_n } \Big\} | \cF_{t-1} \Big].
\end{split}\end{align}
Use that $ \E{ \norm{X_{t+1}} | X_t } $ is bounded on the set $\{z: \|X_t-z\|\le h, z\in \cU \}$ and that 
\[
		\phi(h)^{-1}\p(\|\epsilon_t-x\|\le h) \le \phi(h)^{-1} \p(\|\epsilon_t-z_0\|\le h)
		\]
for all $x\in\cH$ if $h$ is sufficiently small.	Moreover,
\[
	\sup_{y\in\cU} \left| (n\phi(h))^{-1} \sum_{t=1}^n \p( \|X_t-y\|\le h + \veps_n  | \cF_{t-1} ) - \p( \|X_t-y\| \le h - \veps_n  | \cF_{t-1} )  \right| = \cO( \veps_n h^{-1}) = \cO(\lambda_n)
	\]
 because of the Lipschitz continuity of the conditional distribution functions from Condition~\ref{C:SmallBallProbability}~\ref{C:SmallBallProbability1}. This yields that \eqref{UnifConsistencyHatPsiEq9} is at most $c \lambda_n$ for a constant $c\in\R_+$, i.e., \eqref{UnifConsistencyHatPsiEq9} admits a deterministic upper bound. In particular, \eqref{UnifConsistencyHatPsiEq9} is $\cO_{a.c.}(\lambda_n)$. In addition, one can show that
\[
			\p\left(	\max_{1\le u\le k_n}  \left| \sum_{t=1}^n Z_t(z_u) -  \E{ Z_t(z_u) | \cF_{t-1} } \right| \ge c	\right) \le 2 k_n \exp\left( - \frac{1}{2} \frac{n \phi(h) c^2}{A \veps_n h^{-1} + B c} \right),
\]
for certain $A,B\in\R_+$. In particular, as $\veps_n h^{-1} = \lambda_n \rightarrow 0$, we obtain for the choice $c$ (which is proportional to $\lambda_n$) that this last exponential bound is dominated by \eqref{UnifConsistencyHatPsiEq7}. Thus, 
$
			\max_{1\le u\le k_n}  \left| \sum_{t=1}^n Z_t(z_u) -  \E{ Z_t(z_u) | \cF_{t-1} } \right| = \cO_{a.c.}( \lambda_n )
$. 
So the first summand in \eqref{UnifConsistencyHatPsiEq8} is also $\cO_{a.c.}(\lambda_n)$; the same applies to the second summand in \eqref{UnifConsistencyHatPsiEq8}.
\end{proof}

\begin{corollary}\label{C:UnifConsistencyHatPsi}
Let $d=1$ and let $(b_n:n\in\N)$ converge to 0. Then there is a sequence $(r_n:n\in\N)\subseteq\R_+$ with $\lim_{n\rightarrow\infty}r_n=\infty$ such that $\sup_{x\in\cU(r_n)} \|\hat\Psi_{n,b}(x)-\Psi(x)\| = \op(1)$.
\end{corollary}
\begin{proof}
In the first step, we show that the estimator is indeed uniformly consistent on a set of the type $\cU(r)$ for a given sequence $b_n$ which satisfies Condiiton~\ref{C:Bandwidth} as claimed in Theorem~\ref{UnifConsistencyHatPsi}; note that in the present situation, we need to replace the bandwidth $h$ by the bandwidth $b$. Depending on $b$ choose $\veps_n$ such that	$\veps_n^3 = (\log n)^2 b^2	(n\phi(b))^{-1}$. Then
\[
			\frac{\veps_n}{b} = \Big(\frac{(\log n)^2}{b (n\phi(b)) }\Big)^{1/3} = \Big(\frac{(\log n)^2}{(n\phi(b))^{1/2} }\Big)^{1/3} \Big(\frac{1}{b (n\phi(b))^{1/2} }\Big)^{1/3}  \rightarrow 0.
\]		
Moreover, using the fact that $d=1$, $\lambda_n = \veps_n b^{-1}$ and the definition of $\veps_n$, we have
\[
	\frac{\log N(\cU(r),\veps_n,\|\cdot\|) \log n}{n \phi(b) \lambda_n^2 } = \cO\Big( \frac{r \log n b^2}{n \phi(b) \veps_n^3} \Big) = o(1).
\]
And additionally,
\[
		\sum_{n\in\N} \exp\big(	-\lambda_n^2\, \cO(n \phi(b) (\log n)^{-1})	\big) \log n = \sum_{n\in\N} \exp\Big[- \log n\, \cO\Big\{ \Big(\frac{n \phi(b)}{b^2 (\log n)^2}	\Big)^{1/3}	\Big\} \Big] \log n < \infty
\]
because $n\phi(b) (\log n)^{-2} \rightarrow \infty$. Hence, Theorem~\ref{UnifConsistencyHatPsi} holds.

In the second step, let $S\in\R_+$ be arbitrary but fix. Set $a_k = S k$ for $k\in\N_+$. Moreover, let $(f_k:k\in\N)$ and $(g_k:k\in\N)$ be positive sequences converging to 0. Set $m_0 =0$. Then define inductively for $k\in\N_+$
\[
		m_k = \inf \Big\{u\in\N,\;	u > m_{k-1}: \p( \sup_{x\in\cU(a_k)} \| \hat\Psi_{\ell,b_\ell}(x)-\Psi(x)\| > g_k )\le f_k	\quad \forall \ell\ge u	\Big\}.
\]
The definition of $m_k$ is meaningful for each $k$ (given $m_{k-1}$) because of Theorem~\ref{UnifConsistencyHatPsi} and because $g_k$ and $f_k$ are fix for each $k$. Next construct the sequence $(r_n:n\in\N)$ as follows: $r_1=\ldots=r_{m_1-1}=S /2$ and $r_{m_k} = \ldots = r_{m_{k+1}-1} = S k$ ($k\in\N_+$). Note that $\lim_{n\rightarrow \infty} r_n = \infty$ because $(m_k:k\in\N)$ partitions $\N$.

Let now $\epsilon,\delta>0$ be arbitrary. Choose $k_1$ such that $\delta > g_{n}$ and $\epsilon> f_{n}$ for all $n \ge k_1$. Moreover, for each $n$ there is a unique integer $k_2$ such that $n\in \{m_{k_2},\ldots,m_{k_2+1}-1\}$. Choose $n$ large enough such that $k_2 \ge k_1$. Then
\begin{align*}
			\p( \sup_{x\in\cU(r_n)} \| \hat\Psi_{n,b_n}(x)-\Psi(x)\| > \delta ) &\le \p( \sup_{x\in\cU(r_n)} \| \hat\Psi_{n,b_n}(x)-\Psi(x)\| > g_{k_1} ) \\
			&\le \p( \sup_{x\in\cU(S k_2 )} \| \hat\Psi_{n,b_n}(x)-\Psi(x)\| > g_{k_2} ) \le f_{k_2} \le f_{k_1} \le \epsilon. 
\end{align*}
\end{proof}

\begin{proposition}\label{P:AsConvergenceResiduals}
Let $d=1$ and $h=h_n\rightarrow 0$ as in Condition~\ref{C:Bandwidth}. Then for every $r\in\R_+$ 
\begin{align}
\begin{split}\label{E:PConvergenceResiduals1}
			&\sup_{z\in\cU(r)} \sup_{s\in [0,1]} (|\cI_n|\, \phi(h))^{-1} \Big| \sum_{t=1}^n \1{\|z+\epsilon'_t-x\|\le hs}\1{t\in\cI_n} \\
		&\qquad\qquad\qquad\qquad\qquad \qquad\qquad 	-\p(\|z+\epsilon'_t-x\|\le hs, t\in\cI_n) \Big| = \op(1),
				\end{split}\\
				\begin{split}\label{E:PConvergenceResiduals2}
&\sup_{z\in\cU(r)} \1{\exists y\in\cS: \hat{\Psi}_{n,b}(y) = z} (|\cI_n|\,\phi(h))^{-1} \Big\| \sum_{t=1}^n K(h^{-1} \|z+\epsilon'_t-x\| ) (h^{-1} (z+\epsilon'_t-x) ) \1{t\in\cI_n} \\
		&\qquad\qquad\qquad\qquad\qquad - \E{K(h^{-1} \|z+\epsilon'_t-x\| ) (h^{-1} (z+\epsilon'_t-x) ) \1{t\in\cI_n} }  \Big\| = \op(1).
		\end{split}
\end{align}
In particular, there is a sequence $(r'_n:n\in\N)\subseteq\R_+$ with $\lim_{n\rightarrow\infty}r'_n=\infty$ such that \eqref{E:PConvergenceResiduals1} and \eqref{E:PConvergenceResiduals2} are $\op(1)$, when $r$ is replaced by $(r_n:n\in\N)$.
\end{proposition}
\begin{proof}
In the first part, we show the claim for \eqref{E:PConvergenceResiduals1}, \eqref{E:PConvergenceResiduals2} follows similarly, is however more complex as the random variables are $\cH$-valued and is demonstrated in the second part. Then by taking the pairwise minimum in both sequences, we obtain the desired sequence $r_n$.

\textit{Part I:} Clearly, $|\cI_n| n^{-1} \rightarrow 1$ in probability. Furthermore, $\p(\|z+\epsilon'_t-x\|\le hs, t\in\cI_n)$ can be replaced with $\p(\|z+\epsilon'_t-x\|\le hs)\1{t\in\cI_n}$. Indeed, on the one hand, uniformly in $s$ and $z$
\begin{align*}
			&\phi(h)^{-1} |\p(\|z+\epsilon'_t-x\|\le hs, t\in\cI_n) - \p(\|z+\epsilon'_t-x\|\le hs) | \\
			&\le  \phi(h)^{-1} | \p(\|z+\epsilon'_t-x\|\le hs, t\in\cI_n) - \E{ \p(\|z+\epsilon_t+a-x\|\le hs)|_{a=A_{n,t}} \1{t\in\cI_n} } | \\
			&\quad + \phi(h)^{-1} |  \E{ \p(\|z+\epsilon_t+a-x\|\le hs)|_{a=A_{n,t}}\1{t\in\cI_n} - \p(\|z+\epsilon_t+a-x\|\le hs) |_{a=A_{n,t}}  } | \\
			&\quad + \phi(h)^{-1} | \E{  \p(\|z+\epsilon_t+a-x\|\le hs) |_{a=A_{n,t}} } -  \p(\|z+\epsilon'_t-x\|\le hs)  |.
\end{align*}
Using Proposition~2 in \cite{krebs2018large}, the first and the third term on the right-hand side are $\cO(\beta(\epsilon_t,(A_{n,t},\1{t\in \cI_n}))\phi(h)^{-1} )$ which is $o(1)$. The second term is at most $\phi(h)^{-1} \p( \|\epsilon_t-z_0\| \le hs ) (1-\p(t\in\cI_n)) = o(1)$.

And on the other hand,
\begin{align*}
			&\phi(h)^{-1} |\p(\|z+\epsilon'_t-x\|\le hs) - \p(\|z+\epsilon'_t-x\|\le hs)  \1{t\in \cI_n} |\\
			&\le \cO( \1{t \notin \cI_n} ) + \cO\big( \phi(h)^{-1} \beta(\epsilon_t, (A_{n,t},\1{t\in\cI_n}) ) \big)
\end{align*}
uniformly in $z$ and $s$. As $n^{-1} \sum_{t=1}^n \1{t \notin \cI_n}$ is $\op(1)$, this shows that we can perform the replacement. So instead of \eqref{E:PConvergenceResiduals1}, we consider
\begin{align} \label{E:PConvergenceResiduals3b}
			&\sup_{z\in\cU(r)} \sup_{s\in [0,1]} (n \phi(h))^{-1} \Big| \sum_{t=1}^n \big(\1{\|z+\epsilon'_t-x\|\le hs}	-\p(\|z+\epsilon'_t-x\|\le hs) \big) \1{t\in\cI_n} \Big|. 
\end{align}
Let $q>0$, then
\begin{align}
			&\p\Big( ( n \phi(h))^{-1} \big| \sum_{t=1}^n \big(\1{\|z+\epsilon'_t-x\|\le hs }- \p(\|z+\epsilon'_t-x\|\le hs) \big) \1{t\in\cI_n}		\big| > q \Big) \nonumber \\
			\begin{split}
			&\le \p\Big(  ( n \phi(h))^{-1} \big| \sum_{t=1}^n \big(\1{\|z+\epsilon'_t-x\|\le hs }- \p(\|z+\epsilon'_t-x\|\le hs)\big) 	\1{t\in\cI_n} 	\big| > q, \\
			&\qquad\qquad \sup_{x\in \cU(r_n) } \|\Psi(x)-\hat\Psi_{n,b}(x)-\bar\epsilon\| \le \delta \Big)  \label{E:PConvergenceResiduals4} \end{split} \\
			&\quad + \p\Big( \sup_{x\in \cU(r_n) } \|\Psi(x)-\hat\Psi_{n,b}(x)\| > \delta/2 \Big) + \p\Big(  \|\bar\epsilon\| > \delta/2 \Big), \label{E:PConvergenceResiduals5}
\end{align}
where $r_n$ is as in Corollary~\ref{C:UnifConsistencyHatPsi}. By construction, $\|\bar\epsilon\| \le |\cI_n|^{-1} \| \sum_{t\in\cI_n} \epsilon_t \| + \sup_{x\in\cU(r_n)} \| \hat\Psi_{n,b}(x) - \Psi(x) \|$; the first term which is the average of the innovations $\epsilon_t$ is $\op(1)$, this follows from the well-known concentration inequalities. Moreover, by Corollary~\ref{C:UnifConsistencyHatPsi}, the first probability in \eqref{E:PConvergenceResiduals5} converges to 0 for every $\delta>0$. 

In order to obtain a bound on \eqref{E:PConvergenceResiduals4}, we consider the Laplace transform
\begin{align*}
			&\mathbb{E}\Biggl[ \exp\Big( \eta \phi(h)^{-1}  \Big( \sum_{t=1}^n [\1{\|z+\epsilon'_t-x\|\le hs }- \p(\|z+\epsilon'_t-x\|\le hs)]\1{t\in\cI_n} \Big) \Big) \\
			&\qquad\qquad\qquad \times \1{ \sup_{x\in\cU(r_n)} \|\Psi(x)-\hat\Psi_{n,b}(x)-\bar\epsilon \| \le \delta} \Biggl]
\end{align*}
for $\eta,\delta>0$. For that reason, we define the filtration $\cF_t = \sigma( X_1,\ldots,X_t, \Psi(\cdot)-\hat\Psi_{n,b}(\cdot),\bar\epsilon )$ for $t=0,\ldots,n$. Then
\begin{align}
			&\E{ \exp\big( \eta \phi(h)^{-1} [\1{\|z+\epsilon'_t-x\|\le hs }  - \p(\|z+\epsilon'_t-x\|\le hs) ]\1{t\in\cI_n} \big) \,\Big|\, \cF_{t-1} } \nonumber \\
			&\le 1 + \eta \phi(h)^{-1}\Big| \E{ \1{\|z+\epsilon'_t-x\|\le hs } - \p(\|z+\epsilon'_t-x\|\le hs) \,\big|\, \cF_{t-1} } \Big| \1{t\in\cI_n}  \label{E:PConvergenceResiduals6} \\
			&\quad + \sum_{m=2}^\infty \frac{ (\eta \phi(h)^{-1})^m }{m!} \E{ \left| \1{\|z+\epsilon'_t-x\|\le hs }  - \p(\|z+\epsilon'_t-x\|\le hs) \right|^m \,\big|\, \cF_{t-1} }\1{t\in\cI_n}. \label{E:PConvergenceResiduals7}
\end{align}
Using Conditions~\ref{C:SmallBallProbability} and \ref{C:BetaMixing}, it is not difficult to derive that there is a constant $C'\in\R_+$ such that each summand in \eqref{E:PConvergenceResiduals7} is at most
\begin{align*}
		&\frac{ \eta^m \phi(h)^{-m} }{m!} \left\{ \p(\|z+\epsilon'_t-x\| \le hs\,|\, \cF_{t-1} ) + \p(	\|z+\epsilon'_t-x\| \le hs)^m	\right\}  \le  \frac{ C' \eta^m  }{m! \phi(h)^{m-1}},
\end{align*}
uniformly in $n\in\N$. Thus, if $\eta<\phi(h)$, \eqref{E:PConvergenceResiduals7} is at most
\[
			\sum_{m=2}^\infty \frac{ C' \eta^m  }{m! \phi(h)^{m-1}} \le \frac{ C' \eta^2}{\phi(h)} \sum_{m=0}^\infty \eta^m \phi(h)^{-m} \le  \frac{ C' \eta^2}{\phi(h)} \frac{1}{1-\eta/\phi(h)}.
\]
 Next, consider the second term in \eqref{E:PConvergenceResiduals6}. Using the requirement $\sup_{x\in\cU(r_n)} \|\Psi(x)-\hat\Psi_{n,b}(x)-\bar\epsilon \| \le \delta$, we see that $A_{n,t} = \Psi(X_{t-1})-\hat\Psi_{n,b}(X_{t-1})-\bar\epsilon$ has norm of at most $\delta$, if $t\in\cI_n$	. Then
\begin{align}
			& \eta\phi(h)^{-1} \Big| \p(\|z+\epsilon_t+A_{n,t}-x\| \le hs \,|\,\cF_{t-1} ) -  \p(\|z+\epsilon_t+A_{n,t}-x\|\le hs) \Big|  \1{t\in\cI_n} \nonumber \\
			\begin{split}\label{E:PConvergenceResiduals8}
			&\le \eta \phi(h)^{-1} \big|	\p(\|z+\epsilon_t+a-x\| \le hs\,|\,\cF_{t-1} )|_{a=A_{n,t}}	- \E{ \p(\|z+\epsilon_t+a-x\|\le hs)|_{a=A_{n,t}} }  \big| \1{t\in\cI_n}\\
			&\qquad\qquad\quad + 4 \eta \phi(h)^{-1} \beta(\epsilon_t,(A_{n,t},\1{t\in \cI_n})).
			\end{split}
\end{align}
It remains to compute the asymptotic behavior of the first term in \eqref{E:PConvergenceResiduals8}. This term is at most
\begin{align*}
			& \eta  \phi(h)^{-1} \Big| \p(\|z+\epsilon_t+a-x\| \le hs\,|\,\cF_{t-1} )|_{a=A_{n,t}} -  \E{ \p(\|z+\epsilon_t+a-x\|\le hs)|_{a=A_{n,t}} }\Big| \1{t\in\cI_n} \nonumber \\
				&\le \eta L_{n,t} + \eta \phi(h)^{-1}  \Big| \p(\|z+\epsilon_t+a-x\| \le hs) -  \p(\|z+\epsilon_t-x\|\le hs) \Big| \Big|_{a=A_{n,t}}  \1{t\in\cI_n} \nonumber \\
				&\quad +  \eta \phi(h)^{-1} \Big| \p(\|z+\epsilon_t-x\| \le hs) -   \E{ \p(\|z+\epsilon_t+a-x\|\le hs)|_{a=A_{n,t}} } \Big| \1{t\in\cI_n} \\
			&\le  \eta L_{n,t} + \eta \phi(h)^{-1} \Big| \p(\|z+\epsilon_t+a-x\| \le hs) -  \p(\|z+\epsilon_t-x\|\le hs) \Big| \;\Big|_{a=A_{n,t}} \1{t\in\cI_n} \nonumber \\
				&\quad +  \eta \, \E{\phi(h)^{-1} \Big| \p(\|z+\epsilon_t-x\| \le hs) -   \p(\|z+\epsilon_t+a-x\|\le hs)|_{a=A_{n,t}} \Big|} .
\end{align*}
Using the fact that $\|A_{n,t}\|\le \delta$ and the continuity properties of the small ball probability function, one finds that this term is of order $\cO((L_{n,t}+\delta^\alpha + \p( \|A_{n,t}\|>\delta ) ) \eta  )$.

Consequently, applying well-known inequalities, we obtain for the probability in \eqref{E:PConvergenceResiduals4} without considering the normalization by $n^{-1}$
\begin{align}
		 &\p\Big(  \phi(h)^{-1} \big| \sum_{t\in\cI_n} \1{\|z+\epsilon'_t-x\|\le hs } - \p(\|z+\epsilon'_t-x\|\le hs) 		\big| > q, \sup_{x\in \cU(r_n) } \|\Psi(x)-\hat\Psi_{n,b}(x)-\bar\epsilon\| \le \delta \Big) \nonumber \\
			\begin{split}\label{E:PConvergenceResiduals9}
			&\le 2 \exp\Big\{- \eta q + C' \phi(h)^{-1} \frac{\eta^2}{1- \eta \phi(h)^{-1}} n  + c \eta (n^{-1} \sum_{t=1}^n L_{n,t}+\delta^\alpha + p_n(\delta) + o(1) ) n  \Big\}
			\end{split}
\end{align}			
for some $c\in\R_+$ and where $p_n(\delta) = \p( \sup_{x\in\cU(r_n)} \|\Psi(x)-\hat\Psi_{n,b}(x)-\bar\epsilon\|  > \delta )+\p(t\notin \cI_n)$. An admissible choice of $\eta$ is $q (2 C' \phi(h)^{-1} + q \phi(h)^{-1} )^{-1}$ which is less than $\phi(h)$. Using this choice in \eqref{E:PConvergenceResiduals9}, we obtain after normalization
\begin{align}\begin{split}\label{E:PConvergenceResiduals10}
		 &\p\Big(  (n\phi(h))^{-1} \Big| \sum_{t\in\cI_n} \1{\|z+\epsilon'_t-x\|\le hs } - \p(\|z+\epsilon'_t-x\|\le hs) 		\Big| > q,  \\
			&\quad \sup_{x\in \cU(r_n) } \|\Psi(x)-\hat\Psi_{n,b}(x)-\bar\epsilon\| \le \delta \Big) \le 2 \exp\Big\{- \frac{1}{2} \frac{q^2 -2 c q (\delta^\alpha + p_n(\delta) + o(1) ) }{2 C' + q} n\phi(h) \Big\}, 
			\end{split} 
\end{align}
where we use that also $n^{-1} \sum_{t=1}^n L_{n,t} = o(1)$. Clearly, for $q>0$ arbitrary but fixed, there is an $N_0\in\N$ and a $\delta>0$ such that the fraction inside the exponential function in \eqref{E:PConvergenceResiduals10} is positive for all $n\ge N_0$. This shows convergence in probability for fix $z\in \cS$ and $s\in [0,1]$.

Next, we consider \eqref{E:PConvergenceResiduals3b} for a set $\cU(r)$ with fix radius $r$. We argue as in the proof of Theorem~\ref{UnifConsistencyHatPsi} to see that for every fixed $r$, we have also uniform convergence on the parameter space $\cU(r)\times [0,1]$; the covering number of $[0,1]$ is negligible. For this we proceed as follows. Set
\[
		Y_{n,t}(z,s) = \phi(h)^{-1} \left(\1{\|z+\epsilon'_t-x\|\le hs } - \p(\|z+\epsilon'_t-x\|\le hs) \right) \1{t\in\cI_n} .
\]
Consider an $\veps_n$-covering $\{ U(z_i,\veps_n): i=1,\ldots,\kappa_n\}$ of $\cU(r)$ and a $\veps'_n$-covering of $[0,1]$, $\{ B(s_j,\veps'_n): j=1,\ldots,\kappa'_n\}$, where $B(s_j,\veps'_n)$ is the interval $[s_j-\veps'_n,s_j+\veps'_n]$. Note that we can choose the points $s_j$ on an equidistant grid. The radii are as follows
\[
			\veps_n = (n\phi(h))^{-1} \, (\log n)^2 \text{ and } \veps'_n = \veps_n \,h^{-1} = (h (n\phi(h))^{1/2} )^{-1} \, \sqrt{(n\phi(h))^{-1} (\log n)^4 } = o(1)
\]
because $n\phi(h) (\log n)^{-2(2+\alpha)}\rightarrow\infty$ by assumption. We bound the supremum from above as 
\begin{align}\begin{split}\label{E:PConvergenceResiduals11}
		\sup_{\substack{z\in\cU(r)\\ s\in [0,1]}} \Big| n^{-1} \sum_{t=1}^n Y_{n,t}(z,s) \Big| &\le \max_{\substack{i=1,\ldots,\kappa_n\\ j=1,\ldots,\kappa'_n}} \Big| n^{-1} \sum_{t=1}^n Y_{n,t}(z_i,s_j) \Big| \\
		&\quad +   \max_{\substack{i=1,\ldots,\kappa_n\\ j=1,\ldots,\kappa'_n}} \sup_{\substack{z \in U(z_i,\veps_n)\\ s\in B(s_j,\veps'_n)}}  \Big| n^{-1} \sum_{t=1}^n Y_{n,t}(z_i,s_j) - Y_{n,t}(z,s) \Big|. 
\end{split}\end{align}

Consider the first maximum in \eqref{E:PConvergenceResiduals11}. We perform a split as in \eqref{E:PConvergenceResiduals4} and \eqref{E:PConvergenceResiduals5} and use the result from \eqref{E:PConvergenceResiduals10}. Then it suffices to consider 
\begin{align}
			&\p\Biggl(  \max_{\substack{i=1,\ldots,\kappa_n\\ j=1,\ldots,\kappa'_n}} \Big| n^{-1} \sum_{t=1}^n Y_{n,t}(z_i,s_j) \Big| > q , \sup_{x\in \cU(r_n) } \|\Psi(x)-\hat\Psi_{n,b}(x)-\bar\epsilon\| \le \delta \Biggl) \nonumber \\
			&\le 2 \kappa_n \kappa'_n \exp\Big\{- \frac{1}{2} \frac{q^2 -2 c q (\delta^\alpha + p_n(\delta) + o(1) ) }{2C'+ q} n\phi(h) \Big\}.\label{E:PConvergenceResiduals12}
\end{align}
Now $\kappa'_n$ is $\cO( (\veps'_n)^{-1}) = \cO( (n\phi(h))^{1/2} (\log n)^{-2} )$ which grows less than $n$. And as $d=1$, $\log \kappa_n$ is proportional to $\veps_n^{-1}$ which is $\cO( n\phi(h) (\log n)^{-2})$. This shows that for $\delta$ sufficiently small, \eqref{E:PConvergenceResiduals12} is summable over $n\in\N$ for all $q>0$. Consequently, the first term on the right-hand side in \eqref{E:PConvergenceResiduals11} is $\op(1)$.

The second maximum in \eqref{E:PConvergenceResiduals11} is more involved. We need to consider the difference of the indicator functions and the difference of the probabilities separately. Here, we only investigate the difference of the indicator functions in this term, the difference of the probabilities is less complicated and works with the same ideas. We have
\begin{align}
			&|\1{\|z+\epsilon'_t-x\|\le hs } - \1{\|z_i+\epsilon'_t-x\|\le hs_j } |  \nonumber \\
			&\le 2 \1{h(s_j-\veps'_n)-\veps_n \le \|z_i+\epsilon'_t-x\| \le h(s_j+\veps'_n)+\veps_n }, \label{E:PConvergenceResiduals13}
\end{align}
which is independent of $z$ and $s$. The expectation \eqref{E:PConvergenceResiduals13} when multiplied by $\phi(h)^{-1}$ is of order
\begin{align*}
			&\phi(h)^{-1} \p( h(s_j-\veps'_n)-\veps_n \le \|z_i+\epsilon'_t-x\| \le h(s_j+\veps'_n)+\veps_n ) \\
			&= \cO\Big(\frac{\phi(h s_j)}{\phi(h)} \frac{h\veps'_n + \veps_n}{h s_j} \Big) + \cO( \phi(h)^{-1} \beta(\epsilon_t,(A_{n,t},\1{t\in\cI_n})) ) \\
			&= \cO\Big( \tau_0(j\veps'_n) j^{-1} \Big) + \cO( \phi(h)^{-1} \beta(\epsilon_t,(A_{n,t},\1{t\in\cI_n})) ).
\end{align*}
As $\tau_0(0)=0$ and as $\tau_0$ is continuous in a neighborhood of zero, this last bound is $o(1)$ ($n\rightarrow\infty$) uniformly in $1\le i\le \kappa_n$ and $1\le j\le \kappa'_n$. So for \eqref{E:PConvergenceResiduals13}, it remains to consider the following probabilities (for $q>0$)
\begin{align}\begin{split}\label{E:PConvergenceResiduals14}
		& \p\Big( \max_{\substack{i=1,\ldots,\kappa_n\\ j=1,\ldots,\kappa'_n}}  (n\phi(h))^{-1} \Big| \sum_{t=1}^n \1{h(s_j-\veps'_n)-\veps_n \le \|z_i+\epsilon'_t-x\| \le h(s_j+\veps'_n)+\veps_n } \\
		&\quad - \p( h(s_j-\veps'_n)-\veps_n \le \|z_i+\epsilon'_t-x\| \le h(s_j+\veps'_n)+\veps_n ) \Big| > q, \\
		&\qquad \sup_{x\in \cU(r_n) } \|\Psi(x)-\hat\Psi_{n,b}(x)-\bar\epsilon\| \le \delta \Big).
\end{split}\end{align}
Again, we can proceed as in the derivation of \eqref{E:PConvergenceResiduals9} and obtain a result as in \eqref{E:PConvergenceResiduals12}. I.e., for all $q>0$, \eqref{E:PConvergenceResiduals14} is $\cO( \kappa_n \kappa'_n \exp( - c n\phi(h) )$ for $n$ sufficiently large, $\delta$ sufficiently small and some constant $c\in\R_+$ which depends on the choice of $q$ and $\delta$. Consequently, also the second term in \eqref{E:PConvergenceResiduals11} is of order $\op(1)$. This shows \eqref{E:PConvergenceResiduals3b} is $\op(1)$.

\textit{Part II:} We come to the proof of \eqref{E:PConvergenceResiduals2}. Again, note that
\begin{align*}
		&\sup_{z\in\cS}\, (n\phi(h))^{-1} \Big\| \sum_{t=1}^n \E{K(h^{-1} \|z+\epsilon'_t-x\| ) (h^{-1} (z+\epsilon'_t-x) ) \1{t\in\cI_n} } \\
		&\qquad\qquad\quad - \E{K(h^{-1} \|z+\epsilon'_t-x\| ) (h^{-1} (z+\epsilon'_t-x) ) } \1{t\in\cI_n} \Big\| \\
		&= \cO( \p(t \notin\cI_n) ) + \cO( n^{-1} \sum_{t=1}^n \1{t\notin\cI_n} ) + o(1) = \op(1).
		\end{align*}
Define for each $1\le t\le n$ the Hilbert space valued random variables
\begin{align}\begin{split}\label{E:PConvergenceResiduals15}
			Y_{n,t}(z) &=  \phi(h)^{-1}  \Big( K(h^{-1} \|z+\epsilon'_t-x\| ) (h^{-1} (z+\epsilon'_t-x) ) \\
			&\quad - \E{K(h^{-1} \|z+\epsilon'_t-x\| ) (h^{-1} (z+\epsilon'_t-x) ) } \Big) \1{t\in\cI_n}
\end{split}\end{align}
and set $Y_{n,t,k}(z) = \langle Y_{n,t}(z), e_k\rangle$ for the projections onto $e_k$ ($k\in\N$).
Let $m\in\N$ be equal to $[c' \log n]$ for a sufficiently large $c'\in\R_+$, the exact value can be derived below. Define the set of functions which decay at an exponential rate by
\[
		\cT_m = \Big\{z\in\cS: \sum_{k>m} \langle z,e_k\rangle^2 \le a_0 \exp\Big(-\frac{a_1}{2} m\Big) \Big\},
	\]
where the constants are from Condition~\ref{C:Process}~\ref{C:Process2}. So, using that $\sum_{k>m}\E{|X_{k,t}|^2}\le a_0\exp(-a_1m)$, we find
\begin{align}
			&\p\left( z = \hat{\Psi}_{n,b}(y) \text{ for some } y\in\cS, \: z\notin \cT_m		\right) \nonumber \\
			&= \p\left(	\sup_{ y\in\cS} \; \sum_{k>m} \Big( \sum_{t=1}^n \Delta_{h,t}(y) \Big(\sum_{t=1}^n \Delta_{h,t}(y) \Big)^{-1}  \langle X_{t+1},e_k\rangle \Big)^2 > a_0 \exp\Big(-\frac{a_1}{2} m\Big)\right) \nonumber\\
			&\le \p\left(	 n \sum_{k>m} |X_{k,t+1}|^2 > a_0 \exp\Big(-\frac{a_1}{2} m\Big) \right) \nonumber\\
			&\le a_0^{-1} n \exp\Big(\frac{a_1}{2} m\Big) \E{ \sum_{k>m} |X_{k,t+1}|^2 }		\le  n \exp\Big(-\frac{a_1}{2} m\Big), \label{E:PConvergenceResiduals15b}
\end{align}
which vanishes if $c'> 2/a_1+1$. Consider a similar covering as in the first part such that $\cU(r)\subseteq \cup_{i=1}^{\kappa_n} U(z'_i,\veps_n/2)$. Then for each $z'_i$ either $U(z'_i,\veps_n/2)\cap \cT_m = \emptyset$ or there is a $z_i \in U(z'_i,\veps_n/2)$ which satisfies also the tail condition $\cT_m$. I.e., we obtain an $\veps_n$-covering $\cU(r)\subseteq \cup_{i=1}^{\kappa_n} U(z_i,\veps_n)$ where the $z_i$ are also in $\tau_m$. Then
\begin{align}
		& \sup_{z\in\cU(r) } \1{z=\hat\Psi(y) \text{ for some } y\in\cS} n^{-1} \Big\| \sum_{t=1}^n Y_{n,t}(z)  \Big\|  \nonumber\\
		&\le 2 \max_{1\le i \le \kappa_n} \1{z=\hat\Psi(y) \text{ for some } y\in\cS} n^{-1} \Big\| \sum_{t=1}^n Y_{n,t}(z_i)   \Big\|   \label{E:PConvergenceResiduals16} \\
		&\quad + \max_{1\le i \le \kappa_n} \sup_{z\in U(z_i,\veps_n) } \Big\| \sum_{t=1}^n Y_{n,t}(z) - Y_{n,t}(z_i) \Big\| . \label{E:PConvergenceResiduals17}
\end{align}
We split \eqref{E:PConvergenceResiduals16} in a finite-dimensional term and an infinite-dimensional remainder. The finite-dimensional term can be treated with the same tools we used in the first part of the proof. We apply the decay assumptions to the infinite-dimensional remainder. We omit the indicator $\mathds{1}\{z=\hat\Psi(y) \text{ for some } y\in\cS		\}$ in the finite-dimensional part and obtain
\begin{align}
		&\p\Big(	 \max_{1\le i \le \kappa_n}  n^{-2} \sum_{k=1}^m \Big| \sum_{t=1}^n Y_{n,t,k}(z_i)  \Big|^2 > q^2 	\Big) \nonumber \\
		&\le \p\Big(	 \max_{1\le i \le \kappa_n}  n^{-2} \sum_{k=1}^m \Big| \sum_{t=1}^n Y_{n,t,k}(z_i)  \Big|^2 > q^2 , \sup_{x\in \cU(r_n) } \|\Psi(x)-\hat\Psi_{n,b}(x)-\bar\epsilon\| \le \delta	\Big) \nonumber  \\
		&\quad + \p\Big(\sup_{x\in \cU(r_n) } \|\Psi(x)-\hat\Psi_{n,b}(x)-\bar\epsilon\| > \delta	 \Big) \nonumber \\
		\begin{split}\label{E:PConvergenceResiduals18}
		&\le \sum_{i=1}^{\kappa_n } \sum_{k=1}^m \p\Big(	n^{-1} \Big| \sum_{t=1}^n Y_{n,t,k}(z_i)  \Big| > q m^{-1/2},  \sup_{x\in \cU(r_n) } \|\Psi(x)-\hat\Psi_{n,b}(x)-\bar\epsilon\| \le \delta 	\Big) \\
		&\quad + \p\Big(\sup_{x\in \cU(r_n) } \|\Psi(x)-\hat\Psi_{n,b}(x)-\bar\epsilon\| > \delta \Big), 
		\end{split}
\end{align}
for all $q,\delta>0$. One can derive an exponential inequality for the probability inside the sum in \eqref{E:PConvergenceResiduals18} which relies on the Laplace transform and on the approach from \eqref{E:PConvergenceResiduals6} and \eqref{E:PConvergenceResiduals8}. We only give the details for the conditional mean (i.e., the term that corresponds to \eqref{E:PConvergenceResiduals6}), the higher moments follow then with similar computations. The filtration $(\cF_t)_{t}$ is the same as in the first part. Then for $z=z_i$
\begin{align*}
		\Big| \E{ Y_{n,t,k}(z) | \cF_{t-1} } \Big|
		&= \phi(h)^{-1}	\Big| \mathbb{E}[ K(h^{-1}\|z+\epsilon_t+A_{n,t}-x\|) (h^{-1} \langle z+\epsilon_t+A_{n,t}-x,e_k \rangle) \,|\, \cF_{t-1} ] \\
			&\quad - \E{ K(h^{-1}\|z+\epsilon_t+A_{n,t}-x\|) (h^{-1} \langle z+\epsilon_t+A_{n,t}-x,e_k \rangle) }  	\Big|\1{t\in\cI_n} \\
			&\le L_{n,t} + \phi(h)^{-1}	\bigg|\mathbb{E}[ K(h^{-1}\|z+\epsilon_t+a-x\|)  (h^{-1} \langle z+\epsilon_t+a-x,e_k \rangle)]\Big|_{a=A_{n,t}} \\
			&\quad - \E{ \E{ K(h^{-1}\|z+\epsilon_t+a-x\|) (h^{-1} \langle z+\epsilon_t+a-x,e_k \rangle) } \Big|_{a=A_{n,t} } } \,\bigg| \1{t\in\cI_n} 	\\
			&\quad + \cO\big( \phi(h)^{-1}\beta(\epsilon_t,(A_{n,t},\1{t\in\cI_n})) \big) \Big|.
	\end{align*}
	If we use additionally the functionals $\varphi_y$ and their projections onto the basis vectors $\varphi_{y,k}(u) = \langle \varphi_y(u),e_k\rangle$, we obtain for the right-hand side of the last inequality the upper bound
	\begin{align}
			\begin{split}\label{E:PConvergenceResiduals19}
			& L_{n,t} + \phi(h)^{-1}	\bigg| \mathbb{E}\Big[ K(h^{-1}\|z+\epsilon_t+a-x\|) (h^{-1} \varphi_{x-z-a,k}(\|z+\epsilon_t+a-x\|)) \\
			& \quad - K(h^{-1}\|z+\epsilon_t-x\|) (h^{-1} \varphi_{x-z,k}(\|z+\epsilon_t-x\|)) \Big]\Big|_{a=A_{n,t}} \bigg| \1{t\in\cI_n} \\
			&\quad + \mathbb{E}\bigg[ \phi(h)^{-1} \bigg| \E{K(h^{-1}\|z+\epsilon_t-x\|) (h^{-1} \varphi_{x-z,k}(\|z+\epsilon_t-x\|))} \\
			&\quad - \E{ K(h^{-1}\|z+\epsilon_t+a-x\|) (h^{-1}\varphi_{x-z-a,k}(\|z+\epsilon_t+a-x\|) ) } \Big|_{a=A_{n,t} }  \bigg|\bigg]\\
			&\quad + \cO\big( \phi(h)^{-1}\beta(\epsilon_t,(A_{n,t},\1{t\in\cI_n})) \big).
			\end{split}
\end{align}
Consider \eqref{E:PConvergenceResiduals19} on the set $\{ \sup_{x\in\cU(r_n)} \|\Psi(x)-\hat{\Psi}_{n,b}(x)-\bar\epsilon \| \le \delta \}$ (for $\delta>0$). Using the asymptotic properties of the linear approximation of the functionals $\varphi_y$ near 0, we find that \eqref{E:PConvergenceResiduals19} is
\begin{align*}
		&\cO( L_{n,t}+h^\alpha) +\cO( \delta^\alpha + p_n(\delta) ) + \cO\big(  \phi(h)^{-1}\beta(\epsilon_t,(A_{n,t},\1{t\in\cI_n})) \big),
\end{align*}
uniformly in $z$. These considerations lead to an upper bound of \eqref{E:PConvergenceResiduals18} of the form
\begin{align*}
		\cO\Bigg( \kappa_n m \exp\bigg\{- \frac{1}{2} \frac{q^2 -2 c_1 q (\delta^\alpha + p_n(\delta) + o(1) ) }{c_2 + qm^{-1/2} } \frac{n\phi(h)}{m} \bigg\} \Bigg) + o(1),
\end{align*}
for certain constants $c_1,c_2\in\R_+$. As $\log \kappa_n = \cO( \veps_n^{-1}) = \cO( n\phi(h) (\log n)^{-2})$, we see that this upper bound is $o(1)$ for every choice of the constant $c'$ in the definition of $m$.

Next, we consider the infinite-dimensional remainder in \eqref{E:PConvergenceResiduals16}. Here we need the restriction that $z$ is a function the coefficients of which decay at an exponential rate. Let $J_n = \{ 2^{-1} n \le |\cI_n| \le 2 n\}$. Then $\p(J_n)\rightarrow 1$ as $n\rightarrow\infty$. Moreover due to the result from \eqref{E:PConvergenceResiduals15b}, it is enough to consider
\begin{align*}
			&\p\Big( \Big\{\max_{1\le i \le \kappa_n} \1{z_i = {\hat\Psi}_{n,b}(y) \text{ for some } y\in\cS, z_i \in \cT_m} n^{-2} \sum_{k>m} \Big| \sum_{t=1}^n Y_{n,t,k}(z_i)  \Big|^2 > q^2\Big\} \cap J_n \Big),
\end{align*}
for $q>0$ arbitrary but fixed. An application of Markov's inequality yields that this probability is dominated by
\begin{align*}
			&q^{-2} (n\phi(h))^{-2} \sum_{k>m} \E{ \max_{1\le i \le \kappa_n} \1{z_i\in\cT_m} \1{J_n} \Big|\sum_{t=1}^n K(h^{-1} \|z_i+\epsilon'_t-x\|)(h^{-1} \langle z_i+\epsilon'_t-x, e_k \rangle)	\Big|^2 }  \\
			&= \cO\Big(	h^{-2} \phi(h)^{-2} n^{-1} \sum_{t=1}^n \sum_{k>m} \Big\{ \E{\1{J_n}  \langle \epsilon'_t,e_k \rangle^2}	+ \exp\Big(- \frac{a_1}{2} m	\Big) \Big\}	\Big),
\end{align*}
where we use the decay of $x$ and the fact that $z_i$ is in $\cT_m$. Next, use that decay of $\langle \epsilon'_t,e_k \rangle^2$ is determined by the decay of the coefficients $\langle \epsilon_t,e_k \rangle^2$, $\langle X_t,e_k \rangle^2$ and $\langle \hatPsib(X_{t-1}),e_k \rangle^2$ as well as $\langle \bar\epsilon,e_k \rangle^2$.

By assumption,	$\sum_{k>m} \E{ \langle \epsilon_t,e_k \rangle^2 + \langle X_t,e_k \rangle^2 } = \cO( \exp( -a_1 m)$. Moreover, $\sum_{k>m} \Es{ \langle \hatPsib(X_{t-1}),e_k \rangle^2 } = \cO( n \exp( -a_1 m) )$ uniformly in $1\le t\le n$. Furthermore, one can compute that $\E{\1{J_n} \langle \bar\epsilon,e_k \rangle^2 } = \cO(n \exp( -a_1 m) )$. This yields that also the infinite-dimensional remainder in \eqref{E:PConvergenceResiduals16} is of order $\op(1)$.

Finally, we give some details on \eqref{E:PConvergenceResiduals17}, which essentially can be treated as the corresponding term in the first part of the proof. We consider the difference $Y_{n,t}(z) - Y_{n,t}(z_i)$. We study the random part of this difference, it is at most
\begin{align}
		&\max_{1\le i\le \kappa_n} \sup_{z\in U(z_i,\veps_n) } (n\phi(h))^{-1} \sum_{t=1}^n \bigg\{ K(h^{-1} \|z_i +\epsilon'_t-x\|) \frac{\|z_i-z\|}{h} \nonumber\\
		&\quad + \big|K(h^{-1} \|z_i +\epsilon'_t-x\|) - K(h^{-1} \|z +\epsilon'_t-x\|) \big| \frac{\|z+\epsilon'_t -x\|}{h} \bigg\} \nonumber \\
		\begin{split}\label{E:PConvergenceResiduals20}
		&\le\max_{1\le i\le \kappa_n}  (n\phi(h))^{-1} \sum_{t=1}^n K(h^{-1} \|z_i +\epsilon'_t-x\|) \frac{\veps_n}{h} \\
		&\quad + \max_{1\le i\le \kappa_n} \sup_{z\in U(z_i,\veps_n) } (n\phi(h))^{-1} \sum_{t=1}^n  \big|K(h^{-1} \|z_i +\epsilon'_t-x\|) - K(h^{-1} \|z +\epsilon'_t-x\|) \big| \frac{h+\veps_n}{h}. \end{split}
\end{align}
The expectation of first term in \eqref{E:PConvergenceResiduals20} is $\cO( \veps_n h^{-1}) + o(1)$ uniformly in $i$. Moreover, one can derive as before that
\[
			\max_{1\le i\le \kappa_n}  (n\phi(h))^{-1} \Big| \sum_{t=1}^n K(h^{-1} \|z_i +\epsilon'_t-x\| - \E{K(h^{-1} \|z_i +\epsilon'_t-x\|} \Big|	 = \op(1),\quad n\rightarrow\infty.
\]
Next, we study the second term in \eqref{E:PConvergenceResiduals20}, the same approach, which leads to the derivation of \eqref{UnifConsistencyHatPsiEq9}, yields
\[
			\big|K(h^{-1} \|z_i +\epsilon'_t-x\|) - K(h^{-1} \|z +\epsilon'_t-x\|) \big| = \cO\Big(\veps_n \, h^{-1} + \1{h-\veps_n \le \|z_i+\epsilon'_t-x\|\le h+\veps_n} 		\Big). 
\]
And
\[
	(n\phi(h))^{-1} \sum_{t=1} \p(h-\veps_n \le \|z_i+\epsilon'_t-x\|\le h+\veps_n ) = \cO(\veps_n \, h^{-1} ) + o(1).
	\]
Furthermore,
\[
			\max_{1\le i\le \kappa_n} n\phi(h))^{-1} \Big| \sum_{t=1}^n \1{h-\veps_n \le \|z_i+\epsilon'_t-x\|\le h+\veps_n} - \p(h-\veps_n \le \|z_i+\epsilon'_t-x\|\le h+\veps_n )  \Big| = \op(1).
\]
This shows that \eqref{E:PConvergenceResiduals20} is $\Op( \veps_n h^{-1})$ which is $\op(1)$. This completes the proof of \eqref{E:PConvergenceResiduals2}. 
\end{proof}

\begin{lemma}\label{L:UniformConvergenceInNeighborhood}
Consider the sets $\cV_n$ and Condition~\ref{C:Covering}. Then
\[
		\sup_{y\in \cV_n\cap \cU(r_n) } \norm{ \hatPsib(y)-\Psi(y) - (\hatPsib(x)-\Psi(x)) } \1{ \cV_n \text{ satisfies ($\ast$) at $n$} } = \oas( (n\phi(h))^{-1/2} ).
		\]
\end{lemma}
\begin{proof}
First, note that on the sets $\cU(r_n)$, we use indeed the Nadaraya--Watson estimate for $\hatPsib$ and not the average $\bar{X}_n$. By construction, there is a $\overline{L}$ such that the covering number $\kappa_n$ w.r.t.\ $\|\cdot\|$ of $\cV_n$ satisfies $\log \kappa_n \le \overline{L} b h^{-1} \log n$ and for a radius $\ell_n$ which is $o( b (n\phi(h))^{-1/2} (\log n)^{-1} )$. In particular, the random sets $\cV_n$ satisfy the covering condition in \cite{krebs2018doublefunctional}. Thus, we can proceed as in the proof of Lemma 3.12 in \cite{krebs2018doublefunctional}, the randomness of $\cV_n$ does not change the proof. We only give a sketch of the proof and emphasize the major differences. We apply the decomposition
			\begin{align}
			&\frac{ (n\phi(b))^{-1} \sum_{t=1}^n (X_{t+1}-\Psi(y)) \Delta_{b,t}(y) - (X_{t+1}-\Psi(x)) \Delta_{b,t}(x)} {(n\phi(b))^{-1} \sum_{t=1}^n \E{ \Delta_{b,t}(y) |\cF_{t-1}} } \label{Eq:UniformConvergenceInNeighborhood2a} \\
			&\quad + \frac{(n\phi(b))^{-1} \sum_{t=1}^n (X_{t+1}-\Psi(y)) \Delta_{b,t}(x)}{(n\phi(b))^{-1} \sum_{t=1}^n \E{\Delta_{b,t}(x)  |\cF_{t-1} } } \frac{ (n\phi(b))^{-1} \sum_{t=1}^n \E{ \Delta_{b,t}(x) - \Delta_{b,t}(y) |\cF_{t-1} }}{(n\phi(b))^{-1} \sum_{t=1}^n \E{ \Delta_{b,t}(y) |\cF_{t-1} } }, \label{Eq:UniformConvergenceInNeighborhood2b}
	\end{align}
where we use that on the set $\{ \cV_n \text{ satisfies ($\ast$) at $n$} \}$
\[
			\sup_{y\in\cV_n} \Big| (n\phi(b))^{-1}  \sum_{t=1}^n \Delta_{b,t}(y) - (n\phi(b))^{-1}  \sum_{t=1}^n \E{ \Delta_{b,t}(y) | \cF_{t-1} } \Big| =  \oas( (n\phi(h))^{-1/2} ).
\]
In particular, \eqref{Eq:UniformConvergenceInNeighborhood2b} is $\Oas(b^{1+\alpha})$ which is $\oas((n\phi(h))^{-1/2})$, this follows from the uniform H{\"o}lder continuity of the small ball probability function from \eqref{SmallBallLipschitz}. So, it is sufficient to consider the denominator of \eqref{Eq:UniformConvergenceInNeighborhood2a}, which is at most
\begin{align}
\begin{split}
				&\sup_{y\in\cV_n} \Big\| (n\phi(b))^{-1} \sum_{t=1}^n (X_{t+1}-\Psi(y))\Delta_{b,t}(y)-(X_{t+1}-\Psi(x))\Delta_{b,t}(x) \\
				&\qquad\qquad\qquad - \E{ (X_{t+1}-\Psi(y))\Delta_{b,t}(y)-(X_{t+1}-\Psi(x))\Delta_{b,t}(x) | \cF_{t-1} }		\Big\|  \label{Eq:UniformConvergenceInNeighborhood3} \end{split}\\
				&+ \sup_{y\in\cV_n} \Big\| (n\phi(b))^{-1} \sum_{t=1}^n \E{ (X_{t+1}-\Psi(y))\Delta_{b,t}(y)-(X_{t+1}-\Psi(x))\Delta_{b,t}(x) | \cF_{t-1} }		\Big\| \label{Eq:UniformConvergenceInNeighborhood4}
\end{align}
on the set $\{ \cV_n \text{ satisfies ($\ast$) at $n$} \}$. One can derive an exponential inequality for \eqref{Eq:UniformConvergenceInNeighborhood3} and demonstrate that this term is $\oas( (n\phi(h))^{-1/2} )$, for details see the proof of Lemma 3.12 in \cite{krebs2018doublefunctional}. So we can focus on \eqref{Eq:UniformConvergenceInNeighborhood4} which is
\begin{align}
\begin{split}\label{Eq:UniformConvergenceInNeighborhood5}
			 &\sup_{y\in\cV_n} \Big\| (n\phi(b))^{-1} \sum_{t=1}^n \mathbb{E}\Big[ K(b^{-1}\|z+\epsilon_t-y\|) (b^{-1}\|z+\epsilon_t-y\|)\, \Big\{ \sum_{k\in\N} \intd{\psi_k}(y).\intd{\varphi_{y-z}(0)} \, e_k \Big\} \\
			&\quad -K(b^{-1}\|z+\epsilon_t-x\|) (b^{-1}\|z+\epsilon_t-x\|)\, \Big\{ \sum_{k\in\N} \intd{\psi_k}(x).\intd{\varphi_{x-z}(0)} \, e_k \Big\} \Big]	\Big|_{z\in\Psi(X_{t-1})}	\Big\| \, b + \Oas( b^{1+\alpha})
\end{split}
\end{align}
on the set $\{ \cV_n \text{ satisfies ($\ast$) at $n$} \}$, where we use the uniform continuity properties of the operators $\psi_k$ and $\varphi_{x}$. Using the H{\"o}lder continuity of the small ball probability function, we see that uniformly
\[
		\phi(h)^{-1} \E{ K(b^{-1}\|z+\epsilon_t-y\|) (b^{-1}\|z+\epsilon_t-y\|) - K(b^{-1}\|z+\epsilon_t-x\|) (b^{-1}\|z+\epsilon_t-x\|) } = \cO(h^{\alpha}).
\]
Moreover, $\| \sum_{k\in\N}  (\intd{\psi_k}(y).\intd{\varphi_{y-z}(0)} - \intd{\psi_k}(x).\intd{\varphi_{x-z}(0)}) \, e_k \| = \cO(h^\alpha)$. So, \eqref{Eq:UniformConvergenceInNeighborhood5} is $o((n\phi(h))^{-1/2})$.
\end{proof}

The next lemma is useful to bound the $\infty$-norm of a function by its $L^2$-norm
\begin{lemma}\label{BoundInftyNorm}
Let $\cD\subseteq \R^d$ be a compact and convex domain and let $\nu$ be a $\sigma$-finite Borel-measure on $\cD$ which is absolutely continuous w.r.t.\ the Lebesgue measure such that the infimum of the corresponding density is positive. Let $f$ be a Lipschitz-continuous function with Lipschitz constant $L<\infty$. Then there is a constant $C$ which depends on $\cD$ and the measure $\nu$ such that $\norm{f}_{\infty} \le C L^{d/(2+d)} \left(\int_{\cD} f^2 \intd{\nu}\right)^{1/(2+d)}$.
\end{lemma}
\begin{proof}
First consider the case where $\nu$ equals the Lebesgue measure. The function $f$ which maximizes the $\infty$-norm (under the current restrictions) is the function which describes a geometric cone restricted to certain circle segment which depends on the shape of $\cD$. Since $\cD$ is convex, the angle of the circle segment cannot decrease if the 2-norm of the maximizing function $f$ decreases. Hence, w.l.o.g.\ we can consider the function which is given by a full cone with radius $r$ and slope $L$, this is $f(x)=\1{ B(0,r)} L (r-\norm{x}_2)$, where $B(0,r)$ is the closed ball of radius $r$ around 0 w.r.t.\ the 2-norm. In this case, we obtain the following expression for the $\infty$-norm of $f$
$$
	\norm{f}_{\infty} = C L^{d/(d+2)} \left(\int_{\cD} f(x)^2 \intd{x}\right)^{1/(2+d)},
$$
where the constant depends on $\cD$. In the case for a general measure, we have 
$
		\int_{\cD} f^2 \intd{\nu} \ge \inf_{y\in\cD} \frac{\intd{\nu}}{\intd{x}}(y) \int_{\cD} f(x)^2 \intd{x}
$
and we can deduce the claim from the special case of the Lebesgue measure.
\end{proof}

We state an exponential inequality for a sequence of martingale differences, this inequality can also be found in \cite{de1999decoupling}. Similar but more general results for independent data are given in \cite{yurinskiui1976exponential}. 
	\begin{lemma}\label{ExpIneqMDS}
	Let $(Z_i:i=1,\ldots,n)$ be a martingale difference sequence of real-valued random variables adapted to the filtration $(\cF_i:i=1,\ldots,n)$. Assume that $\E{ |Z_i|^m | \cF_{n,i-1} } \le \frac{m!}{2} (a_i)^2 b^{m-2}$  $a.s.$ for some $b\ge 0$ and $a_i>0$ for $i=1,\ldots,n$. Then
	\[
			\p\Big(\Big|\sum_{i=1}^n Z_i\Big| \ge t \Big) \le 2 \exp\left(	- \frac{1}{2} \frac{t^2}{ \sum_{i=1}^n (a_i)^2 + bt }	\right).
	\]
	In particular, if $b=a^2$ and $a_i=a$ for $i=1,\ldots,n$, then $\p(|\sum_{i=1}^n Z_i| > n t) \le 2 \exp( - 2^{-1} n t^2 (a^2(1+t))^{-1} )$.
	\end{lemma}

The next lemma is an important tool to obtain upper bounds for the expectation of a function which contains two nearly independent random variables.
\begin{lemma}\label{L:BetaMixing}
Let $\pspace$ be a probability space and let $(S,\fS)$ and $(T,\fT)$ be measurable, topological spaces. Let $(X,Y) \colon \Omega \rightarrow S\times T$ be $\fS\otimes\fT$-measurable such that the joint distribution of $X$ and $Y$ is absolutely continuous w.r.t.\ their product measure on $\fS\otimes\fT$ with an essentially bounded Radon-Nikod{\'y}m derivative $g$, i.e.,
\[
		\p_{(X,Y)} \ll \p_X \otimes \p_Y \text{ such that } g \coloneqq \frac{ \intd{ \p_{(X,Y)} } }{\intd{ ( \p_{X} \otimes \p_Y ) } } \text{ satisfies } \norm{g}_{\infty, \p_X\otimes \p_Y } < \infty. 
\]
Let $\cH$ be a separable Hilbert space with orthonormal basis $\{e_k:k\in\N\}$ and induced norm $\|\cdot\|$. Let $F\colon S\times T \rightarrow \cH$ be measurable and square-integrable w.r.t.\ the product measure, i.e., $\norm{F}_{2,\p_X \otimes \p_Y}^2 = \E{ \E{ \| F(X,y)\|^2 }|_{y=Y} } < \infty$. Then
\begin{align}\label{EqBetaMixingIntegrability}
	\left\| \int_{S\times T} F \intd{ \p_{(X,Y)} } - \int_{S\times T} F \intd{ (\p_X \otimes \p_Y ) } \right \| \le 2^{1/2} (1+\norm{g}_{\infty, \p_X \otimes \p_Y} )^{1/2} \norm{F}_{2,\p_X \otimes \p_Y} \beta(X,Y)^{1/2}.
\end{align}
\end{lemma}
\begin{proof}
Write $F_k$ for the coordinate functions $\langle F,e_k \rangle$. The square of the left-hand side of \eqref{EqBetaMixingIntegrability} equals
\[
			\sum_{k\in\N} \Big( \int_{S\times T} f_k(x,y) \, (g(x,y)-1) \, \p_X \otimes \p_Y (\intd{ (x,y) }) \Big)^2.
\]
Applying H{\"o}lder's inequality, each summand is at most
\[
		\int_{S\times T} f_k(x,y)^2 \, \p_X \otimes \p_Y (\intd{ (x,y) })  \times 	\frac{1}{2}\int_{S\times T}  |g(x,y)-1| \, \p_X \otimes \p_Y (\intd{ (x,y) } ) \times 2(  \norm{g}_{\infty, \p_X\otimes \p_Y }  +1).
\]
This yields the result because $\beta(X,Y) = \frac{1}{2}\int_{S\times T}  |g(x,y)-1| \, \p_X \otimes \p_Y (\intd{ (x,y) }) $.
\end{proof}


\end{document}